\documentclass[a4paper,10pt,reqno,noindent]{amsart}
\usepackage[centertags]{amsmath}
\usepackage{mathtools}
\usepackage{amsfonts,amssymb,amsthm,dsfont,cases,amscd,esint,enumerate}
\usepackage[T1]{fontenc}
\usepackage[english]{babel}
\usepackage[applemac]{inputenc}
\usepackage{newlfont,accents}
\usepackage{color}
\usepackage[normalem]{ulem}
\usepackage[body={13cm,21cm},centering]{geometry} 
\usepackage{fancyhdr}
\usepackage{enumerate}

\theoremstyle{plain}
\newtheorem{theor10}{Theorem}
\newenvironment{theor1}
  {\pushQED{\qed}\begin{theor10}}
  {\popQED\end{theor10}}
\newtheorem{cor10}{Corollary}
\newenvironment{cor1}
  {\pushQED{\qed}\begin{cor10}}
  {\popQED\end{cor10}}
\newtheorem{theor0}{Theorem}[section]
\newenvironment{theor}
  {\pushQED{\qed}\begin{theor0}}
  {\popQED\end{theor0}}
\newtheorem{lem0}[theor0]{Lemma}
\newenvironment{lem}
  {\pushQED{\qed}\begin{lem0}}
  {\popQED\end{lem0}}
\newtheorem{prop0}[theor0]{Proposition}
\newenvironment{prop}
  {\pushQED{\qed}\begin{prop0}}
  {\popQED\end{prop0}}
\newtheorem{cor0}[theor0]{Corollary}

\theoremstyle{definition}
\newtheorem{defin0}[theor0]{Definition}
\newenvironment{defin}
  {\pushQED{\qed}\begin{defin0}}
  {\popQED\end{defin0}}
\newtheorem{rems0}[theor0]{Remarks}
\newenvironment{rems}
  {\pushQED{\qed}\begin{rems0}}
  {\popQED\end{rems0}}
\newtheorem{rem0}[theor0]{Remark}
\newenvironment{rem}
  {\pushQED{\qed}\begin{rem0}}
  {\popQED\end{rem0}}

\numberwithin{equation}{section}

\newcommand{\N}{\mathbb N}
\newcommand{\e}{\varepsilon}
\newcommand{\Lc}{\mathcal{L}}
\newcommand{\Vc}{\mathcal V}
\newcommand{\R}{\mathbb R}
\newcommand{\E}{\mathbb E}
\newcommand{\C}{\mathbb C}
\newcommand{\F}{\mathcal F}
\newcommand{\loc}{{\operatorname{loc}}}
\newcommand{\Id}{\operatorname{Id}}
\newcommand{\per}{{\operatorname{per}}}
\newcommand{\ee}{e}
\newcommand{\Aa}{\boldsymbol a}
\newcommand{\bb}{{\boldsymbol b}}

\newcommand{\cc}{{\boldsymbol c}}
\newcommand{\Ld}{\operatorname{L}}
\newcommand{\supp}{\operatorname{supp}}

\newcommand{\step}[1]{\noindent \textit{Step} #1.}

\newcommand{\expec}[1]{\mathbb{E}\left[ #1 \right]}
\newcommand{\expecm}[1]{\mathbb{E}\big[ #1 \big]}

\newcommand{\no}[1]{{\sout{#1}}}

\usepackage{color}
\usepackage[colorlinks,citecolor=black,urlcolor=black]{hyperref}

\title[Long-time homogenization of the wave equation]{A spectral ansatz for the long-time homogenization of the wave equation}
\author[M. Duerinckx]{Mitia Duerinckx}
\author[A. Gloria]{Antoine Gloria}
\author[M. Ruf]{Matthias Ruf}
\address[Mitia Duerinckx]{Universit\'e Libre de Bruxelles, D\'epartement de Math\'ematique, 1050 Brussels, Belgium}
\email{mitia.duerinckx@ulb.be}
\address[Antoine Gloria]{Sorbonne Universit\'e, CNRS, Universit\'e de Paris, Laboratoire Jacques-Louis Lions, 75005~Paris, France \& Institut Universitaire de France \& Universit\'e Libre de Bruxelles, D\'epartement de Math\'ematique, 1050~Brussels, Belgium}
\email{antoine.gloria@sorbonne-universite.fr}
\address[Matthias Ruf]{Ecole Polytechnique Fédérale de Lausanne, {Section de math\'ematiques, 1015~Lausanne}, Switzerland}
\email{matthias.ruf@epfl.ch}
\begin{document}
\selectlanguage{english}

\begin{abstract}
Consider the wave equation with heterogeneous coefficients in the homogenization regime.
At large times, the wave interacts in a nontrivial way with the heterogeneities, giving rise to effective dispersive effects. The main achievement of the present work is  a new ansatz for the long-time two-scale expansion inspired by spectral analysis. Based on this spectral ansatz, we extend and refine all previous results in the field, proving homogenization up to optimal timescales with optimal error estimates, and covering all the standard  assumptions on heterogeneities (both periodic and stationary random settings).
\end{abstract}

\maketitle
\setcounter{tocdepth}{1}
\tableofcontents

\section{Introduction}

\subsection{General overview}
Let $d\ge1$ be the space dimension and let $\Aa$ be a symmetric coefficient field on $\R^d$ that satisfies the boundedness and ellipticity properties
\begin{align}\label{eq:unif-ell}
|\Aa(x)\xi|\le|\xi|,\qquad\xi\cdot\Aa(x)\xi\ge\lambda|\xi|^2,\qquad\text{for all $x,\xi\in\R^d$},
\end{align}
for some $\lambda>0$.
We shall consider both the case when $\Aa$ is periodic and the case when $\Aa$ is a stationary ergodic random field (in the latter case, we restrict to a Gaussian model for illustration, cf.~Definition~\ref{def:gauss}).
Given an impulse 
$f\in C^\infty_c((0,\infty);\Ld^2(\R^d))$,
we consider the ancient solution of the associated linear wave equation
\begin{equation}\label{eq:hyperb}
\left\{\begin{array}{ll}
(\partial_t^2-\nabla\cdot\Aa(\tfrac\cdot\e)\nabla)u_\e=f,&\text{in $\R\times\R^d$},\\
u_\e=f=0,&\text{for $t<0$,}
\end{array}\right.
\end{equation}
in the homogenization regime $0<\e \ll 1$, and we are interested in the accurate description of the long-time behavior of the flow. The reason why we focus on ancient solutions (with~$u_\e=f=0$ for $t<0$) is to ensure the well-preparedness of the wave and to avoid propagating time oscillations; see {the} discussion in Section~\ref{sec:ill-posed}.
{Note that we can also consider the case of strongly elliptic systems up to obvious modifications. It is however crucial that coefficients be symmetric (to avoid exponentially growing modes in the Floquet-Bloch theory).}

On short timescales~$t=O(1)$,   standard theory~\cite{Brahim-Otsmane-92} ensures that the flow can be approximated to leading order by the ancient solution of a homogenized wave equation,
\begin{equation}\label{eq:standard-hom}
\left\{\begin{array}{ll}
(\partial_t^2-\nabla\cdot\bar\Aa\nabla)\bar u=f,&\text{in $\R\times\R^d$},\\
\bar u=f=0,&\text{for $t<0$,}
\end{array}\right.
\end{equation}
where the (constant) effective coefficient~$\bar\Aa$ is the same as for the homogenization of the corresponding steady-state problem. This means that homogenization and time evolution decouple to leading order on short timescales.
As first understood by Santosa and Symes~\cite{SaSy-91}, {this is however no longer the case on longer timescales}: more precisely, a non-trivial interaction between homogenization and time evolution appears as soon as $t\ge O(\e^{-2})$ in the periodic setting,
leading to a dispersive correction to the naïve{ly} homogenized wave equation~\eqref{eq:standard-hom}.

In the periodic setting, the first rigorous analysis of this phenomenon is  based on spectral theory, more specifically on Floquet--Bloch theory, and is due to Lamacz~\cite{Lamacz-11} in one space dimension, and to Dohmal, Lamacz, and Schweizer~\cite{DLS-14,DLS-15}
in higher dimension.
They proved the convergence to some suitable dispersive homogenized wave equation up to times~\mbox{$t\ll \e^{-3}$}.
Due to the use of the Floquet--Bloch theory, it was not clear that this approach could be applied beyond the periodic setting to other standard frameworks for homogenization (such as  quasi-periodic or random coefficient fields). To treat such cases, Benoit and the second author developed in~\cite{BG} an approximate version of the Floquet--Bloch theory, which was inspired by~\cite{ABV-16} and by the observation that the derivative of the Bloch wave with respect to the wave number at 0 is a multiple of the standard corrector in homogenization. Extending this to all orders, {\cite{BG}} introduced a notion of `Taylor--Bloch waves', which approximately diagonalize the elliptic operator~$-\nabla\cdot\Aa\nabla$ at low wavenumber. In contrast to the standard Floquet--Bloch analysis, this approximate spectral approach is easily {transferred} to the random setting as it does not rely on the existence of exact Bloch waves (see \cite{DGS} for an extension of these ideas to other regimes).
In the periodic case, this allowed {the authors of \cite{BG}} to derive a whole hierarchy of higher-order homogenized equations that are valid to leading order up to times~$t\le O(\e^{-\ell})$ for any $\ell\ge 0$. 
These higher-order homogenized equations are  well-posed up to truncating high-frequencies.
In the random case, they also managed to cover the case of random coefficient fields, for which a homogenized description can only be found up to some maximal timescale $t=O(\e^{-\ell_*})$.
Although this analysis allows  reaching  long timescales, it does not provide approximations with optimal accuracy. There is indeed a strong limitation in the analysis: since the impulse $f$ in equation~\eqref{eq:hyperb} is not adapted to $O(\e)$ oscillations of the coefficients, Bloch waves at low wave number only describe the solution to leading order, and Bloch waves at higher wave number should further be taken into account for a finer description. The main difficulty is that Bloch waves at higher wave number are not easily related to ``correctors'' in homogenization, so that it was unclear how to improve the accuracy in~\cite{BG}.

Shortly after \cite{BG},  following a variant of classical two-scale expansion methods~\cite{BLP-78}, Allaire, Lamacz, and Rauch~\cite{ALR} and Abdulle and Pouchon~\cite{Pouch-17,Pouch-19}
obtained similar results in the periodic setting and  did improve the accuracy in the description of the wave flow on long timescales in terms of some two-scale expansion.
Interestingly, and as opposed to the equations obtained by approximate spectral theory, the homogenized equations obtained in \cite{ALR,Pouch-17,Pouch-19} have to be significantly reformulated several times before they give rise to a well-posed problem (this is called the ``criminal approximation'' in \cite{ALR}). Because the number of such reformulations increases with the order of accuracy, and although arbitrarily long times and optimal accuracy can be reached, 
the ``infinite-order'' homogenized operator they implicitly define cannot be inverted even for very smooth functions.

On the one hand, approaches inspired by spectral theory are physically-motivated (for waves equations, the spectrum is of the essence),  yield well-posed equations valid for long times, but so far were limited in terms of accuracy. On the other hand, approaches based on systematic two-scale expansions allow { reaching} both long times and optimal accuracy, but essentially require as many reformulations as the order of accuracy to obtain a well-posed equation, which prevents one from inverting the associated ``infinite-order'' homogenized operator (in other words, the bound on the homogenization error is not sharp).
To sum up, a physically-motivated two-scale expansion to reach long times and optimal accuracy was still missing.

The main aim of this paper is to introduce a full spectral two-scale expansion for \eqref{eq:hyperb}, that extends \cite{BG} to any order and allows {{us} to invert the ``infinite-order'' homogenized operator for smooth enough impulse{{s}~$f$.
This spectral two-scale expansion is defined in Theorem~\ref{th:main-per}, whereas the infinite-order result is given in Corollary~\ref{cor:summable}. This infinite-order result shows that the analysis we do here is indeed paying off (it cannot be obtained using the systematic two-scale expansions of  \cite{ALR,Pouch-17,Pouch-19}).
The main insight of this work is encapsulated in Proposition~\ref{prop:Bloch}, which reformulates the 
explicit formula for the solution of \eqref{eq:hyperb} based on Floquet--Bloch theory in a way that leverages an intrinsic two-scale expansion (which we call \emph{spectral} two-scale expansion). 
A fundamental physical feature of this spectral two-scale expansion is the following: corrections due to the fact that the impulse~$f$ is not adapted to $O(\e)$ oscillations of the coefficients
are local with respect to $f$,
see Remark~\ref{rem:vanish}. This original feature should be of interest to the engineering community, as it means in particular that the expansion actually reduces (essentially) to the much simpler one used in~\cite{BG} outside the support of the impulse.
The main merit of this work is to work out this spectral two-scale expansion and its combinatorial structure.
The adaptation from the periodic to the random setting is essentially routine to the expert in stochastic homogenization.
It is however important and shows the limitation of homogenization techniques for waves in random media --- which is why a detailed statement is included in Theorem~\ref{th:main2}.

Last,
we also thoroughly discuss the two-scale approach of \cite{ALR,Pouch-17,Pouch-19} (which we call \emph{geometric} two-scale expansion, cf.~Section~\ref{sec:geom-insight}), and we relate it to the new spectral two-scale expansion. This essentially amounts to comparing
redundant hierarchies of corrector equations, which we do in an algorithmic way. As an output, we
improve the error analysis of~\cite{ALR,Pouch-17,Pouch-19}, cf.~\eqref{eq:2scale-concl-perOpt}.

The rest of this introduction is organized as follows: In Section~\ref{sec:main-res}, we state our main results on long-time homogenization, both in the periodic and in the random settings, comparing the results obtained with the spectral and the geometric approaches.
Next, in Section~\ref{sec:discuss}, we comment on the important question of the well-posedness of the formal homogenized equations (cf.~\eqref{eq:goal2} and~\eqref{eq:form-hom} below). In Sections~\ref{sec:Bloch} and~\ref{sec:geom-insight}, we motivate the special form of the spectral and the  {geometric} two-scale expansions. We conclude in Section~\ref{sec:ill-posed} with a discussion of the well-preparedness assumption for~\eqref{eq:hyperb}, going beyond the framework of ancient solutions.

\newpage
\subsection*{Notation}
\begin{enumerate}[$\bullet$]
\item We write $\nabla=(\nabla_j)_{1\le j\le d}$ for the gradient with respect to the space variable, $\partial_t$ for the time derivative, and $(D_j)_{0\le j\le d}$ for the space-time gradient with $D_0=\partial_t$ and $D_j=\nabla_j$ for $1\le j\le d$. {Given $n\in\N$, we denote by $\nabla^n=(\nabla^n_{i_1\ldots i_n})_{1\leq i_1,\ldots,i_n\leq d}$ the $n$th-order spatial derivative.}
\item For a vector field $F$ and a matrix field $G$, we set $(\nabla F)_{jl}=\nabla_lF_j$ and  $(\nabla\cdot G)_j=\nabla_l G_{jl}$ (we systematically use Einstein's summation convention for repeated indices). {We also denote by $(\nabla F)^T$ the pointwise transpose{d} field, $(\nabla F)^T_{jl}=(\nabla F)_{lj}$.}
\item {Given to matrices $A,B\in\R^{m\times n}$, we denote by $A:B$ their inner product defined by $A:B=A_{ij}B_{ij}$  (again using Einstein's summation convention).}
For $\xi\in\R^d$ and an $n$th-order tensor $T=(T_{j_1\ldots j_n})_{1\le j_1,\ldots,j_n\le d}$, we use the notation \mbox{$T\odot \xi^{\otimes n}=T_{j_1\ldots j_n}\xi_{j_1}\ldots\xi_{j_n}$} for the contraction. For two symmetric tensors $T=(T_{j_1\ldots j_n})_{ j_1,\ldots,j_n}$ and $S=(S_{j_1\ldots j_m})_{j_1,\ldots,j_m}$, we define their symmetric tensor product $T\otimes_sS$ as the $(n+m)$th-order symmetric tensor characterized by $(T\otimes_sS)\odot \xi^{\otimes(n+m)}=(T\odot\xi^{\otimes n})(S\odot\xi^{\otimes m})$ for all $\xi\in\R^d$.
\item  The spatial Fourier transform of a function $f$ defined on $\R^d$ is denoted by $\hat f(\xi)=\int_{\R^d}e^{-i\xi\cdot x}u(x)\,dx$, and the inverse Fourier transform by $f(x)=\int_{\R^d}e^{i\xi\cdot x}\hat f(\xi)\,d^*\xi$ with $d^*\xi=(2\pi)^{-d}d\xi$.
\item We set $\langle s \rangle:= (1+|s|^2)^{1/2}$, and we similarly {define the pseudo-differential operator $\langle\nabla\rangle$ with Fourier symbol $(1+|\xi|^2)^\frac12$}. {More generally, given a continuous map $\chi:\R^d\to\R$, we define the pseudo-differential operator $\chi(\nabla)$ with Fourier-symbol $\chi(\xi)$. Moreover, given an $n$th-order tensor $T=(T_{j_1\ldots j_n})_{1\leq j_1,\ldots,j_n\leq d}$ we define the differential operator $T\odot\nabla^n=T_{j_1\ldots j_n}\nabla^n_{j_1\ldots j_n}$.}
\item $\N$ stands for the set of nonnegative integers. For a multi-index $n=(n_1,\ldots,n_k)\in\N^k$, we let $|n|=n_1+\ldots+n_k$.
\item $\expec{\cdot}$ stands for expectation in the random setting, and is also used in the periodic setting as a short-hand notation for averaging on the unit cell $Q=(-\frac12,\frac12)^d$, $\expec{X}=\fint_Q X$.
\item We denote by $C\ge1$ any constant that only depends on the dimension $d$ and on the ellipticity constant $\lambda$ in~\eqref{eq:unif-ell}. We use the notation $\lesssim$ (resp.~$\gtrsim$) for $\le C\times$ (resp.~$\ge\frac1C\times$) up to such a multiplicative constant $C$. We write $\ll$ (resp.~$\gg$) for $\le \frac1C\times$ (resp.~$\ge C\times$) up to a sufficiently large multiplicative constant $C$. We add subscripts to indicate dependence on other parameters. {We use Landau's big-$O$ notation in a less rigorous way to indicate the scaling behavior of quantities, where the precise bounds can depend on many parameters.}
\item The ball centered at $x$ of radius $r$ in $\R^d$ is denoted by $B_r(x)$, and we set $B(x)=B_1(x)$, $B_r=B_r(0)$, and $B=B_1(0)$.
\item When defining hierarchies of correctors and of homogenized coefficients, we take the convention that all quantities that are not defined are implicitly set to zero: e.g.\@ $\psi^n=0$ for $n<0$ and $\bar\bb^n=0$ for $n\le0$ in Definition~\ref{def:spec-corr-psi}, etc.
\end{enumerate}

\subsection{Main results: long-time homogenization}\label{sec:main-res}
Our main results yield long-time error estimates for the two-scale expansion of the heterogeneous wave 
equation~\eqref{eq:hyperb} with optimal accuracy up to the optimal maximal timescale, { with optimal norms (see below the statements for a precise discussion of optimality)}. We separately consider the periodic and the random setting\no{s}; the case of quasi-periodic coefficient fields could be treated as well but is skipped for shortness.

\subsubsection{Periodic setting}
We start with the case when the coefficient field~$\Aa$ is periodic on the unit cell $Q=(-\frac12,\frac12)^d$. {The main result of this contribution provides} a two-scale expansion with optimal error estimate up to times $t\le O(\e^{-\ell})$ for any~$\ell\ge0$.
This is obtained by extending the spectral approach of~\cite{BG} to higher-order accuracy in form of a suitable two-scale expansion: while the error estimate in~\cite{BG} saturated at $O(\e)$, we now reach accuracy $O(\e^\ell t)$, cf.~\eqref{eq:2scale-concl-per}. 
The formal homogenized equation~\eqref{eq:goal2} takes the form of a dispersive correction of~\eqref{eq:standard-hom} and the discussion of its well-posedness is postponed to Section~\ref{sec:discuss}; note that the homogenized differential operator {in \eqref{eq:goal2} below} is necessarily symmetric in the sense that $\bar\bb^{k}=0$ for all~$k$ even.\footnote{This is a consequence of the fact this operator appears as the homogenization of the self-adjoint wave operator; see~\cite[Proposition~1]{BG} or our direct proof of Proposition~\ref{prop:even-coeff} below.}
The proof of this main result is displayed in Section~\ref{sec:spectral}, together with the definition of spectral correctors. 

\begin{theor1}\label{th:main-per}
Let $\Aa$ be $Q$-periodic.
{There exist sequences of spectral correctors~$\{\psi^{n}\}_{n}$ and~$\{\zeta^{n,m}\}_{n,m}$
obtained as solutions of elliptic problems on the periodic cell $Q$,   a sequence of homogenized tensors $\{\bar\bb^{n}\}_{n}$, and a sequence of Fourier multiplier $\{\gamma_{n}\}_n$ with $|\gamma_{n}(\xi)| \le 1$, cf.~Definition~\ref{def:spec-corr-psi}, such that the following holds.}
{For all $\ell \ge 1$ and} for any impulse $f\in C^\infty(\R;H^\infty(\R^d))$ with $f=0$ for $t<0$, the ancient solution $u_\e$ of the heterogeneous wave equation~\eqref{eq:hyperb}
is accurately described by the  `spectral two-scale expansion'
\begin{multline}\label{eq:goal1}
S_\e^\ell[\bar u_\e^\ell,f]\,:=\,\sum_{n=0}^\ell \e^n  \psi^n(\tfrac\cdot\e)\odot  \gamma_{\ell}(\e\nabla)\nabla^n\bar u_\e^\ell
\\
+\e^3 \sum_{2m=0}^{\ell-3}(-1)^m\e^{2m} \sum_{n=0}^{\ell-3-2m}\e^n\zeta^{n,m}(\tfrac\cdot\e)\odot  \gamma_{\ell} (\e \nabla)    \nabla^{n+1}\partial_t^{2m} f,
\end{multline}
where $\bar u_\e^\ell$ is an ancient solution of the following formal homogenized equation on \mbox{$\R\times\R^d$}, in one of the meanings provided in Lemma~\ref{lem:apriori-baru-sp},
\begin{equation}\label{eq:goal2}
\partial_t^2\bar u_\e^\ell-\nabla\cdot\Big(\bar\Aa+\sum_{k=2}^\ell\bar\bb^k\odot(\e\nabla)^{k-1}\Big)\nabla\bar u_\e^\ell=f+O(\e^\ell).
\end{equation}
(Here, $\psi^n$ is an $n$th-order tensor field, $\zeta^{n,m}$ is an $(n+1)$th-order tensor field, and $\bar\bb^k$ is a matrix-valued $(k-1)$th-order tensor --- see the notation section for the contraction $\odot$ of tensors of the same order.)
More precisely, we have the following error estimate: for all $\ell\ge1$ and $t\ge0$,
\begin{align}\label{eq:2scale-concl-per}
& \|u_\e^t-S^{\ell}_\e[\bar u_\e^{\ell;t},f^t]\|_{\Ld^2(\R^d)}+\|D(u_\e^t-S^{\ell}_\e[\bar u_\e^{\ell;t},f^t])\|_{\Ld^2(\R^d)}\\
&\hspace{7cm} \,\le\,(\e C)^\ell \langle t\rangle\,    \|\langle D\rangle^{C\ell}   f\|_{\Ld^1((0,t);\Ld^2(\R^d))}.\nonumber\qedhere
\end{align}
\end{theor1}

\begin{rem}\label{rem:vanish}
The spectral two-scale expansion~\eqref{eq:goal1} has an important property of physical interest:  the second sum contains a series of terms that are all local with respect to the impulse~$f$, and this local contribution vanishes outside the support of~$f$.
This specific form for the expansion owes to the Bloch wave analysis of Section~\ref{sec:Bloch}. 
 It also illustrates the superiority of the present analysis over \cite{BG}: although outside the support of the impulse the expansion in~\cite{BG} has essentially the same form as~\eqref{eq:goal1} (up to the pseudo-differential operator $\gamma_\ell(\e\nabla)$), it does not reach accuracy beyond {the order}~$O(\e)$.
\end{rem}
\begin{rem}
The above two-scale expansion~\eqref{eq:goal1} involves the pseudo-differential operator $\gamma_\ell(\e \nabla)$. Although convenient in the analysis, it might not be so in practice (e.g.\@ for numerical purposes). 
As shown in~\eqref{e.expansion-gamma}, we can expand $\gamma_\ell(\e \nabla)=1+\sum_{k=2}^\infty \e^k \gamma_{\ell}^k \odot \nabla^{k}$ for some explicit tensors~$\{\gamma_{\ell}^k\}_k$, so that for the purpose of~\eqref{eq:2scale-concl-per} it can be approximated to the required accuracy $O(\e^\ell)$ by a finite-order differential operator. Note that the pseudo-differential operator has some intrinsic spectral interpretation, playing the role of a normalization of Bloch waves, cf.~Section~\ref{sec:Bloch}.
\end{rem}

{The scaling $\e^\ell \langle t\rangle$ of the error  \eqref{eq:2scale-concl-per} is optimal.
This can be seen in this periodic setting on the explicit spectral formula \eqref{eq:full-decomp-per-u} for the solution (with a first order Taylor expansion of the time integrand) for a forcing term $f$ compactly supported in time. The scaling of the error with respect to the norm of $f$ involves the optimal order of derivatives wrt to $\ell$ (we indeed need at least $\ell$ derivatives to define $\bar u^\ell$ -- we have not tried to optimize the multiplicative constant $C$). In particular it implies the summability of the two-scale expansion (or the invertibility of the associated ``infinite-order homogenized operator'' as pointed out in the introduction) in form of the following corollary.}

\begin{cor1}\label{cor:summable}
Let $\Aa$ be $Q$-periodic. Given an impulse $f\in C^\infty(\R;H^\infty(\R^d))$ that decays as~$t\downarrow-\infty$ in the sense of $\int_{-\infty}^0|t|\|\langle D\rangle^kf^t\|_{\Ld^2(\R^d)}\,dt<\infty$ for any $k\ge0$,
consider the unique solution of the associated heterogeneous wave equation
\begin{equation*}
\left\{\begin{array}{ll}
(\partial_t^2-\nabla\cdot\Aa(\tfrac\cdot\e)\nabla)u_\e=f,&\text{in $\R\times\R^d$},\\
{\lim_{t\downarrow -\infty}u_\e^t \equiv 0},&\text{in $\R^d$}.
\end{array}\right.
\end{equation*}
Then, in terms of the two-scale expansion~\eqref{eq:goal1}, with $\bar u_\e^\ell$ now denoting the corresponding solution of the formal homogenized equation~\eqref{eq:goal2} in the sense of Lemma~\ref{lem:apriori-baru-sp}, we have for all $t\in \R$,
\begin{equation}\label{e.bound-summable-1}
\|D(u_\e^t-S^{\ell}_\e[\bar u_\e^{\ell;t},f^t])\|_{\Ld^2(\R^d)} \,
\le \,
(\e C)^\ell \int_{-\infty}^t(t-s)\|\langle D\rangle^{C\ell}   f^s\|_{\Ld^2(\R^d)}\,ds.
\end{equation}
In particular, if for instance the impulse takes the form $f^t(x)=f_1(t)f_2(x)$, where $f_1$ has a smooth and compactly supported Fourier transform on~$\R$ and where $f_2$ has a compactly supported Fourier transform on~$\R^d$,
then the two-scale expansion is summable in the following sense: for $0<\e\ll_f 1$ small enough (only depending on~$d,\lambda,f$), for all $T\in\R$,
\begin{equation}\label{e.bound-summable-2}
\lim_{\ell\uparrow\infty}\,\sup_{t\le T}\,\|D(u_\e^t-S^{\ell}_\e[\bar u_\e^{\ell;t},f^t])\|_{\Ld^2(\R^d)} =0. \qedhere
\end{equation}
\end{cor1}

For comparison,
we display the corresponding result that can be obtained instead of Theorem~\ref{th:main-per} when using a more standard ``geometric'' approach to devise a two-scale expansion as in~\cite{ALR,Pouch-17,Pouch-19} (see Section~\ref{sec:geom-insight} for an explanation of the naming ``geometric'').
We slightly improve the error estimates of~\cite{ALR} thanks to {the} use of suitable flux correctors (see Remark~\ref{rem:flux-corr}). Yet, the scaling with respect to $\ell$ in the error estimate~\eqref{eq:2scale-concl-per-H} is much worse than the one in Theorem~\ref{th:main-per} by a factor $\ell^\ell$, thus showing the advantage of the spectral approach.
(In particular, Corollary~\ref{cor:summable} does not follow from Theorem~\ref{th:main-per2}.)
The proof is displayed in Section~\ref{sec:structure}, together with the definition of hyperbolic correctors.

\begin{theor1}\label{th:main-per2}
Let $\Aa$ be $Q$-periodic.
There exists a sequence of hyperbolic correctors $\{\phi^{n,m}\}_{n,m}$ obtained as solutions of elliptic problems on the periodic cell $Q$, and there exists a sequence of homogenized coefficients $\{\bar\Aa^{n,m}\}_{n,m}$, cf.~Definition~\ref{defi:BC}, such that the following holds.
For all $\ell\ge 1$ and for any impulse $f\in C^\infty(\R;H^\infty(\R^d))$ with $f=0$ for $t<0$, the ancient solution $u_\e$ of the heterogeneous wave equation~\eqref{eq:hyperb} is accurately described by the `hyperbolic two-scale expansion'
\begin{equation}\label{e.2s-hyp-discuss}
H^\ell_\e[\bar v_\e^\ell]\,:=\,\sum_{n=0}^\ell\sum_{m=0}^{\ell-n}\e^{n+m} \phi^{n,m}(\tfrac\cdot\e)\odot\nabla^n\partial_t^m\bar v_\e^\ell,
\end{equation}
where $\bar v_\e^\ell$ is an ancient solution of the following formal homogenized equation on $\R\times\R^d$, up to a suitable revamping, cf.~\eqref{eq:homog-lim/ref}, in one of the meanings provided in Lemma~\ref{lem:apriori-baru-sp},
\begin{equation}\label{eq:form-hom}
\partial_t^2\bar v^\ell_\e-\nabla\cdot\Big(\sum_{n=1}^\ell\sum_{m=0}^{\ell-n}\bar\Aa^{n,m}\odot(\e\nabla)^{n-1}(\e\partial_t)^m\Big)\nabla \bar v^\ell_\e \,=\, f+O(\e^\ell).
\end{equation}
(Here, $\phi^{n,m}$ is an $n$th-order tensor and $\bar\Aa^{n,m}$ is a matrix-valued $(n-1)$th-order tensor.)
More precisely, we have the following error estimate: for all $\ell \ge1$ and $t\ge 0$,
\begin{align}\label{eq:2scale-concl-per-H}
& \|u_\e^t-H^\ell_\e[\bar v_\e^{\ell;t}]\|_{\Ld^2(\R^d)}+\|D(u_\e^t-H^\ell_\e[\bar v_\e^{\ell;t}])\|_{\Ld^2(\R^d)}\\
&\hspace{5.5cm}\,\le\,(\e C\ell)^\ell \langle t\rangle\, \|\langle D\rangle^{C\ell}\langle\e CD\rangle^{\ell^2} f\|_{\Ld^1((0,t);\Ld^2(\R^d))}.\nonumber\qedhere
\end{align}
\end{theor1}

It is instructive to compare spectral and {geometric} two-scale expansions~\eqref{eq:goal1} and~\eqref{e.2s-hyp-discuss}.
Outside the support of~$f$, \eqref{eq:goal1} provides an approximation of $u_\e$ to order $O(\e^\ell)$ by a sum of $\ell+1$ terms, whereas \eqref{e.2s-hyp-discuss} reaches a similar order of approximation with a sum of~$\frac14(\ell+1)(\ell+4)$ terms (note that $\phi^{n,m}$ vanishes for $m$ odd, cf.~Definition~\ref{defi:BC}). This difference between~$O(\ell)$ and~$O(\ell^2)$ terms in the expansions illustrates the more intrinsic character of the spectral two-scale expansion and its superiority in terms of estimates.
There is obviously a link between spectral and hyperbolic correctors, and spectral correctors~$\{\psi^n,\zeta^{n,m}\}_{n,m}$ can {indeed} be recovered as linear combinations of hyperbolic correctors $\{\phi^{n,m}\}_{n,m}$ with coefficients that are \emph{nonlinear functions} of hyperbolic homogenized coefficients $\{\bar\Aa^{n,m}\}_{n,m}$. Working out the precise algorithmic relation between the spectral and geometric approaches is  quite involved and necessarily algorithmic. This is the subject of Section~\ref{sec:fromGeom2Spec}.
In particular, in combination with Theorem~\ref{th:main-per}, one can improve~\eqref{eq:2scale-concl-per-H} a posteriori to 
\begin{equation} \label{eq:2scale-concl-perOpt}
 \|u_\e^t-H^\ell_\e[\bar v_\e^{\ell;t}]\|_{\Ld^2(\R^d)}+\|D(u_\e^t-H^\ell_\e[\bar v_\e^{\ell;t}])\|_{\Ld^2(\R^d)} \,
\le \,
(\e C)^\ell \langle t\rangle\,    \|\langle D\rangle^{C\ell}   f\|_{\Ld^1((0,t);\Ld^2(\R^d))},
\end{equation}
thus removing the spurious factor $\ell^\ell$.
This constitutes a significant strengthening of the error analysis of~\cite{ALR} and could not be obtained from the geometric two-scale expansion approach only.

\medskip
\subsubsection{Random setting}
We turn to the case of a stationary ergodic random coefficient field~$\Aa$. For simplicity and illustration, we shall focus on the following Gaussian model.
\begin{defin}\label{def:gauss}
The coefficient field $\Aa$ is said to be {\it Gaussian with parameter $\beta>0$} if it has the form $\Aa=h(G)$ for some   $h\in\text{Lip}(\R^k)^{d\times d}$ with $k\ge1$ and for some $\R^k$-valued centered stationary Gaussian random field $G$ such that the covariance function
\[c(x):=\expec{G(x)\otimes G(0)},\qquad c:\R^d\to\R^{k\times k},\]
has $\beta$-algebraic decay at infinity, $|c(x)|\,\le\,(1+|x|)^{-\beta}$.
\end{defin}

The analysis of Theorem~\ref{th:main-per} can be repeated in this Gaussian setting, and leads to the following two-scale expansion result with optimal error estimate up to the optimal maximal timescale. Note that the maximal timescale depends on the decay rate $\beta$ for correlations and saturates in case of integrable decay $\beta>d$ (which corresponds to the strongest mixing possible in the Gaussian setting and yields the central limit theorem scaling for large-scale averages of the coefficients).
This result extends~\cite{BG} to higher-order accuracy.
The proof is displayed in Section~\ref{sec:proof-random}.
Exactly as in the periodic case above, a corresponding result could also be obtained in terms of the  geometric two-scale expansion; we skip the detail for shortness.

\begin{theor1}\label{th:main2}
Let $\Aa$ be Gaussian with parameter $\beta>0$ in the above sense, and define
\[\ell_*:=\lceil\tfrac{\beta\wedge d}2\rceil.\]
We can construct spectral correctors $\{\psi^n\}_{n\le\ell_*}$ and $\{\zeta^{n,m}\}_{n+2m\le\ell_*-3}$ as well-behaved solutions of some hierarchy of elliptic problems on the probability space,  homogenized tensors $\{\bar\bb^n\}_{n\le\ell_*}$, and a Fourier multiplier $\gamma_\ell$ with $|\gamma_\ell(\xi)|\le1$, cf.~Appendix~\ref{sec:rand-cor}, such that the following holds.
For any impulse $f\in C^\infty(\R;H^\infty(\R^d))$ with $f=0$ for~$t<0$, the ancient solution $u_\e$ of the heterogeneous wave equation~\eqref{eq:hyperb}
is accurately described by the corresponding spectral two-scale expansion $S_\e^\ell[\bar u_\e^\ell,f]$ given in~\eqref{eq:goal1}, where $\bar u_\e^\ell$ is an ancient solution of the corresponding formal homogenized equation~\eqref{eq:goal2} in one of the meanings provided in Lemma~\ref{lem:apriori-baru-sp}.
More precisely, we have the following error estimate: for all $t\ge0$ and $q<\infty$,
\begin{multline*}
\|u_\e^t-S^{\ell_*}_\e[\bar u_\e^{\ell_*;t},f^t]\|_{\Ld^q(\Omega;\Ld^2(\R^d))}+\|D(u_\e^t-S^{\ell_*}_\e[\bar u_\e^{\ell_*;t},f^t])\|_{\Ld^q(\Omega;\Ld^2(\R^d))}\\
\,\lesssim_q\,\|\langle\cdot\rangle\langle D\rangle^{C\ell}f\|_{\Ld^1(\R;\Ld^2(\R^d))}\times\left\{\begin{array}{lll}
\langle t\rangle\e^{\frac d2}|\!\log\e|^{\frac12}&:&\beta>d,~\text{$d$ even},\\
\langle t\rangle\e^{\frac d2}\big(\langle t\rangle^{\frac12}\wedge\e^{-\frac12}\big)&:&\beta > d,~\text{$d$ odd},\\
\langle t\rangle\e^{\frac d2}|\!\log\e|&:&\beta=d,~\text{$d$ even},\\
\langle t\rangle\e^{\frac d2}\big((\langle t\rangle|\!\log\e|)^{\frac12}\wedge\e^{-\frac12}\big)&:&\beta = d,~\text{$d$ odd},\\
\langle t\rangle\e^{\frac\beta2}|\!\log\e|^{\frac12}&:&\beta<d,~\beta \in 2\N,\\
\langle t\rangle\e^{\frac\beta2}\big(\langle t\rangle^{1-\{\frac\beta2\}}\wedge\e^{-\{\frac\beta2\}}\big)&:&\beta<d,~\beta \notin 2\N,
\end{array}\right.
\end{multline*}
with the short-hand notation $\{\frac\beta2\}=\frac\beta2-\lfloor \frac \beta 2 \rfloor\in[0,1)$ for the fractional part of $\frac\beta2$.
\end{theor1}

{For short times $t=O(1)$, in the {setting of integrable covariance} $\beta>d$, the above two-scale expansion error estimate is  $O(\e^{d/2})$ (up to a logarithmic correction in even dimensions). 
This is optimal and random fluctuations of $u_\e$ become dominant beyond this scaling since fluctuations naturally have the scaling $O(\e^{d/2})$ of the central limit theorem, cf.~e.g.~\cite{DGO1}.
For longer times, the two-scale expansion error further depends on whether the dimension is odd or even due to the growth of correctors: {For $\beta>d$}, the two-scale expansion error remains negligible $\ll1$, in the sense that the wave can be accurately described in terms of some homogenized equation, only provided that
\begin{equation}\label{e.ballistic-regime}
t \ll \left\{\begin{array}{lll}
\e^{-\frac 13}&:&d=1,\\
\e^{-\frac d2}|\!\log \e|^{-\frac12}&:&\text{$d$ even},\\
\e^{-\frac{d-1}2}&:&\text{$d$ odd $>1$}.
\end{array}\right.
\end{equation}
Up to such timescales, the above result can be used in particular to derive ballistic transport properties of the wave, see~\cite{BG,DGS}, as well as to derive spectral information in a suitable low-energy regime, see~\cite{DG23a}.
{ Although the two-scale expansion cannot be pushed further in general (except in the case of very specific structure of the coefficient field, see e.g.~the remark above Corollary~1.3 in \cite{DG23a}), ballistic transport could hold for longer times (see e.g.~the case of matched impedance in~\cite{SS-24}).
}

\begin{rems}[Extensions]\label{rem:ill-prepared}\nopagebreak\
\begin{enumerate}[$\bullet$]
\item {\it Nontrivial initial data:}
In Theorems~\ref{th:main-per} and~\ref{th:main2}, we focus on well-prepared data, or equivalently, on ancient solutions of the heterogeneous wave equation, cf.~\eqref{eq:hyperb}.
If we rather consider the initial-value problem
\begin{equation*}
\left\{\begin{array}{ll}
(\partial_t^2-\nabla\cdot\Aa(\tfrac\cdot\e)\nabla)z_\e=f,&\text{in $\R^+\times\R^d$},\\
z_\e|_{t=0}=u^\circ,&\text{in $\R^d$},\\
\partial_tz_\e|_{t=0}=v^\circ,&\text{in $\R^d$},
\end{array}\right.
\end{equation*}
with initial data $u^\circ,v^\circ\in\Ld^2(\R^d)$,
then, because of ill-preparedness, an $O(\e)$ contribution with almost-periodic time oscillations with $O(\e^{-1})$ frequency is expected to appear and to maintain forever: this is formally described in Section~\ref{sec:ill-posed} below and shows that a two-scale description cannot hold beyond accuracy $O(\e)$.
This issue is naturally by-passed by rather considering the time-averaged solution
\[z_{\e,\theta}^t(x)\,:=\,\int_0^\infty\theta(t-s)\,z_\e^s(x)\,ds,\qquad\text{for some $\theta\in C^\infty_c(\R)$.}\]
Indeed, setting $t_0:=\inf(\supp\theta)$, an integration by parts ensures that $z_{\e,\theta}$ is the ancient solution of
\begin{gather}\label{eq:time-average-form}
\left\{\begin{array}{ll}
(\partial_t^2-\nabla\cdot\Aa(\tfrac\cdot\e)\nabla)z_{\e,\theta}=\theta'(t)u^\circ+\theta(t) v^\circ+f_\theta,&\text{in $\R\times\R^d$},\\
z_{\e,\theta}^t=0,&\text{for $t<t_0$},
\end{array}\right.
\end{gather}
{in terms of the time-averaged impulse $f_\theta^t(x):=\int_0^\infty\theta(t-s)\,f^s(x)\,ds$.}
Hence, an effective approximation for the time-averaged solution $z_{\e,\theta}$ is obtained as a direct consequence of Theorems~\ref{th:main-per} and~\ref{th:main2} for ancient solutions.
{Another way to solve this issue is to consider oscillating initial data $(u^\circ_\e,v^\circ_\e)$ in form 
of a spectral two-scale expansion.}

\medskip
\item {\it Heterogeneous mass density:}
As in \cite{ALR}, we may replace $\partial_t^2$ by $\rho(\frac x\e)\partial_t^2$ in the wave equation~\eqref{eq:hyperb}. Provided that the weight function~$\rho$ satisfies the uniform non-degeneracy condition $\frac1{C_0}\le\rho(x)\le C_0$ for some constant~$C_0>0$, it does not change much in the analysis once the definitions of correctors are suitably adapted. The necessary changes are transparent and we skip the detail for shortness\qedhere
\end{enumerate}
\end{rems}

\subsection{Well-posedness of homogenized equations}\label{sec:discuss}

We start by discussing the formal homogenized equation~\eqref{eq:goal2} obtained with the spectral approach, where we recall that $\bar \bb^{k}=0$ for all $k$ even.
The obstacle to the well-posedness of this equation is the lack of ellipticity of the operator 
\begin{equation}\label{eq:hom-operator}
-\nabla\cdot\Big(\bar\Aa+\sum_{k=2}^\ell\bar\bb^k\odot(\e\nabla)^{k-1}\Big)\nabla,
\end{equation}
because of dispersive corrections.
Indeed, the next-order homogenized coefficient $\bar \bb^3$ can be proved to be non-negative, cf.~\cite{SaSy-91}, so that  equation~\eqref{eq:goal2} is ill-posed in general.
This is not new to homogenization, and a similar difficulty occurs in the elliptic setting when studying higher-order two-scale expansions, see e.g.~\cite{BLP-78,DO1}. In this case, one typically uses an inductive method, which, for the wave equation, would read as follows:
for $\ell\ge1$, set $\bar w^\ell_\e:=\sum_{k=1}^\ell\e^{k-1}\tilde w^k$, where $\tilde w^1$ is the solution of
\begin{equation*}
\left\{\begin{array}{ll}
(\partial_t^2-\nabla\cdot\bar\Aa \nabla)\tilde w^1=f,&\text{in $\R\times\R^d$},\\
\tilde w^1=f=0,&\text{for $t<0$},
\end{array}\right.
\end{equation*}
and where for $2\le k\le\ell$ we inductively define $\tilde w^k$ as the unique solution of
\begin{equation}\label{eq:induct-ell}
\left\{\begin{array}{ll}
(\partial_t^2-\nabla\cdot\bar\Aa \nabla)\tilde w^k=\sum_{j=2}^k \nabla\cdot(\bar\bb^j\odot\nabla^{j-1})\nabla\tilde w^{k+1-j},&\text{in $\R\times\R^d$}\\
\tilde w^{k}=0,&\text{for $t<0$}.\qedhere
\end{array}\right.
\end{equation}
It is easily checked that this $\bar w_\e^\ell$ indeed satisfies~\eqref{eq:goal2}. However, as originally observed in~\cite{ALR} (in the similar setting of~\eqref{eq:form-hom}), this notion of solution displays an immoderate growth {in time}, which destroys any hope of using it for an accurate description of~\eqref{eq:hyperb} on long timescales. More precisely, the energy norm $\|\nabla\bar w_\e^{\ell;t}\|_{\Ld^2(\R^d)}$ is expected to behave like~$O(\langle\e t\rangle^{\ell-1})$, which would make the approximation $u_\e\sim S_\e^\ell[\bar w_\e^\ell,f]$ trivially false on long timescales $t\gg\e^{-1}$.
This time growth {(also called secular growth)} appears as a snowball effect as corrector terms in the above hierarchy of equations for $\{\tilde w^k\}_{k\ge1}$ have the preceding profiles as sources.

Instead of this naïve inductive method, one must look for another way to rearrange equation~\eqref{eq:goal2}.
In fact, as $\bar\Aa\ge\lambda\Id$, we note that the Fourier symbol of the operator~\eqref{eq:hom-operator} remains positive in a fixed Fourier support for~$\e$ small enough. Therefore, if the spatial Fourier transform of the impulse $f$ is compactly supported (uniformly in time), then for~$\e$ small enough the Duhamel formula allows us to define a unique solution that keeps the same compact support in Fourier space at all times.
Without this additional assumption on~$f$, the operator~\eqref{eq:hom-operator} needs to be modified at high frequencies $O(\frac1\e)$ to ensure ellipticity, and there are different ways to proceed. { In the following we discuss the three different regularizing terms that we consider in Theorem \ref{th:main-per}. The resulting solutions satisfy \eqref{eq:goal2} up
to an error of the order $O(\e^{\ell})$.} We briefly describe high-frequency filtering, high{er}-order regularization, and the so-called Boussinesq trick: these three approaches happen to be essentially equivalent up to higher-order $O(\e^\ell)$ errors, and we refer to Section~\ref{sec:homog-eqn/spec} for the details.
\begin{enumerate}[\hspace{-0.15cm}(I)]
\item\emph{High-frequency filtering:}\\
In~\cite{ALR}, the authors proposed to use a low-pass truncation, which amounts to filtering out high frequencies of the impulse: equation~\eqref{eq:goal2} is then replaced by 
\begin{equation*}
\partial_t^2\bar u_\e^{\text{(I)},\ell}-\nabla\cdot\Big(\bar\Aa+\sum_{k=2}^\ell\bar\bb^k\odot(\e\nabla)^{k-1}\Big)\nabla\bar u_\e^{\text{(I)},\ell}= \chi(\e^\alpha\nabla)f,
\end{equation*}
for some $\alpha\in(0,1)$ and some $\chi\in C^\infty_c(\R^d)$ with $\chi|_{\frac12B}=1$ and $\chi|_{\R^d\setminus B}=0$.
\smallskip\item\emph{Higher-order regularization:}\\
The alternative method used in~\cite{BG} amounts to regularizing the operator~\eqref{eq:hom-operator} by adding a high{er}-order positive operator $\kappa_\ell (\e|\nabla|)^\ell(-\triangle)$, where the factor $\kappa_\ell>0$ is chosen for instance as the smallest value that ensures the following uniform positivity,
\[\qquad\xi \cdot\Big(\bar\Aa+\sum_{k=2}^\ell \bar\bb^{k}\odot(i \xi)^{\otimes(k-1)}+\kappa_\ell |\xi|^{\ell}\Big)\xi \,\ge\, \tfrac12 \lambda |\xi|^2,\qquad\text{for all $\xi \in \R^d$.}\]

\smallskip\item\emph{Boussinesq trick:}\\
This last method proceeds by rearranging the ill-posed equation~\eqref{eq:goal2} and is inspired by the standard perturbative procedure to rearrange the ill-posed Boussinesq equation in the theory of water waves, see e.g.~\cite{Christov-Maugin-Velarde}.
This so-called Boussinesq trick was first adapted to the present setting by Lamacz~\cite{Lamacz-11} in one space dimension for $\ell=3$. It was extended in~\cite{DLS-14} to higher dimension for $\ell=3$, and further extended to all orders $\ell\ge3$ by Abdulle and Pouchon~\cite{Pouch-19}.
It is somehow of the same spirit as the higher-order regularization above, but with the additional twist that it further uses the wave equation itself.
This approach is slightly more intrinsic than the previous two ones, but it also has the disadvantage of involving derivatives of the impulse. Let us illustrate the main idea for $\ell=3$. We first choose $\kappa_3>0$ as the smallest value such that 
\[\xi\cdot(\bar \bb^3\odot  (i\xi)^{\otimes 2})\xi +\kappa_{3}|\xi|^2(\xi\cdot\bar \Aa\xi)\, \ge \, 0,\qquad\text{for all $\xi \in \R^d$.}\]
We then decompose the operator~\eqref{eq:hom-operator} as 
\begin{multline*}
\qquad-\nabla\cdot\big(\bar\Aa+\bar\bb^3\odot(\e\nabla)^{2}\big) \nabla\\
\,=\, -\nabla\cdot\Big(\bar\Aa(1-\kappa_3\e^2\triangle)+\bar\bb^3\odot(\e\nabla)^{2}\Big) \nabla -\kappa_{3} \e^2\triangle(\nabla \cdot \bar \Aa \nabla),
\end{multline*}
and we use that at leading order the equation~\eqref{eq:goal2} yields
\[\qquad\nabla \cdot \bar \Aa \nabla \bar u_\e^3\,=\,\partial_t^2 \bar u_\e^3-f+O(\e^2),\]
to the effect that one may reformulate \eqref{eq:goal2}  as
$$
\qquad~\partial_t^2 (1-\e^2 \kappa_3 \triangle) \bar u^{\text{(III)},3}_\e\!-\!\nabla\!\cdot\!\Big(\bar\Aa(1-\kappa_3\e^2\triangle)+\bar\bb^3\odot(\e\nabla)^{2}\Big) \nabla \bar u^{\text{(III)},3}_\e= (1-\e^2 \kappa_3 \triangle)f,
$$
up to an error of order $O(\e^4)$. By our choice of $\kappa_3$, this equation is well-posed.
This method extends to arbitrary order and we refer to Section~\ref{sec:homog-eqn/spec} for the details.
\end{enumerate}

Next, we turn to the well-posedness of the corresponding formal homogenized equation~\eqref{eq:form-hom} obtained for the geometric two-scale expansion.
While for equation~\eqref{eq:goal2} the difficulty only came from the lack of ellipticity of the spatial differential operator~\eqref{eq:hom-operator},
the existence theory for equation~\eqref{eq:form-hom} is much more delicate as this equation involves higher-order mixed space-time derivatives. Just as for equation~\eqref{eq:goal2}, the inductive method~\eqref{eq:induct-ell} used in the elliptic setting leads to secular growth of the approximate solution and is of no use in the present situation.
Before being able to use high-frequency filtering, high{er}-order regularization, or the Boussinesq trick, we need to get rid of higher-order mixed space-time derivatives in~\eqref{eq:form-hom}, which can be done by iteratively using the equation itself (quite in the spirit of the above presentation of the Boussinesq trick for~$\ell=3$). This is called the criminal approximation in~\cite{ALR}. The thorough revisiting of this idea is the object of Section~\ref{sec:homog-eqn/hyp}: more precisely, if $\bar v_\e^\ell$ solves~\eqref{eq:form-hom}, then it is shown to satisfy an equation of the form
\[\partial_t^2\bar v_\e^\ell-\nabla\cdot\Big(\bar\Aa+\sum_{k=2}^\ell\bar\bb^k\odot(\e\nabla)^{k-1}\Big)\nabla\bar v_\e^\ell
\,=\, f+\e^2\nabla\cdot\Big(\sum_{n=1}^{\ell-2}\bar\cc^n\odot(\e D)^{n-1}\Big)\nabla f+O(\e^\ell),\]
where $\{\bar \bb^n\}_n$ coincides with the spectral homogenized coefficients and where $\{\bar \cc^n\}_n$ is some family of nonlinear combinations of the hyperbolic coefficients $\{\bar \Aa^{n,m}\}_{n,m}$; see Lemma~\ref{lem:reform} for a precise statement.
Now the differential operator in the left-hand side of this equation is the same as in~\eqref{eq:goal2}: it displays a lack of ellipticity, but the same approach to well-posedness can be repeated, using either high-frequency filtering, high{{er}-order regularization, or the Boussinesq trick. Note that the right-hand side in the above reformulation of the homogenized equation~\eqref{eq:form-hom} differs from the homogenized equation~\eqref{eq:goal2} obtained with the spectral approach, but there is no contradiction as the spectral and hyperbolic two-scale expansions also differ: this demonstrates the actual complexity of the link between spectral and hyperbolic correctors; see Section~\ref{sec:fromGeom2Spec}.

\subsection{Spectral approach and two-scale expansion}\label{sec:Bloch}
This section constitutes the main insight of this contribution: the form of the spectral two-scale ansatz~\eqref{eq:goal1},
\begin{multline}\label{eq:spectral-Ansatz}
u_\e\,\sim\,S^\infty_\e[\bar u_\e,f]\,:=\,\sum_{n\ge0}\e^n\psi^n(\tfrac\cdot\e)\odot\gamma(\e\nabla)\nabla^n\bar u_\e\\
+\e^3\sum_{n,m\ge0}(-1)^m\e^{n+2m}\zeta^{n,m}(\tfrac\cdot\e)\odot\gamma(\e\nabla)\nabla^{n+1}\partial_t^{2m}f.
\end{multline}
For that purpose, we focus on the periodic setting and first proceed to a fine Floquet--Bloch analysis of the solution $u_\e$ of the heterogeneous wave equation~\eqref{eq:hyperb}.
Starting point is an application of the Floquet transform, which is known to transform the heterogeneous wave operator~\mbox{$-\nabla\cdot\Aa(\tfrac\cdot\e)\nabla$} on $\Ld^2(\R^d)$ into a family of fibered wave operators on the periodic Bloch space $\Ld^2(Q)$.

{When embedding the periodic cell into the physical space $\R^d$, there is an indeterminacy related to the choice of the origin ($q \in Q$), which we consider as an additional variable. Averaging over this variable in $Q$ allows us to place ourselves in the setting of continuum stationarity -- thus unifying the notation with the random setting. (All the upcoming results actually hold for $q$ fixed in the periodic setting.)
Hence, as in~\cite{BG,DGS,DS1}, we enrich} the structure by considering shifts of the periodic coefficient field and by augmenting the physical space $\R^d$ to include such shifts: we define $\tilde u_\e\in\Ld^\infty_\loc(\R^+;\Ld^2(\R^d\times Q))$ such that, for all $q\in Q$, $\tilde u_\e(\cdot,q)$ is the ancient solution of the following shifted wave equation,
\begin{equation}\label{eq:hyperb-augm}
\left\{\begin{array}{ll}
\big(\partial_t^2-\nabla\cdot\Aa(\tfrac\cdot\e+q)\nabla\big)\tilde u_\e(\cdot,q)=f,&\text{in $\R\times\R^d$},\\
\tilde u_\e(\cdot,q)=f=0,&\text{for $t<0$}.
\end{array}\right.
\end{equation}
In this augmented setting, the $\e$-Floquet transform~\cite{DGS,DS1} of an element $\tilde w\in\Ld^2(\R^d\times Q)$ is formally defined as
\[(\Vc_\xi^\e\tilde w)(q):=\int_{\R^d}e^{-i\xi\cdot y}\,\tilde w(y,q-\tfrac y\e)\,dy,\]
which is $Q$-periodic with respect to $q$. The Fourier inversion formula takes on the following guise, cf.~\cite[Lemma 2.2]{DGS},
\begin{equation}\label{eq:invert-Floquet}
\tilde w(x,q)=\int_{\R^d}e^{i\xi\cdot x}\,(\Vc_\xi^\e \tilde w)(\tfrac x\e+q)\,d^* \xi.
\end{equation}
This leads to a direct integral decomposition\footnote{Or, more precisely, a family of direct integral decompositions parameterized by $\e$, which changes the way the unit cell $Q$ is embedded in $\R^d$.}
\[\Ld^2(\R^d\times Q)=\int_{\R^d}^\oplus\Ld^2( Q)\,\mathfrak e_\xi\,d^*\xi\]
via Fourier modes $\mathfrak e_\xi(x):=e^{i\xi\cdot x}$.
Under this decomposition, the above wave equation~\eqref{eq:hyperb-augm} is equivalent to the following family of wave equations on the unit cell $Q${:} for all~$\xi\in\R^d$,
\begin{equation*}
\left\{\begin{array}{ll}
\big(\partial_t^2-\tfrac1{\e^2}(\nabla+i\e\xi)\cdot\Aa(\nabla+i\e\xi)\big)\Vc_\xi^\e\tilde u_\e=\hat f(\xi),&\text{in $\R\times Q$},\\
\Vc_\xi^\e\tilde u_\e=\hat f(\xi)=0,&\text{for $t<0$},
\end{array}\right.
\end{equation*}
where we recall that $\hat f$ stands for the spatial Fourier transform of $f$.
Solving this equation by means of Duhamel's formula and using~\eqref{eq:invert-Floquet} to invert the Floquet transform, we get
\begin{equation}\label{eq:dec-ueps-00}
\tilde u_\e^t(x,q)=\int_0^t\int_{\R^d}e^{i\xi\cdot x} \bigg(\frac{\sin\big(\tfrac1\e(t-s)(\Lc_{\e \xi})^{1/2}\big)}{\tfrac1\e(\Lc_{\e \xi})^{1/2}}1\bigg)(\tfrac x\e+q)\,\hat f^s(\xi)\,d^*\xi\,ds,
\end{equation}
with the short-hand notation 
\begin{equation}\label{eq:defLxi}
	\Lc_{\xi}:=-(\nabla+i\xi)\cdot\Aa(\nabla+i\xi).
\end{equation}
In other words, the solution can be decomposed as a superposition of fibered evolutions on $\Ld^2(Q)$, and, since the force~$f$ is not oscillating, only the fibered spectral measures associated with the constant function~$1$  matter.
It remains to evaluate the space-time oscillating factor in  formula~\eqref{eq:dec-ueps-00} and to extract an effective behavior as $\e\ll1$.

\begin{prop}\label{prop:Bloch}
Let $\Aa$ be $Q$-periodic, let the impulse $f\in C^\infty(\R;H^\infty(\R^d))$ satisfy $f=0$ for $t<0$, and assume for simplicity that the spatial Fourier transform $\hat f$ is compactly supported uniformly in time. For all $\xi$, the self-adjoint operator $\Lc_\xi$ on $\Ld^2(Q)$ {defined in \eqref{eq:defLxi}} has discrete spectrum and we denote its smallest eigenvalue by~$\lambda_\xi\ge0$. For $|\xi|\ll1$, this eigenvalue is simple and we denote by $w_\xi$ a corresponding normalized eigenfunction.
We then set
\begin{equation}\label{eq:defin-pi-ibot}
\pi_\xi1\,:=\,\expec{\overline{w_\xi}}w_\xi,\qquad\text{and}\qquad\pi_\xi^\bot1\,:=\,1-\pi_\xi1,
\end{equation}
we note that the map $\R^d\to\R^+\times\Ld^2( Q):\xi\mapsto(\lambda_\xi,\pi_\xi1)$ is analytic for $|\xi|\ll1$, and that for $|\xi|\lesssim1$ and $\e\ll1$ {it holds that}
\begin{equation}\label{eq:prop-lam-pi-eps}
|\tfrac1{\e^2}\lambda_{\e\xi}|\,\lesssim\,|\xi|^2,\qquad\text{and}\qquad\|\pi_{\e\xi}1-1\|_{\Ld^2( Q)}\,\lesssim\,\e|\xi|.
\end{equation}
With this notation, for $\e\ll_f1$ (depending on the Fourier support of $f$), the above formula~\eqref{eq:dec-ueps-00} for $\tilde u_\e$ can be expanded as follows: for all $n\ge1$,
\begin{multline}\label{eq:full-decomp-per-u}
\tilde u_\e^t(x,q)\,=\,\int_{\R^d}e^{i\xi\cdot x}\,(\pi_{\e\xi}1)(\tfrac x\e+q)\bigg(\int_0^t\frac{\sin\big((t-s)(\tfrac1{\e^2}\lambda_{\e \xi})^{1/2}\big)}{(\tfrac1{\e^2}\lambda_{\e \xi})^{1/2}}\,\hat f^s(\xi)\,ds\bigg)\,d^*\xi\\
+\e^3\sum_{m=0}^{n-1}(-1)^m\e^{2m}\int_{\R^d}e^{i\xi\cdot x}\,\Psi_{\xi,\e}^m(\tfrac x\e+q)\,\partial_t^{2m}\hat f^t(\xi)\,d^*\xi\\
+O(\e^{2(n+1)})\,\|\langle\partial_t\rangle^{2n+1}\hat f\|_{\Ld^1\cap\Ld^\infty(\R;\Ld^1(\R^d))},
\end{multline}
in terms of
\begin{equation}\label{eq:defin-Psi-0}
\Psi_{\xi,\e}^m\,:=\,(\Lc_{\e\xi})^{-m-1}\tfrac1\e\pi_{\e\xi}^\bot1,
\end{equation}
which for $|\xi|\lesssim1$ is analytic with respect to $\e\ll1$ and satisfies
$\|\Psi_{\xi,\e}^m\|_{\Ld^2(Q)}\lesssim_m|\xi|$.
\end{prop}

We emphasize the structure of the above expansion~\eqref{eq:full-decomp-per-u}.
Since in~\eqref{eq:dec-ueps-00} only the fibered spectral measures associated with the constant function~$w_0\equiv1$ {(which is the ground state of $\Lc_0$)}  matter,   the main contribution in~\eqref{eq:full-decomp-per-u} is naturally given by the perturbed ground state $w_{\e\xi}\propto\pi_{\e\xi}1$. However, as the impulse~$f$ is not oscillating, hence is not adapted to oscillations of the ground state $w_{\e\xi}$, higher modes also create another non-vanishing contribution. In other words:
\begin{enumerate}[---]
\item The first term in~\eqref{eq:full-decomp-per-u} corresponds to the contribution of the ground state of the fibered operators~$\{\Lc_\xi\}_\xi$ and the time dependence is expressed by some effective evolution determined by the fibered ground eigenvalues $\{\lambda_\xi\}_\xi$.
\smallskip\item The second term in~\eqref{eq:full-decomp-per-u} is only of order $O(\e^3)$ and is induced by higher modes. More precisely, the ill-preparedness of the impulse $f$ creates a non-trivial oscillatory contribution in Duhamel's formula due to higher modes, which amounts after time integration to a contribution that is local with respect to $f$.
In particular, note that this term vanishes outside the support of $f$.
\end{enumerate}
For comparison, the Bloch wave approach in~\cite{Lamacz-11,DLS-14,DLS-15,BG} rather  focused on the first term in~\eqref{eq:full-decomp-per-u}, thus neglecting the $O(\e^3)$ contribution of higher modes, and further replaced the oscillating factor $\pi_{\e\xi}1$ by a simpler (not normalized) proxy that is easier to characterize but that yields an additional $O(\e)$ error.

\begin{proof}[Proof of Proposition~\ref{prop:Bloch}]
To evaluate the space-time oscillating factor in~\eqref{eq:dec-ueps-00}, we must investigate the spectrum of~$\Lc_\xi$ in the perturbative regime $|\xi|\ll1$. As this self-adjoint operator has compact resolvent by Rellich's theorem, it has discrete spectrum.
We denote by $\lambda_\xi$ its smallest eigenvalue, which is nonnegative as $\Lc_\xi$ is.
Note that for~\mbox{$\xi=0$} the smallest eigenvalue of $\Lc_0$ is $\lambda_0=0$ and is simple (with constant eigenfunction).
Since the perturbation $\Lc_\xi-\Lc_0$ is $\Lc_0$-bounded with relative norm $\lesssim1+|\xi|^2$, standard perturbation theory~\cite{Kato} together with the discreteness of the spectrum of $\Lc_0$ ensures that the smallest eigenvalue~$\lambda_\xi$ remains simple for $|\xi|\ll1$ small enough. Moreover, the branch of eigenvalues $\xi\mapsto\lambda_\xi$ is analytic for $|\xi|\ll1$, and there is a corresponding analytic branch of eigenfunctions.
Recall the definition of the corresponding projectors $\pi_\xi,\pi_\xi^\bot$ in the statement, and note that $\pi_01=1$, {so that}~\eqref{eq:prop-lam-pi-eps} follows.
Now expanding the constant function $1$ with respect to those projectors, identity~\eqref{eq:dec-ueps-00} turns into
\begin{multline}\label{eq:dec-ueps-0}
\tilde u_\e^t(x,q)\,=\,\int_{\R^d}e^{i\xi\cdot x}(\pi_{\e\xi}1)(\tfrac x\e+q)\bigg(\int_0^t \frac{\sin\big((t-s)(\tfrac1{\e^2}\lambda_{\e \xi})^{1/2}\big)}{(\tfrac1{\e^2}\lambda_{\e \xi})^{1/2}}\,\hat f^s(\xi)\,ds\bigg)\,d^*\xi\\
+\int_0^t\int_{\R^d}e^{i\xi\cdot x} \bigg(\frac{\sin\big(\tfrac1\e(t-s)(\Lc_{\e \xi})^{1/2}\big)}{\tfrac1\e(\Lc_{\e \xi})^{1/2}}\pi_{\e\xi}^\bot1\bigg)(\tfrac x\e+q)\,\hat f^s(\xi)\,d^*\xi\,ds.
\end{multline}
The first right-hand side term is already of the desired form, cf.~\eqref{eq:full-decomp-per-u}.
We turn to the second term, which captures the contribution of higher modes.
Due to the discreteness of the spectrum, for $|\xi|\ll1$, the operator $\Lc_\xi$ has a spectral gap $\Lc_\xi|_{\Im\pi_{\xi}^\bot}\,\gtrsim\,1$. Combined with~\eqref{eq:prop-lam-pi-eps}, this yields the following bound for~\eqref{eq:defin-Psi-0},
\[\|\Psi_{\xi,\e}^m\|_{\Ld^2( Q)}\lesssim_m|\xi|.\]
For $n\ge1$, noting that iterated integration by parts in the time integral yields for all $\lambda>0$,
\begin{multline*}
\int_0^t\frac{\sin\big(\tfrac1\e(t-s)\lambda^{1/2}\big)}{\tfrac1\e \lambda^{1/2}}\,\hat f^s(\xi)\,ds\,=\,\sum_{m=0}^{n} \big(\tfrac{\e^2}{\lambda}\big)^{m+1}(-\partial_t^2)^m\hat f^t(\xi)\\
-\big(\tfrac{\e^2}{\lambda}\big)^{n+1}\int_0^t\cos\big(\tfrac1\e(t-s)\lambda^{1/2}\big)\partial_s(-\partial_s^2)^{n}\hat f^s(\xi)\,ds,
\end{multline*}
where we used the vanishing assumption for the impulse at initial time $s=0$,
it then follows from functional calculus that
\begin{multline*}
\int_0^t\int_{\R^d}e^{i\xi\cdot x} \bigg(\frac{\sin\big(\tfrac1\e(t-s)(\Lc_{\e \xi})^{1/2}\big)}{\tfrac1\e(\Lc_{\e \xi})^{1/2}}\pi_{\e\xi}^\bot1\bigg)(\tfrac x\e+q)\,\hat f^s(\xi)\,d^*\xi\,ds\\
\,=\,\e^3\sum_{{m}=0}^{n-1}(-1)^{m}\e^{2{m}}\int_{\R^d}e^{i\xi\cdot x}\,\Psi_{\xi,\e}^{{m}}(\tfrac x\e+q)\,\partial_t^{2{m}}\hat f^t(\xi)\,d^*\xi\\
+O(\e^{2(n+1)})\,\|\langle\partial_t\rangle^{2n+1}\hat f\|_{\Ld^1\cap\Ld^\infty(\R;\Ld^1(\R^d))}.
\end{multline*}
Inserting this into~\eqref{eq:dec-ueps-0} yields the conclusion~\eqref{eq:full-decomp-per-u}.
\end{proof}

We turn to the applicability of this spectral computation beyond the periodic setting.
In case of a stationary ergodic random coefficient field $\Aa$, a similar Floquet decomposition~\eqref{eq:dec-ueps-00} can be justified, cf.~\cite{DGS,DS1}, but the corresponding fibered operators $\{\Lc_{\xi}\}_\xi$ are then defined on $\Ld^2(\Omega)$, where $\Omega$ is the underlying probability space, and typically have non-discrete spectrum. For instance, the spectrum of $\Lc_0=-\nabla\cdot\Aa\nabla$ on $\Ld^2(\Omega)$ is expected to be made of a simple eigenvalue at~$0$ embedded at the bottom of {an} absolutely continuous spectrum, cf.~\cite{DS1}.
In this setting, the Floquet--Bloch theory fails and the above perturbative spectral computation cannot be adapted.
As shown in~\cite{BG,DGS}, however, an `approximate spectral theory' can be developed: formal Rayleigh--Schrödinger series can be approximately constructed up to a certain accuracy, leading to approximate Bloch waves that can be used to approximately diagonalize the heterogeneous wave operator and describe the flow on long timescales.
Equivalently, we may start from a two-scale ansatz given by the formal $\e$-expansion of the spectral formula~\eqref{eq:full-decomp-per-u}, and then use PDE techniques to show that it indeed provides a good approximation of the flow to a certain accuracy. This approach is the one we use for the proof of Theorems~\ref{th:main-per} and~\ref{th:main2}: { it allows us to forget about the underlying spectral interpretation, which is only used to devise an educated guess for a two-scale ansatz.} 

Finally, let us describe the $\e$-expansion of the spectral formula~\eqref{eq:full-decomp-per-u}, showing that it takes {the} form of the spectral two-scale ansatz~\eqref{eq:spectral-Ansatz}, and let us derive the relevant hierarchy of PDEs for its coefficients. For that purpose, following~\cite{BG}, we first consider the (not normalized) ground state $\psi_{\e\xi}$ of $\Lc_{\e\xi}$ satisfying $\expec{\psi_{\e\xi}}=1$.
Expanding {formally}
\begin{equation}\label{eq:def-psi-expand}
\psi_{\e\xi} \sim \sum_{n\ge0}\e^{n}\check\psi^n_\xi, \qquad \lambda_{\e\xi}\sim\sum_{n\ge1} \e^{n}\check\lambda^n_\xi,
\end{equation}
and separating powers of $\e$ in the eigenvalue relation $\Lc_{\e\xi}\psi_{\e\xi}=\lambda_{\e\xi}\psi_{\e\xi}$, we  find that the maps $\check \psi^n_\xi:Q \to \C$ and coefficients $\check\lambda^n_\xi \in \C$ are defined inductively by $\check\psi^0_\xi=1$, $\check\lambda^0_\xi=0$, and for all $n\ge1$,
\begin{gather}\label{e.spec-intro1}
-\nabla\cdot\Aa\nabla\check\psi^{n}_\xi=\nabla\cdot(\Aa i\xi \check\psi^{n-1}_\xi)+i\xi\cdot \Aa(\nabla\check\psi^{n-1}_\xi+ i\xi\check\psi^{n-2}_\xi)+\sum_{k=1}^n\check \lambda^k_\xi\check\psi^{n-k}_\xi,\\
\expec{\check\psi^n_\xi}=0,\qquad \check \lambda^n_\xi=-\expec{i\xi\cdot(\Aa\nabla\check\psi^{n-1}_\xi+\Aa i\xi\check\psi^{n-2}_\xi)}.\nonumber
\end{gather}
(Recall that by convention we implicitly set $\check\psi_\xi^{n}\equiv0$ for $n<0$.)
This hierarchy of equations uniquely defines $\{\check\psi_\xi^n,\check\lambda_\xi^n\}_n$ by the Fredholm alternative since by induction the periodic average of the right-hand side of~\eqref{e.spec-intro1}  vanishes; cf.~Section~\ref{sec:def-corr-spec}. Note that we find $\check\lambda^1_\xi=0$ in agreement with~\eqref{eq:prop-lam-pi-eps}.
Next, we normalize $\psi_{\e\xi}$ to get a normalized ground state $w_{\e\xi}$ and we define the corresponding projections $\pi_{\e\xi}$, cf.~\eqref{eq:defin-pi-ibot},
\begin{equation}\label{eq:link-w-psi}
w_{\e \xi}=\frac{\psi_{\e\xi}}{\expec{|\psi_{\e\xi}|^2}^{1/2}}, \qquad \expec{\overline{w_{\e \xi}}}=\frac1{\expec{|\psi_{\e\xi}|^2}^{1/2}},
\qquad\pi_{\e\xi}1=\frac{\psi_{\e\xi}}{\expec{|\psi_{\e\xi}|^2}},
\end{equation}
and, in terms of~\eqref{eq:def-psi-expand},
\begin{eqnarray}
\tfrac1\e\pi^\bot_{\e\xi}1&=&\frac{1}{\expec{|\psi_{\e\xi}|^2}}\tfrac1\e\big(-\psi_{\e\xi}+\expec{|\psi_{\e\xi}|^2}\big)\nonumber\\
&=&\frac1{\expec{|\psi_{\e\xi}|^2}}\sum_{n\ge0}\e^n \Big(-\check\psi^{n+1}_\xi+\sum_{k=0}^{n+1}\expec{\overline{\check\psi^{n+1-k}_\xi}\check\psi^k_\xi}\Big).\label{eq:expand-pibot1}
\end{eqnarray}
We turn to the $\e$-expansion of $\{\Psi_{\xi,\e}^m\}_m$, cf.~\eqref{eq:defin-Psi-0}: we write their expansions as
\begin{equation}\label{eq:expand-Psixieps}
\Psi_{\xi,\e}^m\,\sim\,\frac1{\expec{|\psi_{\e\xi}|^2}}\sum_{n\ge0}\e^n\check\zeta_\xi^{n,m},
\end{equation}
and it remains to write PDEs to characterize the coefficients $\{\check\zeta_\xi^{n,m}\}_{n,m}$.
For $m=0$, inserting~\eqref{eq:expand-pibot1} and separating powers of $\e$ in the defining equation
\begin{equation*}
\Lc_{\e\xi} \Psi_{\xi,\e}^0=\tfrac1\e\pi_{\e\xi}^\bot1,
\end{equation*}
we find that the maps $\check\zeta_\xi^{n,0}:Q\to\C$ are defined inductively for all $n\ge0$ by
\begin{equation}\label{e.spec-intro3}
-\nabla\cdot\Aa\nabla\check\zeta^{n,0}_\xi=\nabla\cdot(\Aa i\xi\check\zeta^{n-1,0}_\xi) +i\xi\cdot\Aa\big(\nabla\check\zeta^{n-1,0}_\xi+i\xi\check\zeta^{n-2,0}_\xi\big)-\check\psi^{n+1}_\xi+\sum_{k=0}^{n+1}\expec{\overline{\check\psi^{n+1-k}_\xi}\check\psi^k_\xi},
\end{equation}
with nontrivial integration constant fixed as
\begin{equation}\label{e.spec-intro4}
\expecm{\check\zeta^{n,0}_\xi}\,:=\,-\sum_{k=1}^{n}\sum_{l=2}^{k+2}(\overline{\check\lambda^l_\xi/\check\lambda^2_\xi})\expec{\overline{\check\psi^{k+2-l}_\xi} \check\zeta^{n-k,0}_\xi}.
\end{equation}
(Recall again that by convention we implicitly set $\check\zeta_\xi^{n,0}\equiv0$ for $n<0$.)
Again, these objects are well-defined by induction and the Fredholm alternative since our choice~\eqref{e.spec-intro4} precisely ensures that the right-hand side of~\eqref{e.spec-intro3} has vanishing periodic average; cf.\@ Section~\ref{sec:def-corr-spec}.
Next, we argue iteratively for $m\ge1$:
starting from the defining equation
\[\Lc_{\e\xi}\Psi_{\xi,\e}^{m}= \Psi_{\xi,\e}^{m-1},\]
we find that the maps $\check\zeta_\xi^{n,m}:Q\to\C$ are defined inductively for all $n\ge0$ by
\begin{equation}\label{e.spec-intro4re}
-\nabla\cdot\Aa\nabla\check\zeta^{n,m}_\xi\,=\,\nabla\cdot(\Aa i\xi\check\zeta^{n-1,m}_\xi)+i\xi\cdot\Aa\big(\nabla\check\zeta^{n-1,m}_\xi+i\xi\check\zeta^{n-2,m}_\xi\big)+\check\zeta^{n,m-1}_\xi,
\end{equation}
with nontrivial integration constant fixed as
\begin{equation}\label{e.spec-intro5re}
\expecm{\check\zeta^{n,m}_\xi}\,:=\,-(\overline{1/\check\lambda^2_\xi})\sum_{k=0}^{n+2}\expec{\overline{\check\psi^{n+2-k}_\xi }\check\zeta^{k,m-1}_\xi}-{\sum_{k=2}^{n}}\sum_{l=2}^{k+1}(\overline{\check\lambda^l_\xi/\check\lambda^2_\xi})\expec{\overline{\check\psi^{k+1-l}_\xi }\check\zeta^{n-k,m}_\xi}.
\end{equation}

Putting all this together, using expansions~\eqref{eq:def-psi-expand} and~\eqref{eq:expand-Psixieps}, and extracting the polynomial $\xi$-dependence of the coefficients in the notation,
\[\check\psi_\xi^n\,=\,\psi^n\odot(i\xi)^{\otimes n},
\qquad\check\lambda_\xi^{n+1}\,=\,\xi\cdot(\bar\bb^{n}\odot(i\xi)^{\otimes(n-1)})\xi,
\qquad\check\zeta_\xi^{n,m}\,=\,\zeta^{n,m}\odot(i\xi)^{\otimes (n+1)},\]
the spectral formula~\eqref{eq:full-decomp-per-u} appears to be precisely equivalent to the spectral two-scale ansatz~\eqref{eq:spectral-Ansatz},
with Fourier multiplier $\gamma$ given by
\[\gamma(\xi)\,:=\,\expec{|\psi_{\xi}|^2}^{-1},\]
and with $\bar u_\e$ satisfying the associated formal homogenized equation, cf.~\eqref{eq:goal2},
\begin{equation*}
\partial_t^2\bar u_\e-\nabla\cdot\Big(\bar\Aa+\sum_{n\ge2}\bar\bb^n\odot(\e\nabla)^{n-1}\Big)\nabla\bar u_\e\,\sim\,f,
\end{equation*}
where dispersive corrections correspond to derivatives of the fibered ground states $\{\lambda_\xi\}_\xi$ at~$\xi=0$.
We have also derived the PDE hierarchies defining correctors $\{\psi^n,\zeta^{n,m}\}_{n,m}$, cf.~\eqref{e.spec-intro1}, \eqref{e.spec-intro3}, and~\eqref{e.spec-intro4re}.
For the proof of Theorems~\ref{th:main-per} and~\ref{th:main2}, we replace infinite series by finite sums and show by PDE techniques that these are still good approximations for the wave flow.
In the random setting, just as in the elliptic case~\cite{GNO-quant,AKM-book,Gu-17,DO1}, we can only solve a finite number of the above corrector equations, which is why a homogenized description is only obtained up to some maximal timescale and accuracy, depending both on space dimension and on mixing properties of the coefficient field.

\subsection{Geometric approach and hyperbolic two-scale expansion}\label{sec:geom-insight}
Instead of star\-ting from the above spectral analysis,
another way to describe oscillations of the solution $u_\e$ of the heterogeneous wave equation~\eqref{eq:hyperb} is to appeal more directly to two-scale expansion techniques~\cite{BLP-78} and rather postulate the following natural hyperbolic two-scale ansatz,
\begin{equation}\label{eq:2scale-hyp}
u_\e\,\sim\,H^\infty_\e[\bar v_\e]\,:=\,\sum_{n,m\ge0}\e^{n+m} \phi^{n,m}(\tfrac\cdot\e)\odot\nabla^n\partial_t^m\bar v_\e,
\end{equation}
where time and space play essentially symmetric roles and where $\bar v_\e$ should satisfy some effective (constant-coefficient) hyperbolic equation.
Inserting this ansatz into the heterogeneous wave equation~\eqref{eq:hyperb} and separating powers of~$\e$, we are led to defining hyperbolic correctors as Allaire, Lamacz, and Rauch in~\cite[Definition~2.2]{ALR}.
These correctors can be viewed as a refinement of usual elliptic correctors: for all $n\ge0$, the $n$th-order hyperbolic corrector~$\phi^{n,0}$ and homogenized tensor~$\bar\Aa^{n,0}$ defined in Section~\ref{sec:structure} indeed coincide with their elliptic counterparts~\cite{BLP-78,Gu-17,DO1}.

As in the elliptic setting, hyperbolic correctors have a natural geometric interpretation.
We focus on the periodic case to simplify the presentation.
The first corrector $\phi^{1,0}$, which is the same as in the elliptic setting, is defined to correct Euclidean coordinates $x\mapsto x_i$ into $\Aa$-harmonic ones $x\mapsto x_i+\phi_i^{1,0}(x)$: indeed, $\phi_i^{1,0}$ is the unique periodic solution of
\[-\nabla\cdot\Aa(\nabla\phi_i^{1,0}+\ee_i)=0,\]
with $\expec{\phi^{1,0}}=0$.
The corresponding two-scale expansion $H^1_\e[\bar v]\,:=\,\bar v+\e\phi_i^{1,0}(\frac\cdot\e)\nabla_i\bar v$ is then viewed as an intrinsic Taylor expansion of the limiting profile $\bar v$ in terms of $\Aa$-harmonic coordinates.
In order to describe oscillations with finer accuracy, higher-order correctors are iteratively defined to correct higher-order polynomials and make them adapted to the heterogeneous wave operator. More precisely, higher-order correctors $\{\phi^{n,m}\}_{n,m}$ are defined in such a way that, for any polynomial ${\bar p}$ in space-time variables $x,t$, the corrected polynomial
\begin{equation}\label{eq:correct-pol}
H^\infty[{\bar p}]\,:=\,\sum_{n,m\ge0}\phi^{n,m}\odot\nabla^n\partial_t^m{\bar p}
\end{equation}
captures oscillations of the heterogeneous wave operator in the sense that
\[(\partial_t^2-\nabla\cdot\Aa\nabla)H^\infty[{\bar p}]\]
has no periodic oscillations any longer. In that case, this quantity can automatically be written as
\begin{equation}\label{eq:describe-osc-op}
(\partial_t^2-\nabla\cdot\Aa\nabla)H^\infty[{\bar p}]\,=\,\bigg(\partial_t^2-\nabla\cdot\Big(\sum_{n\ge1}\sum_{m\ge0}\bar\Aa^{n,m}\odot\nabla^{n-1}\partial_t^m\Big)\nabla\bigg){\bar p},
\end{equation}
for some suitable family $\{\bar\Aa^{n,m}\}_{n,m}$ of constant tensors; see Proposition~\ref{prop:form-hyperb}. This reflects the fact that on large scales the heterogeneous constitutive relation $\nabla u\mapsto\Aa\nabla u$ is replaced by the effective relation $\nabla\bar u\mapsto\bar\Aa\nabla\bar u$ at leading order, while additional corrective terms must be included when looking for finer accuracy,
\begin{equation}\label{eq:const-law-ref}
\nabla\bar u~~\mapsto~~\Big(\sum_{n\ge1}\sum_{m\ge0}\bar\Aa^{n,m}\odot\nabla^{n-1}\partial_t^m\Big)\nabla\bar u.
\end{equation}
The difference with the elliptic setting is that in the present hyperbolic setting both space and time variables need to be corrected alike.
Comparing to the spectral approach, note that the presence of mixed space-time derivatives in the resulting homogenized equation~\eqref{eq:form-hom} leads to additional well-posedness issues: naïve notions of solution display a secular growth, which was first avoided in~\cite{ALR} as explained at the end of Section~\ref{sec:discuss}.

\subsection{Ill-prepared data}\label{sec:ill-posed}
Up to now, we have focused on well-prepared data, or equivalently, on ancient solutions of the heterogeneous wave equation~\eqref{eq:hyperb}. 
We now briefly discuss the effect of ill-prepared data by means of Floquet--Bloch theory, which provides further insight on the claims of Remark~\ref{rem:ill-prepared}.
For that purpose, as in Proposition~\ref{prop:Bloch}, we assume that the coefficient field~$\Aa$ is periodic and that the impulse $f$ has compactly supported spatial Fourier transform: in this setting, for $\e$ small enough, the operator $\Lc_{\e\xi}$ has discrete spectrum and its lowest eigenvalue~$\lambda_{\e\xi}$ is simple for all $\xi$ in the Fourier support of~$f$.
We then show that, if initial data do not fit spatial oscillations of the ground state,
their projection on higher modes propagates and maintains forever, leading to an~$O(\e)$ contribution that consists of a superposition of typically incommensurate time oscillations with~$O(\e^{-1})$ frequency. This almost-periodic structure prohibits any approximate description by a two-scale expansion beyond accuracy~$O(\e)$.
More precisely, we consider the initial-value problem
\begin{gather*}
\left\{\begin{array}{ll}
\big(\partial_t^2-\nabla\cdot\Aa(\tfrac\cdot\e+q)\nabla\big)z_\e(\cdot,q)=0,&\text{in $\R^+\times\R^d$},\\
z_\e(\cdot,q)|_{t=0}=u^\circ,&\text{in $\R^d$},\\
\partial_tz_\e(\cdot,q)|_{t=0}=v^\circ,&\text{in $\R^d$}.
\end{array}\right.
\end{gather*}
By Floquet--Bloch theory, arguing as for~\eqref{eq:dec-ueps-0}, we now get
\begin{multline*}
z_\e(x,q)\\
=\,\int_{\R^d}e^{ix\cdot\xi}(\pi_{\e\xi}1)(\tfrac x\e+q)\bigg(\cos\big(t(\tfrac1{\e^2}\lambda_{\e\xi})^{1/2}\big)\,\hat u^\circ(\xi)+\frac{\sin\big(t(\tfrac1{\e^2}\lambda_{\e\xi})^{1/2}\big)}{(\tfrac1{\e^2}\lambda_{\e\xi})^{1/2}}\,\hat v^\circ(\xi)\bigg)\,d^*\xi\\
+r_\e(x,q),
\end{multline*}
where the remainder $r_{\e}$ contains all the effects of initial ill-preparedness,
\begin{multline*}
r_\e(x,q)\,:=\,\int_{\R^d}e^{ix\cdot\xi}\Big(\cos\big(t(\tfrac1{\e^2}\Lc_{\e\xi})^{1/2}\big)\pi_{\e\xi}^\bot1\Big)(\tfrac x\e+q)\,\hat u^\circ(\xi)\,d^*\xi\\
+\int_{\R^d}e^{ix\cdot\xi}\bigg(\frac{\sin\big(t(\tfrac1{\e^2}\Lc_{\e\xi})^{1/2}\big)}{(\tfrac1{\e^2}\Lc_{\e\xi})^{1/2}}\pi_{\e\xi}^\bot1\bigg)(\tfrac x\e+q)\,\hat v^\circ(\xi)\,d^*\xi.
\end{multline*}
Denoting by $\{\nu_{\e\xi}^n\}_{n\ge0}$ the non-decreasing sequence of eigenvalues of $\Lc_{\e\xi}$ on $\Ld^2(Q)$ (repeated according to multiplicity, with $\nu_{\e\xi}^1>\nu_{\e\xi}^0=\lambda_{\e\xi}$), and denoting by $\{\gamma_{\e\xi}^n\}_{n\ge0}$ a corresponding sequence of normalized eigenfunctions, the above remainder can be written as
\begin{multline*}
r_\e(x,q)\\
=\,\sum_{n=1}^\infty\int_{\R^d}e^{ix\cdot\xi}\kappa^n_{\e\xi}\gamma_{\e\xi}^n(\tfrac x\e+q)
\bigg(\cos\big(\tfrac t\e(\nu_{\e\xi}^n)^{1/2}\big)\hat u^\circ(\xi)
+\frac{\e\sin\big(\frac t\e(\nu^n_{\e\xi})^{1/2}\big)}{(\nu^n_{\e\xi})^{1/2}}\,\hat v^\circ(\xi)\bigg)\,d^*\xi,
\end{multline*}
in terms of $\kappa^n_{\e\xi}:=\expecm{\gamma_{\e\xi}^n}=O(\e\xi)$.
As claimed, this shows that the remainder $r_{\e}$ is of order $O(\e)$ and oscillates both in space and time with~$O(\e^{-1})$ frequency. Moreover, at a fixed Fourier mode $\xi$, the time dependence is (typically) almost-periodic, which prohibits any approximate description by means of a two-scale expansion beyond accuracy $O(\e)$.

As explained in Remark~\ref{rem:ill-prepared}, these complicated oscillations are naturally removed by taking time averages. In terms of the above remainder $r_\e$, this amounts to noting that, given~$\theta\in C^\infty_c(\R)$, we formally have for all $\lambda>0$,
\begin{eqnarray*}
\int_0^\infty\theta(t-s)\cos(\tfrac s\e\lambda^{1/2})\,ds
&=&\e^2\sum_{k=0}^\infty\e^{2k}{(-1)^{k+1}}\lambda^{-k-1}\theta^{(2k+1)}(t),\\
\int_0^\infty\theta(t-s)\frac{\e\sin(\tfrac s\e\lambda^{1/2})}{\lambda^{1/2}}\,ds
&=&\e^2\sum_{k=0}^\infty\e^{2k}{(-1)^{k+1}}\lambda^{-k-1}\theta^{(2k)}(t),
\end{eqnarray*}
so that the above flow decomposition is precisely turned into an expansion of the form~\eqref{eq:full-decomp-per-u} for the time-averaged flow~\eqref{eq:time-average-form}.

\section{Spectral approach and two-scale expansion}\label{sec:spectral}
This section is devoted to the definition of spectral correctors and to the proof of Theorems~\ref{th:main-per} and~\ref{th:main2} and of Corollary~\ref{cor:summable}, including the well-posedness of the formal homogenized equation~\eqref{eq:goal2}.

\subsection{Definition of spectral correctors} \label{sec:def-corr-spec}
We start by recalling the definition of the spectral correctors $\{\psi^n\}_n$ and homogenized tensors $\{\bar \bb^n\}_n$, as first introduced in \cite[Definition~2.1]{BG} and motivated in Section~\ref{sec:Bloch}, cf.~\eqref{e.spec-intro1}. We further introduce Fourier multipliers~$\{\gamma_{\ell}\}_\ell$, which are proxies for the factor $\expec{|\psi_{\e\xi}|^2}^{-1}$ in~\eqref{eq:link-w-psi}, \eqref{eq:expand-pibot1}, and~\eqref{eq:expand-Psixieps}.

\begin{defin}[Spectral correctors]\label{def:spec-corr-psi}
For all $\xi\in\R^d$, we define $\{\check\psi^n_\xi,\check\lambda^n_\xi\}_{n\ge0}$ inductively via $\check\psi^0_\xi=1$, $\check\lambda^0_\xi=0$, and for all $n\ge1$ we define $\check\psi^n_\xi\in H^1_\per(Q)$ as the periodic scalar field that has vanishing average $\expecm{\check\psi^n_\xi}=0$ and satisfies
\begin{equation}\label{e.spec-1}
-\nabla\cdot\Aa\nabla\check\psi^{n}_\xi=\nabla\cdot(\Aa i\xi\check\psi^{n-1}_\xi)+i\xi\cdot\Aa(\nabla\check\psi^{n-1}_\xi+i\xi\check\psi^{n-2}_\xi)+\sum_{k=2}^n\check\lambda_\xi^k\check\psi_\xi^{n-k},
\end{equation}
where we have defined
\begin{equation}\label{e.spec-2}
\check \lambda^n_\xi=-\expec{i\xi\cdot\Aa(\nabla\check\psi^{n-1}_\xi+ i\xi\check\psi^{n-2}_\xi)},
\end{equation}
recalling that we implicitly set $\check\psi_\xi^{n}\equiv0$ for $n<0$ for notational convenience.
{Factoring out powers of $i\xi$ in the above,} we may then define the real-valued symmetric tensor fields~$\{\psi^n\}_n$ and symmetric tensors $\{\bar\bb^{n}\}_n$ via
\begin{equation}\label{eq:decomp-check-rep}
\psi^n \odot (i\xi)^{\otimes n}=  \check \psi^n_\xi,\qquad \xi\cdot (\bar \bb^{n} \odot (i\xi)^{\otimes(n-1)})\xi=  \check \lambda^{n+1}_\xi.
\end{equation}
We shall also use the notation $\bar \Aa:=\bar \bb^1$ as the latter coincides with the homogenized coefficient for the associated elliptic equation.
Next, for all $\ell\ge 1$, we can define the following Fourier multiplier,
\begin{equation}\label{e.def:gamma}
\gamma_{\ell}(\xi)\,:=\, \expecm{\textstyle|\sum_{n=0}^{\ell} \psi^n\odot (i\xi)^{\otimes n}|^2}^{-1}.\qedhere
\end{equation}
\end{defin}

\begin{rems}$ $\nopagebreak
\begin{enumerate}[$\bullet$]
\item As the choice~\eqref{e.spec-2} precisely ensures that the right-hand side of~\eqref{e.spec-1} has vanishing average, an iterative use of the Poincaré inequality on the unit cell $Q$ easily yields the following estimates{:} for all $n$,
\begin{equation}\label{e.bd-coeff}
|\bar \bb^n|+\|\psi^n\|_{H^1(Q)} \le C^n.
\end{equation}
This ensures in particular the well-posedness of equations~\eqref{e.spec-1} \&~\eqref{e.spec-2}.
\smallskip\item Since by definition we have
\[\expecm{\textstyle\sum_{n=0}^{\ell} \psi^n\odot (i\xi)^{\otimes n}}=1,\]
Jensen's inequality yields
\[\textstyle\expecm{|\sum_{n=0}^{\ell} \psi^n\odot (i\xi)^{\otimes n}|^2}\,\ge\,1,\]
which ensures that the definition~\eqref{e.def:gamma} of the Fourier multiplier $\gamma_\ell(\xi)$ indeed makes sense and satisfies $\gamma_{\ell}(\xi) \le1$ for all $\xi$. In addition, in view of~\eqref{e.bd-coeff}, we can expand, for $|\xi|\ll1$ small enough,
\begin{equation}\label{e.expansion-gamma}
\gamma_\ell(\xi)=1+\sum_{k=2}^\infty \gamma_{\ell}^k \odot \xi^{\otimes k},
\end{equation}
for some coefficients $\{\gamma_{\ell}^k\}_k$. In addition, this can be truncated to any order $n\ge0$,
\begin{equation}\label{eq:expand-gam}
\Big|\gamma_\ell(\xi)-\sum_{k=0}^n\gamma_\ell^k\odot\xi^{\otimes k}\Big|\,\le\,(C|\xi|)^{n+1}.\qedhere
\end{equation}
\item {The uniform ellipticity of $\Aa$, cf.~\eqref{eq:unif-ell}, ensures that $\bar\Aa=\bar\bb^1$ is elliptic, and therefore we have $\check\lambda_\xi^2=\xi\cdot\bar\Aa\xi\ge\lambda|\xi|^2$ for all $\xi\in\R^d$.}
\end{enumerate}
\end{rems}

As a consequence of the self-adjointness of fibered operators $\{\Lc_\xi\}_\xi$ in Section~\ref{sec:Bloch},
their first eigenvalues $\{\lambda_\xi\}_\xi$ are real, hence, in view of~\eqref{eq:def-psi-expand} and~\eqref{eq:decomp-check-rep}, we deduce that $\bar\bb^n$ must vanish for all $n$ even. Equivalently, $\check\lambda_\xi^n$ vanishes for $n$ odd. This can alternatively be proven by a direct computation starting from definition~\eqref{e.spec-2}, cf.~\cite[Proposition~1]{BG}; a similar (more involved) argument will be provided in Proposition~\ref{prop:even-coeff}, so we skip the detail for now.

\begin{prop}\label{prop:vanish}
We have $\bar\bb^{n}=0$ for all $n$ even.
\end{prop}

As is common in the elliptic setting, it is useful to further introduce suitable flux correctors, which will allow us to refine error estimates by directly exploiting cancellations due to fluxes having vanishing average.

\begin{defin}[Spectral flux correctors and auxiliary correctors]\label{def:spec-fluc-corr-psi}
For $n\ge0$, we define the spectral flux correctors $\sigma^n:=(\sigma^n_{i_1\ldots i_n})_{1\le i_1,\ldots,i_n\le d}$  by
\[\sigma_{i_1\ldots i_n}^n=(\nabla\Phi_{i_1\ldots i_n}^n)^T,\]
where $\Phi^n_{i_1\ldots i_n}\in H^1_\per(Q)^d$ is the periodic vector field that satisfies $\expecm{\Phi_{i_1\ldots i_n}^n}=0$ and
\[-\triangle (\Phi_{i_1\ldots i_n}^n)_{i_{n+1}}\,=\,e_{i_{n+1}}\cdot \Aa(\nabla\psi^{n}_{i_1\ldots i_{n}}+\psi^{n-1}_{i_1\ldots i_{n-1}}\ee_{i_n})
-\sum_{k=2}^{n+1}(\ee_{i_{k}}\cdot\bar\bb^{k-1}_{i_1\ldots i_{k-2}}\ee_{i_{k-1}})\psi^{n+1-k}_{i_{k+1}\ldots i_{n+1}}.
\]
For $n\ge2$, we also define the auxiliary correctors $\rho^n:=(\rho^n_{i_1\ldots i_n})_{1\le i_1,\ldots,i_n\le d}$ by
\[\rho^n_{i_1\ldots i_n}=\nabla_{i_n}\Psi^n_{i_1\ldots i_{n-1}},\]
where $\Psi^n_{i_1\ldots i_{n-1}}\in H^1_\per(Q)$ is the periodic vector field that satisfies $\expecm{\Psi^n_{i_1\ldots i_{n-1}}}=0$ and
\begin{equation*}
-\triangle\Psi^{n}_{i_1\ldots i_{n-1}}=\psi^{n-1}_{i_1\ldots i_{n-1}}.\qedhere
\end{equation*}
\end{defin}

Next, we turn to the definition of the correctors $\{\zeta^{n,m}\}_{n,m}$, as motivated in Section~\ref{sec:Bloch}, cf.~\eqref{e.spec-intro3} and~\eqref{e.spec-intro4re}, and we start with the case $m=0$.

\begin{lem}\label{lem:spec-corr-zeta-1}
For all $n\ge0$ and $\xi\in \R^d$, we recursively define $\check\zeta^{n,0}_\xi\in H^1_{{\per}}(Q)$ as the unique periodic scalar field that satisfies
\begin{multline}\label{e.zeta0n}
-\nabla\cdot\Aa\nabla\check \zeta^{n,0}_\xi\,=\,\nabla\cdot(\Aa i\xi\check\zeta^{n-1,0}_\xi)
+i\xi\cdot\Aa\big(\nabla\check\zeta^{n-1,0}_\xi+i\xi\check\zeta^{n-2,0}_\xi\big)\\
-\check\psi^{n+1}_\xi+\sum_{k=0}^{n+1}\expec{\overline{\check\psi^{n+1-k}_\xi}\check\psi^k_\xi},
\end{multline}
with integration constant
\begin{equation}\label{eq:expec-zeta}
\expecm{\check \zeta^{n,0}_\xi}\,:=\,-\sum_{k=1}^{n}\sum_{l=2}^{k+2}(\overline{\check \lambda_\xi^l/\check \lambda_\xi^2})\,\expec{\overline{\check\psi^{k+2-l}_\xi}\check\zeta^{n-k,0}_\xi}.
\end{equation}
With this choice of the constant, the above equation for $\check\zeta^{n,0}_\xi$ is indeed well-posed by the Fredholm alternative as it iteratively ensures that the right-hand side has vanishing average,
\begin{equation}\label{eq:chocolate}
\expec{i\xi\cdot\Aa\big(\nabla\check\zeta^{n-1,0}_\xi+i\xi\check\zeta^{n-2,0}_\xi\big)}+\sum_{k=0}^{n+1}\expec{\overline{\check\psi^{n+1-k}_\xi}\check\psi^k_\xi}=0.
\end{equation}
{Factoring out powers of $i\xi$,} we may then define the real-valued symmetric tensor field~$\zeta^{n,0}$ such that
\[\zeta^{n,0} \odot (i\xi)^{\otimes (n+1)}\,:=\, \check\zeta^{n,0}_{\xi}.\qedhere\]
\end{lem}

\begin{proof}
It suffices to show by induction that~\eqref{eq:chocolate} holds for all $n\ge0$.
For $k\ge0$, testing the defining equation~\eqref{e.spec-1} for $\check\psi^{k+1}_\xi$ with $\check\zeta^{n,0}_\xi$, we get after averaging and integrating by parts,
\begin{multline*}
\expec{\overline{\check \psi^{k+1}_\xi} (-\nabla \cdot \Aa \nabla \check\zeta^{n,0}_\xi)}\\
=\, \expec{\overline{\check\psi^k_\xi} \Big(\nabla \cdot (\Aa i \xi \check\zeta_\xi^{n,0})+i \xi \cdot \Aa\nabla \check\zeta_\xi^{n,0}\Big) +\overline{\check \psi_\xi^{k-1}}(i\xi \cdot\Aa  i\xi)\check \zeta_\xi^{n,0}}
+\sum_{l=2}^{k+1}\overline{\check \lambda_\xi^l}\,\expec{\overline{\check \psi_\xi^{k+1-l}}\check\zeta_\xi^{n,0}}.
\end{multline*}
Now inserting the defining equation~\eqref{e.zeta0n} for $\check\zeta_\xi^{n,0}$, and reorganizing the terms, this becomes
\begin{multline*}
\expec{\overline{\check \psi_\xi^{k+1}} \Big(\nabla\cdot(\Aa i\xi\check\zeta_\xi^{n-1,0})+i\xi\cdot\Aa(\nabla\check\zeta_\xi^{n-1,0}+i\xi\check\zeta_\xi^{n-2,0})\Big)
+\overline{\check\psi_\xi^k} (i \xi \cdot \Aa i\xi)\check\zeta_\xi^{n-1,0} }\\
\,=\, \expec{\overline{\check\psi_\xi^k} \Big(\nabla \cdot (\Aa i \xi \check\zeta_\xi^{n,0})+ i \xi \cdot \Aa(\nabla \check\zeta_\xi^{n,0}+i\xi\check\zeta_\xi^{n-1,0})\Big)
+\overline{\check \psi_\xi^{k-1}}(i\xi \cdot\Aa  i\xi)\check \zeta_\xi^{n,0}}\\
+\expec{\overline{\check \psi_\xi^{k+1}}\check\psi_\xi^{n+1}}
+\sum_{l=2}^{k+1}\overline{\check \lambda_\xi^l}\,\expec{ \overline{\check \psi_\xi^{k+1-l}}\check\zeta_\xi^{n,0}}.
\end{multline*}
Iterating this identity $(n+1)$-times, starting from $k=0$, and using that $\psi^0=1$ and that $\expec{\psi^m}=0$ for all $m\geq 1$, we deduce after straightforward simplifications
\begin{equation*}
\expec{i\xi \cdot \Aa(\nabla \check\zeta_\xi^{n,0}+i\xi\check\zeta_\xi^{n-1,0})}+\sum_{k=0}^{n+2}\expec{\overline{\check \psi_\xi^{n+2-k}}\check\psi_\xi^{k}}
\,=\,-\sum_{k=2}^{n+1}\sum_{l=2}^{k}\overline{\check \lambda_\xi^l}\,\expec{ \overline{\check \psi_\xi^{k-l}}\check\zeta_\xi^{n+1-k,0}},
\end{equation*}
or equivalently,
\begin{multline*}
\expec{i\xi \cdot \Aa(\nabla \check\zeta_\xi^{n,0}+i\xi\check\zeta_\xi^{n-1,0})}+\sum_{k=0}^{n+2}\expec{\overline{\check \psi_\xi^{n+2-k}}\check\psi_\xi^{k}}\\
\,=\,{-}\overline{\check \lambda_\xi^2}\,\expecm{\check\zeta_\xi^{n-1,0}}
{-}\sum_{k=1}^{n-1}\sum_{l=2}^{k+2}\overline{\check \lambda_\xi^l}\,\expec{ \overline{\check \psi_\xi^{k+2-l}}\check\zeta_\xi^{n-k-1,0}}.
\end{multline*}
Now the choice~\eqref{eq:expec-zeta} of the integration constant for $\check\zeta_\xi^{n-1,0}$ precisely ensures that the right-hand side vanishes, which proves that the identity~\eqref{eq:chocolate} holds with $n$ replaced by $n+1$.
\end{proof}

Next, we turn to the construction of corresponding correctors $\{\zeta^{n,m}\}_{n}$ for $m\ge1$, as motivated in~\eqref{e.spec-intro4re}. The proof of this lemma is analogous to the one above and is skipped for shortness.

\begin{lem}\label{lem:spec-corr-zeta-2}
Given $m\ge1$, for all $n\ge0$ and $\xi\in\R^d$, we recursively define $\check\zeta^{n,m}_\xi\in H^1_{{\per}}(Q)$ as the unique periodic scalar field that satisfies
\begin{equation}\label{e.zetajn}
-\nabla\cdot\Aa\nabla\check\zeta^{n,m}_\xi=
\nabla\cdot(\Aa i\xi\check\zeta^{n-1,m}_\xi)+
i\xi\cdot\Aa(\nabla\check\zeta^{n-1,m}_\xi+i\xi\check\zeta^{n-2,m}_\xi)+\check\zeta^{n,m-1}_\xi,
\end{equation}
with integration constant
\[\expecm{\check\zeta^{n,m}_\xi}\,:=\,(\overline{1/\check\lambda_\xi^2})\sum_{k=0}^{n+2}\expec{\overline{\check\psi^{n+2-k}_\xi}\check\zeta^{k,m-1}_\xi}-\sum_{k=1}^{n}\sum_{l=2}^{k+2}(\overline{\check\lambda_\xi^l/\check\lambda_\xi^2})\,\expec{\overline{\check\psi^{k+2-l}_\xi}\check\zeta^{n-k,m}_\xi}.\]
With this choice of the integration constants, the above equation for $\check\zeta^{n,m+1}_\xi$ is indeed well-posed by the Fredholm alternative as it iteratively ensures that the right-hand side has vanishing average,
\begin{equation}\label{eq:chocolate-2}
\expec{i\xi\cdot\Aa\big(\nabla\check\zeta^{n-1,m}_\xi+i\xi\check\zeta^{n-2,m}_\xi\big)+\check\zeta^{n,m-1}_\xi}\,=\,0.
\end{equation}
In particular, $\expecm{\check\zeta_\xi^{0,m}}=0$ for all $m\ge0$.
{Factoring out powers of $i\xi$,} we may then define the real-valued symmetric tensor field $\zeta^{n,m}$ such that
\[\zeta^{n,m} \odot (i\xi)^{\otimes (n+1)}\,:=\,\check\zeta^{n,m}_{\xi}.\qedhere\]
\end{lem}

Again, it is useful to further introduce associated flux correctors, which allow {us} to refine error estimates by directly exploiting cancellations due to fluxes having vanishing average.

\begin{defin}[Spectral flux correctors]\label{def:spec-fluc-corr-zeta}
For $n\ge0$, we define the spectral flux corrector $\tau^{n,0}:=(\tau^{n,0}_{i_1\ldots i_{n+1}})_{1\le i_1,\ldots,i_{n+1}\le d}$ by
\[\tau_{i_1\ldots i_{n+1}}^{n,0}=(\nabla\Phi_{i_1\ldots i_{n+1}}^{n,0})^T,\]
where $\Phi^{n,0}_{i_1\ldots i_{n+1}}\in H^1_\per(Q)^d$ is the periodic vector field that satisfies $\expecm{\Phi_{i_1\ldots i_{n+1}}^{n,0}}=0$ and
\[-\triangle\Phi^{n,0}_{i_1\ldots i_{n+1}}=\Aa\big(\nabla\zeta^{n,0}_{i_1\ldots i_{n+1}}+\zeta^{n-1,0}_{i_1\ldots i_{n}}\ee_{i_{n+1}}\big)-\expec{\Aa\big(\nabla\zeta^{n,0}_{i_1\ldots i_{n+1}}+\zeta^{n-1,0}_{i_1\ldots i_{n}}\ee_{i_{n+1}}\big)}.\]
We define the spectral flux corrector $\tau^{n,m}:=(\tau^{n,m}_{i_1\ldots i_{n+1}})_{1\le i_1,\ldots,i_{n+1}\le d}$ for $n\ge-1$ and $m\ge1$ by
\[\tau_{i_1\ldots i_{n+1}}^{n,m}=(\nabla\Phi_{i_1\ldots i_{n+1}}^{n,m})^T,\]
where $\Phi^{n,m}_{i_1\ldots i_{n+1}}\in H^1_\per(Q)^d$ is the periodic vector field that satisfies $\expecm{\Phi_{i_1\ldots i_{n+1}}^{n,m}}=0$ and
\begin{equation*}
-\triangle\Phi^{n,m}_{i_1\ldots i_{n+1}}=\Aa\big(\nabla\zeta^{n,m}_{i_1\ldots i_{n+1}}+\zeta^{n-1,m}_{i_1\ldots i_{n}}\ee_{i_{n+1}}\big)+\zeta_{i_1\ldots i_{n+2}}^{n+1,m-1}\ee_{i_{n+2}}.
\end{equation*}
(Note that the definition of the $\zeta^{n,m}$'s ensures that the expression in the right-hand side has vanishing average.)
\end{defin}

Based on the above definitions, as for~\eqref{e.bd-coeff}, an iterative use of the Poincaré inequality on the unit cell $Q$ easily yields the following estimates.

\begin{lem}\label{lem:cor-est-per}
For all $m,n\ge0$,
\begin{equation}\label{eq:cor-est-per}
\|(\psi^n,\sigma^n,\rho^n)\|_{H^1(Q)} +|\bar \bb^n|\le C^n,\qquad \|(\zeta^{n,m},\tau^{n,m})\|_{H^1(Q)}\lesssim C^{n+m}.\qedhere
\end{equation}
\end{lem}

\subsection{Spectral two-scale expansion}
Given a smooth function $\bar w$, we consider its two-scale expansion of order $\ell\ge 0$ associated with the above-defined spectral correctors, as motivated by spectral considerations in Section~\ref{sec:Bloch},
\begin{multline}\label{eq:2s-spec}
S_\e^\ell[\bar w,f]\,:=\,\sum_{n=0}^\ell \e^n  \psi^n(\tfrac\cdot\e)\odot  \gamma_{\ell}(\e\nabla)\nabla^n\bar w\\
+\e^3  \sum_{2m=0}^{\ell-3}(-1)^m\e^{2m} \sum_{n=0}^{\ell-3-2m} \e^n   \zeta^{n,m}(\tfrac\cdot\e)\odot  \gamma_{\ell} (\e \nabla)    \nabla^{n+1}\partial_t^{2m} f.
\end{multline}
We show by PDE techniques that this expansion is indeed well-adapted to describe the local behavior of the solution to the hyperbolic equation in the following sense: the heterogeneous hyperbolic operator applied to $S^\ell_\e[\bar w,f]$ is equivalent to a higher-order effective operator applied to $\bar w$, up to error terms of formal order~$O(\e^\ell)$. The proof is postponed to Section~\ref{sec:pr-prop:form-spectral}.

\begingroup\allowdisplaybreaks
\begin{prop}[Spectral two-scale expansion]\label{prop:form-spectral}
Let $\ell\ge1$, $\e>0$, and let $\bar w,f$ be smooth functions satisfying
\begin{equation}\label{eq:def-f-effab}
\partial_t^2\bar w -\nabla\cdot\Big(\bar\Aa+\sum_{k=2}^\ell\bar\bb^k\odot(\e\nabla)^{k-1}\Big)\nabla\bar w\,=\,f.
\end{equation}
Then, the associated spectral two-scale expansion of order $\ell$, defined in~\eqref{eq:2s-spec}, satisfies the following relation in the distributional sense in $\R\times\R^d$,
\begin{align*}
&(\partial_t^2-\nabla \cdot \Aa \nabla) S_\e^\ell[\bar w,f]
\,=\,f\\
&\qquad-\sum_{n=\ell}^{2\ell}\e^n\sum_{k=(n-\ell)\vee1}^{\ell\wedge(n-1)}(-1)^{n-k}\expec{\psi^{n-k}_{i_{k+1}\ldots i_{n}}\psi^k_{i_{1}\ldots i_k}}\gamma_\ell(\e\nabla)\nabla^{n}_{i_1\ldots i_{n}}f\\
&\qquad-\e^\ell\nabla\cdot\Big(\rho^\ell_{i_1\ldots i_\ell}(\tfrac\cdot\e)e_{i_\ell}\gamma_\ell(\e\nabla)\nabla^{\ell-1}_{i_1\ldots i_{\ell-1}}f\Big)\\
&\qquad-\e^\ell\nabla \cdot \Big((\Aa \psi^\ell_{i_1\ldots i_\ell}-\sigma^\ell_{i_1\ldots i_\ell})(\tfrac\cdot\e)\,\gamma_\ell(\e\nabla)\,\nabla\nabla^{\ell}_{i_1\ldots i_{\ell}}\bar w\Big)\\
&\qquad+\e^\ell(\psi^\ell+\rho^\ell)(\tfrac\cdot\e)\odot\gamma_\ell(\e\nabla) \nabla^\ell f
-\e^\ell\sigma^\ell_{i_1\ldots i_{\ell}}(\tfrac\cdot\e): \gamma_\ell(\e\nabla) \nabla^2\nabla^{\ell}_{i_1\ldots i_{\ell}} \bar w
\\
&\qquad+\sum_{n=1}^\ell \sum_{k=\ell+2-n}^{\ell+1}\e^{n+k-2} \psi^n_{i_1\ldots i_n}(\tfrac\cdot\e) \bar \bb^{k-1}_{i_{n+1}\ldots i_{n+k-2}}:\gamma_\ell(\e\nabla)\nabla^2\nabla_{i_1\ldots i_{n+k-2}}^{n+k-2}\bar w\\
&\qquad-\e^\ell\sum_{2m=0}^{\ell-2}(-1)^m\nabla\cdot\Big(\big(\Aa\zeta^{\ell-3-2m,m}_{i_1\ldots i_{\ell-2-2m}}-\tau^{\ell-3-2m,m}_{i_1\ldots i_{\ell-2-2m}}\big)(\tfrac\cdot\e)\gamma_{\ell}(\e\nabla)\nabla\nabla^{\ell-2-2m}_{i_1\ldots i_{\ell-2-2m}}\partial_t^{2m} f\Big)\\
&\qquad-\e^\ell\sum_{2m=0}^{\ell-2}(-1)^m\tau^{\ell-3-2m,m}_{i_1\ldots i_{\ell-2-2m}}(\tfrac\cdot\e):\gamma_{\ell}(\e\nabla)\nabla^2\nabla^{\ell-2-2m}_{i_1\ldots i_{\ell-2-2m}}\partial_t^{2m} f\\
&\qquad+\e^\ell\partial_t^2\sum_{2m=0}^{\ell-3}(-1)^m\zeta^{\ell-3-2m,m}(\tfrac\cdot\e)\odot\gamma_{\ell}(\e\nabla)\nabla^{\ell-2-2m}\partial_t^{2m} f.\qedhere
\end{align*}
\end{prop}
\endgroup

\subsection{Well-posedness of homogenized equation}\label{sec:homog-eqn/spec}
As motivated in Proposition~\ref{prop:form-spectral}, cf.~\eqref{eq:def-f-effab},
we consider the following formal homogenized equation, for $\ell\ge1$,
\begin{gather}\label{eq:homog-lim-sp}
\left\{\begin{array}{ll}
\partial_t^2\bar U^\ell_\e-\nabla\cdot\big(\bar\Aa+\sum_{k=2}^\ell\bar\bb^k\odot(\e\nabla)^{k-1}\big)\nabla \bar U^\ell_\e = f,&\text{in $\R\times\R^d$},\\
\bar U_\e^\ell=f=0,&\text{for $t<0$}.
\end{array}\right.
\end{gather}
However, as explained in the introduction, the symbol of the operator
\begin{equation}\label{eq:diffop-sp}
-\nabla\cdot\Big(\bar\Aa+\sum_{k=2}^\ell\bar\bb^k\odot(\e\nabla)^{k-1}\Big)\nabla
\end{equation}
may vanish, so equation~\eqref{eq:homog-lim-sp} is ill-posed in general.
As described in Section~\ref{sec:discuss}, several high{er}-order modifications of this equation can then be used to ensure well-posedness: high-frequency filtering as in~\cite{ALR}, high{er}-order regularization as in~\cite{BG},
or the Boussinesq trick as in~\cite{Pouch-19}. Let us precisely define each of them:
\begingroup\allowdisplaybreaks
\begin{enumerate}[\hspace{-0.15cm}(I)]
\item \emph{High-frequency filtering}.\\
Let $\alpha\in(0,1)$, and let $\chi\in C^\infty_c(\R^d)$ be a cut-off function with
\[\chi|_{\frac12B}=1,\qquad\chi|_{\R^d\setminus B}=0.\]
Provided that $0<\e\ll_{\alpha}1$ is small enough, the Fourier symbol of the operator~\eqref{eq:diffop-sp} is strictly positive on $\e^{-\alpha} B$,
and we may then define $\bar u_\e^{\operatorname{(I)},\ell}$ as the unique solution in~$\R\times\R^d$ of
\begin{equation}\label{eq:proxy-baru-I0}
\partial_t^2\bar u_\e^{\operatorname{(I)},\ell}-\nabla\cdot\Big(\bar\Aa+\sum_{k=2}^\ell\bar\bb^{k}\odot(\e\nabla)^{n-1}\Big)\nabla\bar u_\e^{\operatorname{(I)},\ell}
= \chi(\e^\alpha\nabla)f,
\end{equation}
with $\bar u_\e^{\operatorname{(I)},\ell}=f=0$ for $t<0$, such that the spatial Fourier transform of $\bar u_\e^{\operatorname{(I)},\ell}$ is supported in $\R\times\e^{-\alpha}B$.
\smallskip\item \emph{High{er}-order regularization}.\\
Choose $\kappa_\ell\ge0$ as the smallest real number such that for all $\xi\in\R^d$,
\begin{equation}\label{eq:choice-kappaell}
\qquad\xi \cdot\Big(\bar\Aa+\sum_{k=2}^\ell \bar\bb^{k}\odot(i \xi)^{\otimes(k-1)}+\kappa_\ell |\xi|^{\ell}\Big)\xi \,\ge\, \tfrac12 \lambda |\xi|^2.
\end{equation}
Note that~\eqref{e.bd-coeff} entails $\kappa_\ell \le C^\ell$.
Then, for all $\e>0$, we can define $\bar u_\e^{\operatorname{(II)},\ell}$ as the unique solution in $\R\times\R^d$ of
\begin{equation}\label{eq:proxy-baru-II}
\qquad\quad
\partial_t^2\bar u_\e^{\operatorname{(II)},\ell}-\nabla\cdot\Big(\bar\Aa+\sum_{k=2}^\ell\bar\bb^{k}\odot(\e\nabla)^{k-1}+\kappa_\ell (\e|\nabla|)^{\ell}\Big)\nabla\bar u_\e^{\operatorname{(II)},\ell}
\,=\, f,
\end{equation}
with $\bar u_\e^{\operatorname{(II)},\ell}=f=0$ for $t<0$.
\smallskip\item \emph{Boussinesq trick}.\\
Set $\kappa_1=1$, $\kappa_{2j}=0$ for all $j$, and for $j>0$ we define inductively $\kappa_{2j+1}\ge0$ as the smallest value such that for all $\xi \in \R^d$,
\begin{equation}\label{eq:choice-kappaell-B}
\qquad\xi\cdot\Big(\kappa_{2j+1}\bar\Aa\,+\,\sum_{l=1}^{2j} \kappa_{l}\,\bar \bb^{2j+2-l}\odot(i\tfrac{\xi}{|\xi|})^{2j+1-l}\Big)\xi\,\ge\,0.
\end{equation}
Note that~\eqref{e.bd-coeff} entails $|\kappa_l| \le C^l$ for all $l$.
Then, for all $\e>0$, we can define $\bar u_\e^{\text{(III)},\ell}$ as the unique solution in~$\R\times\R^d$ of
\begin{multline}\label{eq:proxy-baru-III}
\quad\qquad\partial_t^2\Big(1+\sum_{l=2}^{\ell}  \kappa_l (\e |\nabla|)^{l-1}\Big)\bar u_\e^{\operatorname{(III)},\ell}\\
-\nabla\cdot\bigg(\sum_{n=1}^\ell \Big(
 \kappa_{n}\bar\Aa+ \sum_{l=1}^{n-1} \kappa_{l}\bar \bb^{n+1-l}\odot(\tfrac{\nabla}{|\nabla|})^{n-l}\Big)(\e|\nabla|)^{n-1}\bigg) \nabla\bar u_\e^{\operatorname{(III)},\ell}\\
\,=\, \Big(1+\sum_{l=2}^{\ell}    \kappa_l (\e |\nabla|)^{l-1}\Big)f,
\end{multline}
with $\bar u_\e^{\operatorname{(III)},\ell}=f=0$ for $t<0$.
\end{enumerate}
We analyze these three modifications of the formal homogenized equation~\eqref{eq:homog-lim-sp} and show that they are well-posed and all equivalent up to higher-order errors.
The proof is postponed to Section~\ref{sec:pr-lem:apriori-baru-sp}.
\endgroup

\begin{lem}[Well-posedness of effective equation]\label{lem:apriori-baru-sp}
Let $f\in C^\infty(\R;H^\infty(\R^d))$ and let $\ell\ge1$.
\begin{enumerate}[(i)]
\item If the spatial Fourier transform of $f$ is supported in $\R\times B_R$ for some $R\ge1$, and provided that $\e R\ll1$ is small enough (independently of~$\ell$), then the formal effective equation~\eqref{eq:homog-lim-sp} admits a unique ancient solution $\bar U_\e^\ell\in\Ld^\infty_\loc(\R;\Ld^2(\R^d))$ with spatial Fourier transform supported in $\R\times B_R$.
Moreover, it satisfies for all $r,t\ge0$,
\begin{eqnarray*}
\|\langle D\rangle^rD\bar U_\e^{\ell;t}\|_{\Ld^2(\R^d)}&\lesssim&\|\langle D\rangle^rf\|_{\Ld^1((0,t),\Ld^2(\R^d))},\\
\|\bar U_\e^{\ell;t}\|_{\Ld^2(\R^d)}&\lesssim&\langle t\rangle\|f\|_{\Ld^1((0,t);\Ld^2(\R^d))}.
\end{eqnarray*}
\item The modified equations~\eqref{eq:proxy-baru-I0}, \eqref{eq:proxy-baru-II}, and~\eqref{eq:proxy-baru-III} are well-posed in $\Ld^\infty_\loc(\R;\Ld^2(\R^d))$ in their respective sense, and their solutions satisfy for all $r\ge0$, for $(\star)=\operatorname{(I)}$ or $\operatorname{(II)}$,
\begin{eqnarray*}
\|\langle D\rangle^rD\bar u_\e^{{(\star)},\ell;t}\|_{\Ld^2(\R^d)}&\lesssim&\|\langle D\rangle^rf\|_{\Ld^1((0,t);\Ld^2(\R^d))},
\\
\|\bar u_\e^{(\star),\ell;t}\|_{\Ld^2(\R^d)}&\lesssim&\langle t\rangle\|f\|_{\Ld^1((0,t);\Ld^2(\R^d))},
\end{eqnarray*}
and for $(\star)=\operatorname{(III)}$,
\begin{eqnarray*}
\|\langle D\rangle^rD\bar u_\e^{\operatorname{(III)},\ell;t}\|_{\Ld^2(\R^d)}&\le& C^{\ell}\|\langle D\rangle^r\langle\e\nabla\rangle^{\lfloor\frac{\ell-1}2\rfloor}f\|_{\Ld^1((0,t);\Ld^2(\R^d))},\\
\|\bar u_\e^{\operatorname{(III)},\ell;t}\|_{\Ld^2(\R^d)}&\le & C^{\ell}\langle t\rangle\|\langle \e\nabla\rangle^{\lfloor\frac{\ell-1}2\rfloor}f\|_{\Ld^1((0,t);\Ld^2(\R^d))}.
\end{eqnarray*}
\item If the spatial Fourier transform of $f$ is supported in $\R\times B_R$ for some $R\ge1$, and provided that $\e R\ll1$ is small enough (independently of~$\ell$), we have for all $r,t\ge0$,
\begin{align*}
\|\langle D\rangle^r(\bar u_\e^{\operatorname{(I)},\ell;t}-\bar U_\e^{\ell;t})\|_{\Ld^2(\R^d)}&\,\le\,(\e C)^{\ell}\|\langle D\rangle^{K_{\alpha}\ell+r}f\|_{\Ld^1((0,t);\Ld^2(\R^d))},
\\
\|\langle D\rangle^r(\bar u_\e^{\operatorname{(II)},\ell;t}-\bar U_\e^{\ell;t})\|_{\Ld^2(\R^d)}&\,\le\,(\e C)^{\ell}\langle t\rangle   \|\langle D\rangle^{\ell+r}f\|_{\Ld^1((0,t);\Ld^2(\R^d))} ,
\\
\|\langle D\rangle^r(\bar u_\e^{\operatorname{(III)},\ell;t}-\bar U_\e^{\ell;t})\|_{\Ld^2(\R^d)}&\,\le\,(\e C)^{\ell}\langle t\rangle  \|\langle D \rangle^{2\ell+r-2}f\|_{\Ld^1((0,t);\Ld^2(\R^d)},
\end{align*}
with {$K_{\alpha}\le1/\alpha$} in the first estimate.
\qedhere
\end{enumerate}
\end{lem}

\begingroup\allowdisplaybreaks
\subsection{Proof of Proposition~\ref{prop:form-spectral}}\label{sec:pr-prop:form-spectral}
By scaling, it suffices to consider the case $\e=1$ and we omit the subscript $\e=1$ for notational convenience. The two-scale expansion~\eqref{eq:2s-spec} can be decomposed as $S^\ell[\bar w,f]=S_1^{\ell}[\bar w]+S_2^{\ell}[f]$, in terms of
\begin{eqnarray*}
S_1^{\ell}[\bar w]&:=&\sum_{n=0}^\ell\psi^n \odot  \gamma_{\ell}(\nabla)\nabla^n\bar w,\\
S_2^{\ell}[f]&:=&\sum_{2m=0}^{\ell-3}\sum_{n=0}^{\ell-3-2m}(-1)^m\zeta^{n,m}\odot\gamma_{\ell}(\nabla)    \nabla^{n+1}\partial_t^{2m} f.
\end{eqnarray*}
We split the proof into three steps, separately deriving equations for each part.

\medskip
\step1 Equation for $S_1^{\ell}[\bar w]$: {we show that}
\begin{multline}\label{eq:rep-eqn-S1ell}
(\partial_t^2-\nabla \cdot \Aa \nabla) S_1^\ell[\bar w]
\,=\,\sum_{n=0}^\ell \psi^n\odot  \gamma_\ell(\nabla) \nabla^n\bigg(\partial_t^2\bar w-\nabla\cdot\Big(\bar\Aa+\sum_{k=2}^{\ell} \bar \bb^{k}\odot \nabla^{k-1}\Big)\nabla\bar w\bigg)\\
- \nabla \cdot \Big((\Aa \psi^\ell_{i_1\ldots i_\ell}-\sigma^\ell_{i_1\ldots i_\ell})\,\gamma_\ell(\nabla)\,\nabla\nabla^{\ell}_{i_1\ldots i_{\ell}}\bar w\Big)
-\sigma^\ell_{i_1\ldots i_{\ell}}: \gamma_\ell(\nabla) \nabla^2\nabla^{\ell}_{i_1\ldots i_{\ell}} \bar w
\\
+\sum_{n=1}^\ell \sum_{k=\ell+2-n}^{\ell+1} \psi^n_{i_1\ldots i_n} \bar \bb^{k-1}_{i_{n+1}\ldots i_{n+k-2}}:\gamma_\ell(\nabla)\nabla^2\nabla_{i_1\ldots i_{n+k-2}}^{n+k-2}\bar w.
\end{multline}
A direct calculation {based on the general formula $-\nabla\cdot \Aa\nabla(hg)=(-\nabla\cdot\Aa\nabla h)g-\nabla\cdot(ah)\cdot\nabla g-a\nabla h\cdot\nabla g-ha:\nabla^2g$} yields, for all $n\ge0$,
\begin{multline*}
-\nabla \cdot \Aa \nabla \big( \psi^n \odot \gamma_\ell(\nabla) \nabla^n \bar w\big)
\\
=\, (-\nabla \cdot \Aa \nabla \psi^n_{i_1\ldots i_n}) \gamma_\ell(\nabla) \nabla^n_{i_1\ldots i_n} \bar w
-\nabla \cdot (\Aa \psi^n_{i_1\ldots i_n} e_{i_{n+1}}) \gamma_\ell(\nabla) \nabla^{n+1}_{i_1\ldots i_{n+1}} \bar w
\\
-(\ee_{i_{n+1}}\cdot\Aa \nabla \psi^n_{i_1\ldots i_n}) \gamma_\ell(\nabla) \nabla^{n+1}_{i_1\ldots i_{n+1}} \bar w- (\ee_{i_{n+2}}\cdot\Aa \psi^n_{i_1\ldots i_n} e_{i_{n+1}})\gamma_\ell(\nabla) \nabla^{n+2}_{i_1\ldots i_{n+2}} \bar w.
\end{multline*}
Combined with the defining equation~\eqref{e.spec-1} for $\psi^n$, this entails
\begin{multline*}
-\nabla \cdot \Aa \nabla \big( \psi^n \odot \gamma_\ell(\nabla) \nabla^n \bar w\big)
\\
=\, \nabla\cdot(\Aa\psi^{n-1}_{i_1\ldots i_{n-1}}\ee_{i_n}) \gamma_\ell(\nabla) \nabla^n_{i_1\ldots i_n} \bar w
-\nabla \cdot (\Aa \psi^n_{i_1\ldots i_n} e_{i_{n+1}}) \gamma_\ell(\nabla) \nabla^{n+1}_{i_1\ldots i_{n+1}} \bar w\\
+(\ee_{i_n}\cdot\Aa\nabla\psi^{n-1}_{i_1\ldots i_{n-1}})\gamma_\ell(\nabla) \nabla^n_{i_1\ldots i_n} \bar w
-(\ee_{i_{n+1}}\cdot\Aa \nabla \psi^n_{i_1\ldots i_n}) \gamma_\ell(\nabla) \nabla^{n+1}_{i_1\ldots i_{n+1}} \bar w\\
+(\ee_{i_n}\cdot\Aa\psi^{n-2}_{i_1\ldots i_{n-2}}\ee_{i_{n-1}})\gamma_\ell(\nabla) \nabla^n_{i_1\ldots i_n} \bar w
-(\ee_{i_{n+2}}\cdot\Aa \psi^n_{i_1\ldots i_n} e_{i_{n+1}})\gamma_\ell(\nabla) \nabla^{n+2}_{i_1\ldots i_{n+2}} \bar w\\
-\sum_{k=2}^{n}(\ee_{i_{k}}\cdot\bar\bb^{k-1}_{i_1\ldots i_{k-2}}\ee_{i_{k-1}})\,\psi^{n-k}_{i_{k+1}\ldots i_n}\gamma_\ell(\nabla) \nabla^n_{i_1\ldots i_n} \bar w,
\end{multline*}
and thus, after summation over $0\le n\le\ell$, {taking into account the telescoping sum we obtain that}
\begin{multline*}
(\partial_t^2-\nabla \cdot \Aa \nabla) S_1^\ell[\bar w]\,=\,
\sum_{n=0}^{\ell} \psi^n_{i_1\ldots i_n} \gamma_\ell(\nabla)\nabla^n_{i_1\ldots i_n} \Big(\partial_t^2-\sum_{k=2}^{\ell-n} \bar \bb^{k-1}_{j_1\ldots j_{k-2}}:\nabla^2\nabla^{k-2}_{j_1\ldots j_{k-2}}\Big) \bar w
\\
- \nabla \cdot (\Aa \psi^\ell_{i_1\ldots i_\ell}\ee_{i_{\ell+1}}) \gamma_\ell(\nabla) \nabla^{\ell+1}_{i_1\ldots i_{\ell+1}}\bar w
- \ee_{i_{\ell+1}}\cdot\Aa(\nabla \psi^\ell_{i_1\ldots i_\ell}+\psi^{\ell-1}_{i_1\ldots i_{\ell-1}} e_{i_\ell})\gamma_\ell(\nabla) \nabla^{\ell+1}_{i_1\ldots i_{\ell+1}} \bar w\\
-(\ee_{i_{\ell+2}}\cdot\Aa \psi^\ell_{i_1\ldots i_{\ell}} e_{i_{\ell+1}}) \gamma_\ell(\nabla) \nabla^{\ell+2}_{i_1\ldots i_{\ell+2}} \bar w,
\end{multline*}
which we can rewrite as
\begin{multline*}
(\partial_t^2-\nabla \cdot \Aa \nabla) S_1^\ell[\bar w]\,=\,
\sum_{n=0}^{\ell} \psi^n_{i_1\ldots i_n} \gamma_\ell(\nabla)\nabla^n_{i_1\ldots i_n} \Big(\partial_t^2-\sum_{k=2}^{\ell+1-n}\bar \bb^{k-1}_{j_1\ldots j_{k-2}}:\nabla^2\nabla^{k-2}_{j_1\ldots j_{k-2}}\Big) \bar w
\\
- \nabla \cdot (\Aa \psi^\ell_{i_1\ldots i_\ell}\ee_{i_{\ell+1}}) \gamma_\ell(\nabla) \nabla^{\ell+1}_{i_1\ldots i_{\ell+1}}\bar w -(\ee_{i_{\ell+2}}\cdot\Aa \psi^\ell_{i_1\ldots i_{\ell}} e_{i_{\ell+1}}) \gamma_\ell(\nabla) \nabla^{\ell+2}_{i_1\ldots i_{\ell+2}} \bar w\\
- \Big(\ee_{i_{\ell+1}}\cdot\Aa(\nabla \psi^\ell_{i_1\ldots i_\ell}+\psi^{\ell-1}_{i_1\ldots i_{\ell-1}} e_{i_\ell}\big)-\sum_{k=2}^{\ell+1}(\ee_{i_{k}}\cdot\bar\bb^{k-1}_{i_1\ldots i_{k-2}}\ee_{i_{k-1}})\psi^{\ell+1-k}_{i_{k+1}\ldots i_{\ell+1}} \Big) \gamma_\ell(\nabla)\nabla^{\ell+1}_{i_1\ldots i_{\ell+1}} \bar w.
\end{multline*}
Recalling the definition of flux correctors, cf.~Definition~\ref{def:spec-fluc-corr-psi}, this means
\begin{multline*}
(\partial_t^2-\nabla \cdot \Aa \nabla) S_1^\ell[\bar w]\,=\,
\sum_{n=0}^\ell \psi^n_{i_1\ldots i_n} \gamma_\ell(\nabla)\nabla^n_{i_1\ldots i_n}\Big(\partial_t^2-\sum_{k=2}^{\ell+1-n} \bar \bb^{k-1}_{j_1\ldots j_{k-2}}:\nabla^2\nabla^{k-2}_{j_1\ldots j_{k-2}}\Big) \bar w\\
- \nabla \cdot \big((\Aa \psi^\ell_{i_1\ldots i_\ell}-\sigma^\ell_{i_1\ldots i_\ell})\ee_{i_{\ell+1}}\big) \gamma_\ell(\nabla) \nabla^{\ell+1}_{i_1\ldots i_{\ell+1}}\bar w
\\
-(\ee_{i_{\ell+2}}\cdot\Aa \psi^\ell_{i_1\ldots i_{\ell}} e_{i_{\ell+1}}) \gamma_\ell(\nabla) \nabla^{\ell+2}_{i_1\ldots i_{\ell+2}} \bar w,
\end{multline*}
and the claim~\eqref{eq:rep-eqn-S1ell} follows.

\medskip
\step 2 Equation for $S_2^{\ell}[f]$: {we show that}
\begin{multline}\label{eq:rep-eqn-S2ell}
(\partial_t^2-\nabla\cdot\Aa\nabla) S_2^\ell[f]
\,=\,-\sum_{n=1}^{\ell-2}\psi^{n}\odot\gamma_{\ell}(\nabla)\nabla^{n} f\\
-\sum_{2m=0}^{\ell-2}(-1)^m\nabla\cdot\Big(\big(\Aa\zeta^{\ell-3-2m,m}_{i_1\ldots i_{\ell-2-2m}}-\tau^{\ell-3-2m,m}_{i_1\ldots i_{\ell-2-2m}}\big)\gamma_{\ell}(\nabla)\nabla\nabla^{\ell-2-2m}_{i_1\ldots i_{\ell-2-2m}}\partial_t^{2m} f\Big)\\
-\sum_{2m=0}^{\ell-2}(-1)^m\tau^{\ell-3-2m,m}_{i_1\ldots i_{\ell-2-2m}}:\gamma_{\ell}(\nabla)\nabla^2\nabla^{\ell-2-2m}_{i_1\ldots i_{\ell-2-2m}}\partial_t^{2m} f\\
+\sum_{n=0}^{\ell-2}\sum_{k=0}^{n+1}(-1)^{n+1-k}\expec{\psi^{n+1-k}_{i_{k+1}\ldots i_{n+1}}\psi^k_{i_{1}\ldots i_k}}\gamma_{\ell}(\nabla)\nabla^{n+1}_{i_1\ldots i_{n+1}} f\\
+\partial_t^2\sum_{2m=0}^{\ell-3}(-1)^m\zeta^{\ell-3-2m,m}\odot\gamma_{\ell}(\nabla)\nabla^{\ell-2-2m}\partial_t^{2m} f.
\end{multline}
 A direct calculation yields for all $m,n$,
\begin{multline*}
-\nabla\cdot\Aa\nabla\big(\zeta^{n,m} \odot  \gamma_{\ell}(\nabla)\nabla^{n+1}\partial_t^{2m} f\big)\\
\,=\,(-\nabla\cdot\Aa\nabla\zeta^{n,m}_{i_1\ldots i_{n+1}})  \gamma_{\ell}(\nabla)\nabla^{n+1}_{i_1\ldots i_{n+1}}\partial_t^{2m} f
-\nabla\cdot(\Aa\zeta^{n,m}_{i_1\ldots i_{n+1}}\ee_{i_{n+2}})  \gamma_{\ell}(\nabla)\nabla^{n+2}_{i_1\ldots i_{n+2}}\partial_t^{2m} f\\
-(\ee_{i_{n+2}}\cdot\Aa\nabla\zeta^{n,m}_{i_1\ldots i_{n+1}})  \gamma_{\ell}(\nabla)\nabla^{n+2}_{i_1\ldots i_{n+2}}\partial_t^{2m} f
-(\ee_{i_{n+3}}\cdot\Aa\zeta^{n,m}_{i_1\ldots i_{n+1}}\ee_{i_{n+2}})  \gamma_{\ell}(\nabla)\nabla^{n+3}_{i_1\ldots i_{n+3}}\partial_t^{2m} f.
\end{multline*}
For $m=0$, inserting the defining equation for $\zeta^{n,0}$, cf.~\eqref{e.zeta0n}, this entails
\begin{multline*}
-\nabla\cdot\Aa\nabla\big(\zeta^{n,0} \odot  \gamma_{\ell}(\nabla)\nabla^{n+1} f\big)\\
\,=\,\nabla\cdot\big(\Aa\zeta^{n-1,0}_{i_1\ldots i_n}\ee_{i_{n+1}}\big)\gamma_{\ell}(\nabla)\nabla^{n+1}_{i_1\ldots i_{n+1}} f
-\nabla\cdot\big(\Aa\zeta^{n,0}_{i_1\ldots i_{n+1}}\ee_{i_{n+2}}\big)\gamma_{\ell}(\nabla)\nabla^{n+2}_{i_1\ldots i_{n+2}} f\\
+\big(\ee_{i_{n+1}}\cdot\Aa\nabla\zeta^{n-1,0}_{i_1\ldots i_n}\big)\gamma_{\ell}(\nabla)\nabla^{n+1}_{i_1\ldots i_{n+1}} f
-\big(\ee_{i_{n+2}}\cdot\Aa\nabla\zeta^{n,0}_{i_1\ldots i_{n+1}}\big)  \gamma_{\ell}(\nabla)\nabla^{n+2}_{i_1\ldots i_{n+2}} f\\
+\big(\ee_{i_{n+1}}\cdot\Aa\zeta^{n-2,0}_{i_1\ldots i_{n-1}}\ee_{i_n}\big)\gamma_{\ell}(\nabla)\nabla^{n+1}_{i_1\ldots i_{n+1}} f
-\big(\ee_{i_{n+3}}\cdot\Aa\zeta^{n,0}_{i_1\ldots i_{n+1}}\ee_{i_{n+2}}\big)\gamma_{\ell}(\nabla)\nabla^{n+3}_{i_1\ldots i_{n+3}} f\\
-\psi^{n+1}_{i_1\ldots i_{n+1}}\gamma_{\ell}(\nabla)\nabla^{n+1}_{i_1\ldots i_{n+1}} f
+\sum_{k=0}^{n+1}(-1)^{n+1-k}\,\expec{\psi^{n+1-k}_{i_{k+1}\ldots i_{n+1}}\psi^k_{i_{1}\ldots i_k}}\gamma_{\ell}(\nabla)\nabla^{n+1}_{i_1\ldots i_{n+1}} f,
\end{multline*}
and thus, after summation over $0\le n\le\ell-3$,
\begin{multline*}
-\nabla\cdot\Aa\nabla\sum_{n=0}^{\ell-3}\zeta^{n,0} \odot  \gamma_{\ell}(\nabla)\nabla^{n+1} f
\,=\,-\nabla\cdot\big(\Aa\zeta^{\ell-3,0}_{i_1\ldots i_{\ell-2}}\ee_{i_{\ell-1}}\big)\gamma_{\ell}(\nabla)\nabla^{\ell-1}_{i_1\ldots i_{\ell-1}} f\\
-\ee_{i_{\ell-1}}\cdot\Aa\big(\nabla\zeta^{\ell-3,0}_{i_1\ldots i_{\ell-2}}+\zeta^{\ell-4,0}_{i_1\ldots i_{\ell-3}}\ee_{i_{\ell-2}}\big)\gamma_{\ell}(\nabla)\nabla^{\ell-1}_{i_1\ldots i_{\ell-1}} f\\
-\big(\ee_{i_{\ell}}\cdot\Aa\zeta^{\ell-3,0}_{i_1\ldots i_{\ell-2}}\ee_{i_{\ell-1}}\big)\gamma_{\ell}(\nabla)\nabla^{\ell}_{i_1\ldots i_{\ell}} f\\
-\sum_{n=1}^{\ell-2}\psi^{n}_{i_1\ldots i_{n}}\gamma_{\ell}(\nabla)\nabla^{n}_{i_1\ldots i_{n}} f
+\sum_{n=0}^{\ell-3}\sum_{k=0}^{n+1}(-1)^{n+1-k}\,\expec{\psi^{n+1-k}_{i_{k+1}\ldots i_{n+1}}\psi^k_{i_{1}\ldots i_k}}\gamma_{\ell}(\nabla)\nabla^{n+1}_{i_1\ldots i_{n+1}} f.
\end{multline*}
Proceeding similarly for $m>0$, we find
\begin{multline*}
-\nabla\cdot\Aa\nabla\sum_{n=0}^{\ell-3-2m}\zeta^{n,m} \odot  \gamma_{\ell}(\nabla)\nabla^{n+1}\partial_t^{2m} f\\
\,=\,
-\nabla\cdot\big(\Aa\zeta^{\ell-3-2m,m}_{i_1\ldots i_{\ell-2-2m}}\ee_{i_{\ell-1-2m}}\big)\gamma_{\ell}(\nabla)\nabla^{\ell-1-2m}_{i_1\ldots i_{\ell-1-2m}}\partial_t^{2m} f\\
-\ee_{i_{\ell-1-2m}}\cdot\Aa\big(\nabla\zeta^{\ell-3-2m,m}_{i_1\ldots i_{\ell-2-2m}}+\zeta^{\ell-4-2m,m}_{i_1\ldots i_{\ell-3-2m}}\ee_{i_{\ell-2-2m}}\big)\gamma_{\ell}(\nabla)\nabla^{\ell-1-2m}_{i_1\ldots i_{\ell-1-2m}}\partial_t^{2m} f\\
-\big(\ee_{i_{\ell-2m}}\cdot\Aa\zeta^{\ell-3-2m,m}_{i_1\ldots i_{\ell-2-2m}}\ee_{i_{\ell-1-2m}}\big)\gamma_{\ell}(\nabla)\nabla^{\ell-2m}_{i_1\ldots i_{\ell-2m}}\partial_t^{2m} f\\
+\sum_{n=0}^{\ell-3-2m}\zeta^{n,m-1}_{i_1\ldots i_{n+1}}\gamma_{\ell}(\nabla)\nabla^{n+1}_{i_1\ldots i_{n+1}}\partial_t^{2m} f.
\end{multline*}
Combining the above two identities, we are led to
\begin{multline*}
-\nabla\cdot\Aa\nabla S_2^\ell[f]
\,=\,-\sum_{2m=0}^{\ell-3}(-1)^m\nabla\cdot\big(\Aa\zeta^{\ell-3-2m,m}_{i_1\ldots i_{\ell-2-2m}}\ee_{i_{\ell-1-2m}}\big)\gamma_{\ell}(\nabla)\nabla^{\ell-1-2m}_{i_1\ldots i_{\ell-1-2m}}\partial_t^{2m} f\\
-\sum_{2m=0}^{\ell-3}(-1)^m\ee_{i_{\ell-1-2m}}\cdot\Aa\big(\nabla\zeta^{\ell-3-2m,m}_{i_1\ldots i_{\ell-2-2m}}+\zeta^{\ell-4-2m,m}_{i_1\ldots i_{\ell-3-2m}}\ee_{i_{\ell-2-2m}}\big)\gamma_{\ell}(\nabla)\nabla^{\ell-1-2m}_{i_1\ldots i_{\ell-1-2m}}\partial_t^{2m} f\\
-\sum_{2m=0}^{\ell-3}(-1)^m\big(\ee_{i_{\ell-2m}}\cdot\Aa\zeta^{\ell-3-2m,m}_{i_1\ldots i_{\ell-2-2m}}\ee_{i_{\ell-1-2m}}\big)\gamma_{\ell}(\nabla)\nabla^{\ell-2m}_{i_1\ldots i_{\ell-2m}}\partial_t^{2m} f\\
-\sum_{n=1}^{\ell-2}\psi^{n}_{i_1\ldots i_{n}}\gamma_{\ell}(\nabla)\nabla^{n}_{i_1\ldots i_{n}} f
+\sum_{n=0}^{\ell-3}\sum_{k=0}^{n+1}(-1)^{n+1-k}\,\expec{\psi^{n+1-k}_{i_{k+1}\ldots i_{n+1}}\psi^k_{i_{1}\ldots i_k}}\gamma_{\ell}(\nabla)\nabla^{n+1}_{i_1\ldots i_{n+1}} f\\
+\sum_{2m=2}^{\ell-3}(-1)^m\sum_{n=0}^{\ell-3-2m}\zeta^{n,m-1}_{i_1\ldots i_{n+1}}\gamma_{\ell}(\nabla)\nabla^{n+1}_{i_1\ldots i_{n+1}}\partial_t^{2m} f.
\end{multline*}
Recalling the definition of flux correctors, cf.~Definition~\ref{def:spec-fluc-corr-zeta}, as well as~\eqref{eq:chocolate} and~\eqref{eq:chocolate-2}, and reorganizing the terms, the claim~\eqref{eq:rep-eqn-S2ell} follows. Note that the flux correctors $\tau^{n,m}$'s are nontrivial even for $n=-1$, but those appear only in the case when $\ell-3$ is odd.
 
\medskip
\step3 Conclusion.\\
Combining the results of the last two steps, reorganizing the terms, and recalling the relation~\eqref{eq:def-f-effab} between $\bar w,f$, we are led to
\begin{multline}\label{eq:almost-decomp-eqn-S}
(\partial_t^2-\nabla \cdot \Aa \nabla) S^\ell[\bar w,f]
\,=\,f\\
-\bigg(\gamma_\ell(\nabla)^{-1}-1-\sum_{n=0}^{\ell-2}\sum_{k=0}^{n+1}(-1)^{n+1-k}\expec{\psi^{n+1-k}_{i_{k+1}\ldots i_{n+1}}\psi^k_{i_{1}\ldots i_k}}\nabla^{n+1}_{i_1\ldots i_{n+1}}\bigg)\gamma_{\ell}(\nabla)f\\
- \nabla \cdot \Big((\Aa \psi^\ell_{i_1\ldots i_\ell}-\sigma^\ell_{i_1\ldots i_\ell})\,\gamma_\ell(\nabla)\,\nabla\nabla^{\ell}_{i_1\ldots i_{\ell}}\bar w\Big)\\
+\sum_{n=(\ell-1)\vee1}^\ell \psi^n_{i_1\ldots i_n}  \gamma_\ell(\nabla) \nabla^n_{i_1\ldots i_n}f-\sigma^\ell_{i_1\ldots i_{\ell}}: \gamma_\ell(\nabla) \nabla^2\nabla^{\ell}_{i_1\ldots i_{\ell}} \bar w
\\
+\sum_{n=1}^\ell \sum_{k=\ell+2-n}^{\ell+1} \psi^n_{i_1\ldots i_n} \bar \bb^{k-1}_{i_{n+1}\ldots i_{n+k-2}}:\gamma_\ell(\nabla)\nabla^2\nabla_{i_1\ldots i_{n+k-2}}^{n+k-2}\bar w\\
-\sum_{2m=0}^{\ell-2}(-1)^m\nabla\cdot\Big(\big(\Aa\zeta^{\ell-3-2m,m}_{i_1\ldots i_{\ell-2-2m}}-\tau^{\ell-3-2m,m}_{i_1\ldots i_{\ell-2-2m}}\big)\gamma_{\ell}(\nabla)\nabla\nabla^{\ell-2-2m}_{i_1\ldots i_{\ell-2-2m}}\partial_t^{2m} f\Big)\\
-\sum_{2m=0}^{\ell-2}(-1)^m\tau^{\ell-3-2m,m}_{i_1\ldots i_{\ell-2-2m}}:\gamma_{\ell}(\nabla)\nabla^2\nabla^{\ell-2-2m}_{i_1\ldots i_{\ell-2-2m}}\partial_t^{2m} f\\
+\partial_t^2\sum_{2m=0}^{\ell-3}(-1)^m\zeta^{\ell-3-2m,m}\odot\gamma_{\ell}(\nabla)\nabla^{\ell-2-2m}\partial_t^{2m} f,
\end{multline}
and it remains to reformulate the second and fourth right-hand side terms.
For the second right-hand side term, we recall the definition~\eqref{e.def:gamma} of $\gamma_\ell$, which yields
\begin{eqnarray*}
\gamma_\ell(\xi)^{-1}-1-\sum_{n=0}^{\ell-2}\sum_{k=0}^{n+1}\expec{\overline{\check\psi_\xi^{n+1-k}}\check\psi_\xi^k}
&=&\E\bigg[{\Big|\sum_{n=0}^{\ell} \check\psi_\xi^n\Big|^2}\bigg]-1-\sum_{n=0}^{\ell-2}\sum_{k=0}^{n+1}\expec{\overline{\check\psi_\xi^{n+1-k}}\check\psi_\xi^k}\\
&=&\sum_{n=\ell}^{2\ell}\sum_{k=n-\ell}^{\ell}\expec{\overline{\check\psi_\xi^{n-k}}\check\psi_\xi^k}.
\end{eqnarray*}
For the fourth right-hand side term in~\eqref{eq:almost-decomp-eqn-S}, we use the auxiliary correctors of Definition~\ref{def:spec-fluc-corr-psi} to write for $\ell\ge2$,
\[\psi^{\ell-1}_{i_1\ldots i_{\ell-1}}  \gamma_\ell(\nabla) \nabla^{\ell-1}_{i_1\ldots i_{\ell-1}}f\,=\,
\rho^{\ell}_{i_1\ldots i_{\ell}}\gamma_\ell(\nabla) \nabla^{\ell}_{i_1\ldots i_{\ell}}f
-\nabla\cdot\Big(\rho^{\ell}_{i_1\ldots i_{\ell}}e_{i_\ell}\gamma_\ell(\nabla) \nabla^{\ell-1}_{i_1\ldots i_{\ell-1}}f\Big).\]
Combining these identities yields the conclusion.\qed

\subsection{Proof of Lemma~\ref{lem:apriori-baru-sp}}\label{sec:pr-lem:apriori-baru-sp}
We split the proof into five steps.

\medskip
\step1 Proof of~(i): well-posedness provided $\supp\hat f\subset \R\times B_R$.\\
In Fourier space, the operator $-\nabla\cdot(\bar\Aa+\sum_{k=2}^\ell \bar\bb^{k}\odot(\e\nabla)^{k-1})\nabla$ has symbol
\begin{equation*}
	\mu^{\ell}_{\e}(\xi)\,:=\,\xi\cdot \Big(\bar\Aa+\sum_{k=2}^{\ell}\bar{\bb}^k\odot(i\e \xi)^{\otimes (k-1)}\Big)\xi.
\end{equation*}
By Proposition~\ref{prop:vanish}, we know that $\mu^{\ell}_{\e}$ is real-valued. Recalling that the uniform ellipticity condition~\eqref{eq:unif-ell} entails
$\lambda|\xi|^2\le \xi\cdot\bar\Aa\xi\le|\xi|^2$ after homogenization,
and taking advantage of~\eqref{e.bd-coeff}, we find for $|\xi|\leq R$,
\begin{equation*}
	\mu^{\ell}_{\e}(\xi)\,\ge\, |\xi|^2\Big(\lambda-\sum_{k=2}^{\ell}(\e C|\xi|)^{k-1}\Big)
	\,\ge\, |\xi|^2\Big(\lambda-\sum_{k=2}^{\ell}(\e CR)^{k-1}\Big).
\end{equation*}
Provided that $\e R\ll1$ is small enough, we deduce for $|\xi|\le R$,
\begin{equation}\label{eq:coercivityinFourier}
	\tfrac{1}{C}|\xi|^2\,\le\,\mu^{\ell}_{\e}(\xi)\,\le\, C|\xi|^2.
\end{equation}
We may thus define a solution of~\eqref{eq:homog-lim-sp} in Fourier space via Duhamel's formula, that is,
\begin{equation*}
	\F[\bar U_\e^{\ell;t}](\xi)\,:=\,\int_{0}^t \frac{\sin\big((t-s)\mu^{\ell}_{\e}(\xi)^{1/2}\big)}{\mu^{\ell}_{\e}(\xi)^{1/2}}\hat f^s(\xi)\,ds,
\end{equation*}
where $\F g=\hat g$ stands for {the} spatial Fourier transform.
This formula indeed satisfies in $\R\times\R^d$,
\begin{equation}\label{eq:homog-eqn-Fourier-(i)}
	\partial_t^2\F[\bar U_\e^{\ell}]+\mu^{\ell}_{\e}\F[\bar U_\e^{\ell}]\,=\,\hat f,
\end{equation}
and thus, upon inverse Fourier transformation, this provides a weak solution $\bar U_\e^\ell$ of~\eqref{eq:homog-lim-sp} in $\Ld^\infty_\loc(\R;\Ld^2(\R^d))$. In addition, by construction, the spatial Fourier transform is supported in $\R\times B_R$ as $\hat f$ is.

\medskip\noindent
We turn to the proof of a priori estimates. As the equation is linear and has constant coefficients, derivatives of the solution satisfy the same equation up to replacing $f$ by its corresponding derivatives. It is therefore enough 
to prove the stated estimates with $r=0$,
\begin{eqnarray}
\|D\bar U_\e^{\ell;t}\|_{\Ld^2(\R^d)}&\lesssim&\|f\|_{\Ld^1((0,t),\Ld^2(\R^d))},\label{eq:DU-apriori}\\
\|\bar U_\e^{\ell;t}\|_{\Ld^2(\R^d)}&\lesssim&\langle t\rangle\|f\|_{\Ld^1((0,t);\Ld^2(\R^d))}.\label{eq:U-apriori}
\end{eqnarray}
Moreover, we note that the $\Ld^2$-estimate~\eqref{eq:U-apriori} directly follows from~\eqref{eq:DU-apriori}: recalling that $D=(\partial_t,\nabla)$, we indeed get from~\eqref{eq:DU-apriori} and integration that
\begin{equation*}
\|\bar U_{\e}^{\ell;t}\|_{\Ld^2(\R^d)}\,\le\, \int_{0}^t\|\partial_t \bar U_{\e}^{\ell}\|_{\Ld^2(\R^d)}\,\lesssim\, \int_{0}^t\|f\|_{\Ld^1((0,s);\Ld^2(\R^d))}\,ds\,\le\, \langle t\rangle \|f\|_{\Ld^1((0,t);\Ld^2(\R^d))},
\end{equation*}
that is,~\eqref{eq:U-apriori}.
It remains to establish~\eqref{eq:DU-apriori}. For that purpose, multiplying both sides of equation~\eqref{eq:homog-eqn-Fourier-(i)} by {the complex conjugate of} $\partial_t\F[\bar U_\e^\ell]$ {and taking the real part}, we get
\begin{equation*}
	\tfrac{1}{2}\partial_t\int_{\R^d}\Big(|\partial_t\F[\bar U_\e^\ell]|^2+\mu^{\ell}_{\e}|\F[\bar U_\e^\ell]|^2\Big)\,\le\, \|\hat f^t\|_{\Ld^2(\R^d)}\|\partial_t\F[\bar U_\e^\ell]\|_{\Ld^2(\R^d)},
\end{equation*}
which implies
\begin{equation*}
	\partial_t\left(\int_{\R^d}\Big(|\partial_t\F[\bar U_\e^\ell]|^2+\mu^{\ell}_{\e}|\F[\bar U_\e^\ell]|^2\Big)\right)^{\frac{1}{2}}\,\lesssim\, \|f^t\|_{\Ld^2(\R^d)}.
\end{equation*}
Integrating in time and appealing to~\eqref{eq:coercivityinFourier}, this yields the claim~\eqref{eq:DU-apriori}.
Note that uniqueness follows from these a priori estimates by linearity.

\medskip\noindent
For Step 5, we shall also need an a priori estimate for $\bar U_\e^\ell$ in terms of the \mbox{$\dot{H}^{-1}$-norm} of the impulse. Multiplying~\eqref{eq:homog-eqn-Fourier-(i)} with the complex conjugate of $(|\cdot|^2+\delta)^{-1}\partial_t\F[\bar U_\e^\ell]$ and repeating the above argument, we infer that
\begin{equation*}
\partial_t\bigg(\int_{\R^d}(|\xi|^2+\delta)^{-1}\mu_{\e}^{\ell}(\xi)\,|\F[\bar U_\e^\ell](\xi)|^2\,d\xi\bigg)^{\frac{1}{2}}\,\le\, \bigg(\int_{\R^d}(|\xi|^2+\delta)^{-1}|\hat{f}^t(\xi)|^2\,d\xi\bigg)^{\frac{1}{2}}.
\end{equation*}
{Integrating in time, using the monotone convergence theorem to pass to the limit $\delta\downarrow0$, and appealing again to~\eqref{eq:coercivityinFourier}, we deduce
\begin{equation*}
\|\bar U_\e^{\ell;t}\|_{\Ld^2(\R^d)}\,\lesssim\,\|f\|_{\Ld^1((0,t);\dot{H}^{-1}(\R^d))},
\end{equation*}
and} similarly, due to the constant coefficients and linearity of the equation, we get for all~$r\geq 0$,
\begin{equation}\label{eq:negativeNorm}
\|\langle D\rangle^r\bar U_\e^{\ell;t}\|_{\Ld^2(\R^d)}\lesssim \|\langle D^r\rangle f\|_{\Ld^1((0,t);\dot{H}^{-1}(\R^d))}.
\end{equation}

\step2 Proof of~(ii) for high-frequency filtering.\\
Let $\alpha,\chi$ be fixed. We appeal to~\eqref{eq:coercivityinFourier} with $R=\e^{-\alpha}$: provided that $\e^{1-\alpha}\ll1$ is small enough, we deduce for $|\xi|\le\e^{-\alpha}$,
\begin{equation}\label{eq:coercivitybyfiltering}
\tfrac{1}{C}|\xi|^2\leq\mu^{\ell}_{\e}(\xi)\leq C|\xi|^2.
\end{equation}
Hence, as in Step~1, replacing $f$ by $\chi(\e^{\alpha}\nabla)f$, there is a unique solution $u_{\e}^{\text{(I)},\ell}$ of~\eqref{eq:proxy-baru-I0} in~$\Ld^\infty_\loc(\R;\Ld^2(\R^d))$ with spatial Fourier transform supported in $\R\times\e^{-\alpha}B$, and the claimed a priori estimates similarly follow.

\medskip
\step3 Proof of~(ii) for high{er}-order regularization.\\
In Fourier space, the regularized operator $-\nabla\cdot\big(\bar\Aa+\sum_{k=2}^\ell\bar\bb^{k}\odot(\e\nabla)^{k-1}+\kappa_\ell(\e|\nabla|)^\ell\big)\nabla$
has symbol
\begin{equation*}
\mu^{\operatorname{(II)},\ell}_{\e}(\xi)\,:=\,\xi\cdot\Big(\bar\Aa+\sum_{k=2}^{\ell}\bar{\bb}^k\odot(i\e \xi)^{\otimes k-1}+\kappa_\ell(\e|\xi|)^\ell\Big)\xi.
\end{equation*}
Recall that by Proposition~\ref{prop:vanish} this symbol is real-valued.
Moreover, the lower bound $\xi\cdot\bar\Aa\xi\ge\lambda|\xi|^2$ ensures that $\kappa_\ell$ can indeed be chosen as the smallest value satisfying~\eqref{eq:choice-kappaell}, while the bound~\eqref{e.bd-coeff} entails $\kappa_\ell\le C^\ell$.
This choice of $\kappa_\ell$, together with~\eqref{e.bd-coeff}, yields
\begin{equation*}
\tfrac12 \lambda|\xi|^2 \,\le\, \mu^{\operatorname{(II)},\ell}_{\e}(\xi)\,\le\, C^\ell|\xi|^2\langle \e \xi\rangle^\ell.
\end{equation*}
We can then solve~\eqref{eq:proxy-baru-II} in Fourier space again via Duhamel's formula, and the stated a priori estimates are deduced as in Step~1 using the above coercivity of the regularized symbol. Uniqueness follows by linearity.

\medskip
\step4 Proof of~(ii) for Boussinesq trick.\\
In terms of the symbol
\begin{equation}\label{eq:defin-symbol-III}
\mu_{\e}^{\operatorname{(III)},\ell}(\xi)\,:=\,\frac{\xi \cdot \Big(\sum_{n=1}^\ell 
\big(\kappa_n\bar\Aa+ \sum_{l=1}^{n-1}\kappa_{l} \bar \bb^{n+1-l} \odot (i\frac\xi{|\xi|})^{n-l}\big)(\e|\xi|)^{n-1}\Big)\xi}{1+\sum_{l=2}^{\ell}\kappa_l (\e |\xi|)^{l-1}},
\end{equation}
equation~\eqref{eq:proxy-baru-III} can be written in Fourier space as
\begin{equation}\label{eq:PDEFourierIII}
\partial_t^2\mathcal{F}[\bar u_{\e}^{\operatorname{(III)},\ell}]+\mu_{\e}^{\operatorname{(III)},\ell}\mathcal{F}[\bar u_{\e}^{\operatorname{(III)},\ell}]=\hat f.
\end{equation}
Note that~\eqref{eq:defin-symbol-III} makes sense as $\kappa_l\ge0$ for all $l$.
In addition, the choice~\eqref{eq:choice-kappaell-B} of $\{\kappa_l\}_l$ precisely ensures that all the terms of the sum over $n$ in the numerator of~\eqref{eq:defin-symbol-III} are nonnegative (and actually vanish for $n$ even {due to Proposition \ref{prop:vanish}}). Only keeping the term for $n=1$, and recalling $\kappa_1=1$ by definition, we deduce the lower bound
\begin{equation*}
\mu_{\e}^{\operatorname{(III)},\ell}(\xi)\,\ge\,\frac{\xi \cdot \bar\Aa\xi}{1+\sum_{l=2}^{\ell}\kappa_l (\e |\xi|)^{l-1}}\,\ge\,\frac{\lambda|\xi|^2}{1+\sum_{l=2}^{\ell}\kappa_l (\e |\xi|)^{l-1}},
\end{equation*}
which is pointwise  non-negative.
We can then define a solution of~\eqref{eq:PDEFourierIII} via Duhamel's formula
\begin{equation*}
\mathcal{F}[\bar u_{\e}^{\operatorname{(III)},\ell}](\xi)\,=\,\int_0^t \frac{\sin\big((t-s)\mu^{\operatorname{(III)},\ell}_{\e}(\xi)^{1/2}\big)}{\mu^{\operatorname{(III)},\ell}_{\e}(\xi)^{1/2}}\hat f^s(\xi)\,ds,
\end{equation*}
and thus, upon inverse Fourier transformation, this provides a weak solution of~\eqref{eq:proxy-baru-III} in~$\Ld^\infty_\loc(\R;\Ld^2(\R^d))$.

\medskip\noindent
We turn to the proof of a priori estimates. As in Step~1, it suffices to establish the estimate in energy norm.
For that purpose, we start from the following equivalent formulation of~\eqref{eq:PDEFourierIII},
\begin{equation}\label{eq:PDEFourier_other}
\beta_{\e}^{\ell}\partial_t^2\mathcal{F}[\bar u_{\e}^{\operatorname{(III)},\ell}]+\gamma_{\e}^{\ell}\mathcal{F}[\bar u_{\e}^{\operatorname{(III)},\ell}]\,=\,\beta_{\e}^{\ell}\hat f,
\end{equation}
in terms of the symbols
\begin{eqnarray*}
\beta_{\e}^{\ell}(\xi)&:=&1+\sum_{l=2}^{\ell}\kappa_l (\e |\xi|)^{l-1},\\
\gamma_{\e}^{\ell}(\xi)&:=&\xi \cdot \bigg(\sum_{n=1}^\ell \Big(\kappa_n\bar\Aa+ \sum_{l=1}^{n-1}\kappa_{l} \bar \bb^{n+1-l} \odot (i\tfrac\xi{|\xi|})^{n-l}\Big)(\e|\xi|)^{n-1}\bigg)\xi,
\end{eqnarray*}
with $\gamma_\e^\ell/\beta_\e^\ell=\mu_{\e}^{\operatorname{(III)},\ell}$.
Arguing as in Step~1, and using that $\beta_{\e}^{\ell}(\xi)\geq 1$ and $\gamma_{\e}^{\ell}(\xi)\ge \lambda |\xi|^2$, we find that any ancient solution of~\eqref{eq:PDEFourier_other} satisfies
\begin{equation*}
\|D\bar u_{\e}^{\operatorname{(III)},\ell;t}\|_{L^2(\R^d)}\,\lesssim\,\int_0^t \|(\beta_{\e}^{\ell})^\frac12 \hat f\|_{L^2(\R^d)}.
\end{equation*}
Inserting the upper bound
\begin{equation*}
\beta_{\e}^{\ell}(\xi)\,\le\, \sum_{l=1}^{\ell}(\e C|\xi|)^{l-1}\mathds1_{l\text{ odd}}\,\le\, C^{\ell}\langle\e\xi\rangle^{2\lfloor\frac{\ell-1}2\rfloor}, 
\end{equation*}
we are led to the claimed a priori estimate on the energy norm.

\medskip
\step5 Proof of~(iii): comparison of modified equations.\\
We analyze the differences $\bar v^{(\star),\ell}_\e:=\bar u_{\e}^{(\star),\ell}-\bar U_{\e}^{\ell}$, and we start with $(\star)=\text{(I)}$. By definition, it satisfies the equation
\begin{equation*}
\partial_t^2 \bar v^{\operatorname{(I)},\ell}_\e-\nabla\cdot\Big(\bar\Aa+\sum_{k=2}^\ell\bar\bb^{k}\odot(\e\nabla)^{k-1}\Big)\nabla\bar v^{\operatorname{(I)},\ell}_\e=(\chi(\e^{\alpha}\nabla )-1)f,
\end{equation*}
so that \eqref{eq:negativeNorm} yields for all $r\ge0$,
\begin{equation*}
\|\langle D\rangle^r\bar v_\e^{\operatorname{(I)},\ell;t}\|_{\Ld^2(\R^d)}\,\lesssim\,\|\langle D\rangle^{r}(\chi(\e^{\alpha}\nabla)-1)f\|_{\Ld^1((0,t);\dot H^{-1}(\R^d))}.
\end{equation*}
By the properties of the cut-off function $\chi$, we find for all $\xi\in\R^d$ and $k\ge0$,
\begin{equation}\label{eq:cutoff-error}
|\chi(\e^{\alpha}\xi)-1|\,\le\,\mathds{1}_{|\e^{\alpha}\xi|\geq\frac12}\,\le\, (2\e^{\alpha}|\xi|)^k.
\end{equation}
Choosing $k={\lfloor \ell/\alpha\rfloor}+1$, we then get by Plancherel's formula,
\begin{equation*}
\|\langle D\rangle^r \bar v_\e^{\operatorname{(I)},\ell;t}\|_{\Ld^2(\R^d)}\,\lesssim\,(\e C)^{\ell}\|\langle D\rangle^{K_{\alpha}\ell +r}f\|_{\Ld^1((0,t);\Ld^2(\R^d))}
\end{equation*}
with {$K_{\alpha}\le1/\alpha$}, as claimed.

\medskip\noindent
We turn to the case $(\star)=\operatorname{(II)}$. By definition, the difference $\bar v_\e^{\operatorname{(II)},\ell}$ satisfies the following equation,
\begin{equation*}
\partial_t^2\bar v_\e^{\operatorname{(II)},\ell}-\nabla\cdot\Big(\bar\Aa+\sum_{k=2}^\ell\bar\bb^{k}\odot(\e\nabla)^{k-1}+\kappa_\ell(\e|\nabla|)^\ell\Big)\nabla\bar v_\e^{\operatorname{(II)},\ell}
\,=\,\kappa_\ell (\e|\nabla|)^\ell\triangle\bar U_{\e}^{\ell}.
\end{equation*}
Hence, combining~\eqref{eq:negativeNorm} and the a priori estimate of item~(ii),  together with the bound $\kappa_{\ell}\leq C^{\ell}$, we get
\begin{eqnarray*}
\|\langle D\rangle^r v_\e^{\operatorname{(II)},\ell;t}\|_{\Ld^2(\R^d)}&\le&C^\ell\|\langle D\rangle^r (\e|\nabla|)^\ell\nabla\bar U_{\e}^{\ell}\|_{\Ld^1((0,t);\Ld^2(\R^d))}\\
&\lesssim& (\e C)^{\ell}\|\langle D\rangle^{\ell+r}\nabla\bar U_{\e}^{\ell}\|_{\Ld^1((0,t);\Ld^2(\R^d))}\\
&\lesssim& (\e C)^{\ell}\langle t\rangle\|\langle D\rangle^{\ell+r}f\|_{\Ld^1((0,t);\Ld^2(\R^d))}.
\end{eqnarray*}
It remains to treat the case $(\star)=\operatorname{(III)}$. Starting from~\eqref{eq:homog-lim-sp} and~\eqref{eq:proxy-baru-III}, and recalling $\kappa_1=1$, we get the following equation for the corresponding difference,
\begin{multline*}
\partial_t^2\Big(1+\sum_{l=2}^{\ell}  \kappa_l (\e |\nabla|)^{l-1}\Big)\bar v_\e^{\operatorname{(III)},\ell}
\\
-\nabla\cdot\bigg(\sum_{n=1}^\ell \Big(
 \kappa_{n}\bar\Aa+ \sum_{l=1}^{n-1} \kappa_{l}\bar \bb^{n+1-l}\odot(\tfrac{\nabla}{|\nabla|})^{n-l}\Big)(\e|\nabla|)^{n-1}\bigg) \nabla\bar v_\e^{\operatorname{(III)},\ell}\\
\,=\, 
\sum_{l=2}^{\ell}  \kappa_l (\e |\nabla|)^{l-1}(f-\partial_t^2\bar U_\e^{\ell})
\\
+\nabla\cdot\bigg(\sum_{n=2}^\ell \Big(
 \kappa_{n}\bar\Aa+ \sum_{l=2}^{n-1} \kappa_{l}\bar \bb^{n+1-l}\odot(\tfrac{\nabla}{|\nabla|})^{n-l}\Big)(\e|\nabla|)^{n-1}\bigg) \nabla\bar U_\e^{\ell},
\end{multline*}
and thus, further using the equation for $\bar U_\e^{\ell}$ in the right-hand side, we get after reorganizing the terms,
\begin{multline}\label{eq:reform-uIII-U}
\partial_t^2\Big(1+\sum_{l=2}^{\ell}  \kappa_l (\e |\nabla|)^{l-1}\Big)\bar v_\e^{\operatorname{(III)},\ell}
\\
-\nabla\cdot\bigg(\sum_{n=1}^\ell \Big(
 \kappa_{n}\bar\Aa+ \sum_{l=1}^{n-1} \kappa_{l}\bar \bb^{n+1-l}\odot(\tfrac{\nabla}{|\nabla|})^{n-l}\Big)(\e|\nabla|)^{n-1}\bigg) \nabla\bar v_\e^{\operatorname{(III)},\ell}\\
\,=\,-\nabla\cdot\Big(\sum_{l=2}^{\ell}\sum_{k=\ell+2-l}^\ell\kappa_l\bb^k\odot(\e\nabla)^{k-1}(\e |\nabla|)^{l-1}\Big)\nabla\bar U_\e^{\ell}.
\end{multline}
Combining~{\eqref{eq:negativeNorm} and the a priori estimate of item~(ii)}, together with the bounds $|\bb^k|\le C^k$ and $\kappa_\ell\le C^\ell$, we get
\begin{eqnarray*}
\|\langle D\rangle^r\bar v_\e^{\operatorname{(III)},\ell}\|_{\Ld^2(\R^d)}
&\le&(\e C)^\ell\|\langle D\rangle^{2\ell+r-2}\nabla\bar U_\e^\ell\|_{\Ld^1((0,t);\Ld^2(\R^d))}\\
&\le&(\e C)^\ell\langle t\rangle\|\langle D\rangle^{2\ell+r-2}f\|_{\Ld^1((0,t);\Ld^2(\R^d))},
\end{eqnarray*}
and the conclusion follows.\qed
\endgroup

\subsection{Proof of Theorem~\ref{th:main-per}}
Let $\Aa$ be $Q$-periodic.
We split the proof into three steps. We first establish~\eqref{eq:2scale-concl-per} for the energy norm, before turning to the $\Ld^2$-estimate, which requires some additional care. We start by assuming that $\supp\hat f\subset \R\times B_R$ for some $R\ge1$, and then conclude with the general case in the last step.

\medskip
\step1 Proof of~\eqref{eq:2scale-concl-per} for the energy norm in case $\supp\hat f\subset\R\times B_R$ with $\e R\ll1$.\\
For simplicity, we start by assuming momentarily that the corrector estimates in Lemma~\ref{lem:cor-est-per} hold uniformly in the sense of
\begin{eqnarray}\label{eq:corr-estim-Linfty}
\|(\psi^n,\sigma^n)\|_{{W^{1,\infty}(Q)}}&\le& C^n,\nonumber\\
\text{and}~~\|(\zeta^{n,m},\tau^{n,m})\|_{{W^{1,\infty}(Q)}}&\le& C^{n+m+1},\qquad\text{for all $n,m\ge0$.}
\end{eqnarray}
As $\supp\hat f\subset \R\times B_R$ with $\e R\ll1$, we can consider the solution $\bar U_\e^\ell$ of the formal effective equation~\eqref{eq:homog-lim-sp} as given by Lemma~\ref{lem:apriori-baru-sp}(i).
From Proposition~\ref{prop:form-spectral} and Lemma~\ref{lem:apriori}, using the assumed uniform corrector estimate~\eqref{eq:corr-estim-Linfty}, we then obtain
\begin{multline*}
\|D(u_\e^{t}-S_\e^{\ell}[\bar U_\e^{\ell;t},f^t])\|_{\Ld^2(\R^d)}\\
 \,\le\,
(\e C)^{\ell}\|\langle D\rangle^{2\ell}f\|_{\Ld^1((0,t);\Ld^2(\R^d))}
+(\e C)^\ell\|\langle D\rangle^{2\ell}D\bar U_\e^{\ell}\|_{\Ld^1((0,t);\Ld^2(\R^d))},
\end{multline*}
and thus, combining this with the a priori bounds of Lemma~\ref{lem:apriori-baru-sp}(i),
\begin{equation}\label{e.difficult-L2-er}
\|D(u_\e^{t}-S_\e^{\ell}[\bar U_\e^{\ell;t},f^t])\|_{\Ld^2(\R^d)}\\
 \,\le\,
(\e C)^\ell\langle t\rangle\|\langle D\rangle^{2\ell}f\|_{\Ld^1((0,t);\Ld^2(\R^d))}.
\end{equation}
It remains to replace $\bar U_\e^{\ell}$ by $\bar u_\e^{(\star),\ell}$ in the left-hand side for $(\star)=\text{(I)}$, $\text{(II)}$, or $\text{(III)}$. For that purpose, recalling the definition of the spectral two-scale expansion, cf.~\eqref{eq:2s-spec}, and using the assumed uniform corrector estimates~\eqref{eq:corr-estim-Linfty}, we note that
\begin{equation*}
\|D(S_\e^{\ell}[\bar u_\e^{(\star),\ell},f]-S_\e^{\ell}[\bar U_\e^{\ell},f])\|_{\Ld^2(\R^d)}
\,\le\,C^\ell\|\langle\nabla\rangle^\ell D(\bar u_\e^{(\star),\ell}-\bar U_\e^{\ell})\|_{\Ld^2(\R^d)},
\end{equation*}
hence, by Lemma~\ref{lem:apriori-baru-sp}(iii),
\begin{equation*}
\|D(S_\e^{\ell}[\bar u_\e^{(\star),\ell;t},f^t]-S_\e^{\ell}[\bar U_\e^{\ell;t},f^t])\|_{\Ld^2(\R^d)}
\,\le\,(\e C)^\ell\langle t\rangle\|\langle D\rangle^{C\ell}f\|_{\Ld^1((0,t);\Ld^2(\R^d))}.
\end{equation*}
Combined with~\eqref{e.difficult-L2-er}, this proves the claim~\eqref{eq:2scale-concl-per} for the energy norm in case $\supp\hat f\subset \R\times B_R$ with $\e R\ll1$, provided that~\eqref{eq:corr-estim-Linfty} holds.

\medskip\noindent
It remains to treat the case when the uniform boundedness assumption~\eqref{eq:corr-estim-Linfty} for correctors is not satisfied.
In that case, we rather appeal to the Sobolev embedding to estimate products with correctors: for any periodic corrector {or corrector gradient $\varphi\in\{\psi^n,\nabla\psi^n,\sigma^n,\nabla\sigma^n,\zeta^{n,m},\nabla\zeta^{n,m},\tau^{n,m},\nabla\tau^{n,m}\}_{n,m}$}, we can estimate, for any function $g$,
\begin{eqnarray}
\|\varphi(\tfrac\cdot\e) g\|_{\Ld^2(\R^d)}^2
&\le&\int_{\R^d} \Big(\fint_{B_\e(x)} |\varphi(\tfrac\cdot\e)|^2\Big)\Big(\sup_{B_\e(x)} |g|^2\Big)\,dx\nonumber\\
&\lesssim&\|\varphi\|_{\Ld^2(Q)}^2\int_{\R^d}\Big(\sup_{B_\e(x)} |g|^2\Big)\,dx\nonumber\\
&\lesssim&\|\varphi\|_{\Ld^2(Q)}^2\|g\|_{H^a(\R^d)}^2,\label{eq:use-Sobolev-cor}
\end{eqnarray}
provided $a>\frac d2$.
Up to a fixed loss of derivatives in the estimates, we may then appeal to the $\Ld^2$ corrector estimates in Lemma~\ref{lem:cor-est-per}, and the above proof of~\eqref{eq:2scale-concl-per} for the energy norm is adapted directly.

\medskip
\step2 Proof of \eqref{eq:2scale-concl-per} for the $\Ld^2$-norm in case $\supp\hat f\subset B_R$ with $\e R\ll1$.\\
As in Step~1, we aim to apply Proposition~\ref{prop:form-spectral} and Lemma~\ref{lem:apriori}, together with corrector estimates. However, the following terms are a priori problematic in the right-hand side of the equation for the spectral two-scale expansion given by Proposition~\ref{prop:form-spectral},
\begin{multline}\label{eq:def-TepsL}
T_\e^\ell\,:=\,-\e^\ell\sigma^\ell_{i_1\ldots i_{\ell}}(\tfrac\cdot\e): \gamma_\ell(\e\nabla) \nabla^2\nabla^{\ell}_{i_1\ldots i_{\ell}} \bar U_\e^{\ell}
\\
+\sum_{n=1}^\ell\sum_{k=\ell+2-n}^{\ell+1} \e^{n+k-2} \psi^n_{i_1\ldots i_n}(\tfrac\cdot\e) \bar \bb^{k-1}_{i_{n+1}\ldots i_{n+k-2}}:\gamma_\ell(\e\nabla)\nabla^2\nabla_{i_1\ldots i_{n+k-2}}^{n+k-2}\bar U_\e^{\ell}.
\end{multline}
Indeed, these terms are not total derivatives and involve $\bar U_\e^\ell$ itself: when applying Lemma~\ref{lem:apriori} to estimate the $\Ld^2$-norm of the two-scale expansion error, these terms would therefore contribute like
\[(\e C)^\ell\langle t\rangle^2\|{\langle{D}\rangle^{C\ell}} f\|_{\Ld^1((0,t);\Ld^2(\R^d))}\]
with a prefactor $\langle t\rangle^2$ instead of $\langle t\rangle$.
In order to improve on this, we shall reformulate $T_\e^\ell$ as a total time-derivative up to terms that depend only locally on~$f$.
Using the short-hand notation
\[\bar \Lc_\e^\ell\,:=\,-\nabla\cdot\Big(\bar\Aa+\sum_{k=2}^\ell\bar\bb^{k}\odot(\e\nabla)^{k-1}\Big)\nabla,\]
the effective equation for $\bar U_\e^\ell$ entails
\[\bar \Lc_\e^{\ell} \nabla^2 \bar U_\e^{\ell} = - \partial_{t}^2\nabla^2 \bar U_\e^{\ell}+\nabla^2 f.\]
As in the proof of Lemma~\ref{lem:apriori-baru-sp}(i), cf.~\eqref{eq:coercivityinFourier}, the assumption $\e R\ll1$ precisely ensures that the operator $\bar \Lc_\e^{\ell}:\Ld^2(\R^d)\to\dot H^{-2}(\R^d)$ can be inverted when restricted to functions with spatial Fourier transform supported in $B_R$. As by definition both $\bar U_\e^{\ell}$ and $f$ have spatial Fourier transform supported in $B_R$, we may then write
\begin{equation}\label{eq:reform-barU-ut}
\nabla^2 \bar U_\e^{\ell} = - \partial_{t}^2 (\bar \Lc_\e^{\ell})^{-1} \nabla^2  \bar U_\e^{\ell}+ (\bar \Lc_\e^{\ell})^{-1} \nabla^2 f.
\end{equation}
By uniform ellipticity~\eqref{eq:coercivityinFourier}, we have for any function $g$ with $\supp\hat g\subset B_R$,
\begin{equation}\label{eq:coerxivityinFourier-appl}
\|(\bar \Lc_\e^{\ell})^{-1} \nabla^2 g\|_{\Ld^2(\R^d)} \,\lesssim\, \|g\|_{\Ld^2(\R^d)}.
\end{equation}
Now using~\eqref{eq:reform-barU-ut} to reformulate $T_\e^\ell$, cf.~\eqref{eq:def-TepsL}, we get
\begin{multline*}
T_\e^\ell\,=\,
\e^\ell\partial_{t}^2 \Big(\sigma^\ell_{i_1\ldots i_{\ell}}(\tfrac\cdot\e): \gamma_\ell(\e\nabla) \nabla^{\ell}_{i_1\ldots i_{\ell}}(\bar \Lc_\e^{\ell})^{-1} \nabla^2  \bar U_\e^{\ell}\Big)\\
-\sum_{n=1}^\ell\sum_{k=\ell+2-n}^{\ell+1} \e^{n+k-2}\partial_{t}^2\Big( \psi^n_{i_1\ldots i_n}(\tfrac\cdot\e) \bar \bb^{k-1}_{i_{n+1}\ldots i_{n+k-2}}:\gamma_\ell(\e\nabla)\nabla_{i_1\ldots i_{n+k-2}}^{n+k-2} (\bar \Lc_\e^{\ell})^{-1} \nabla^2  \bar U_\e^{\ell}\Big)\\
-\e^\ell\sigma^\ell_{i_1\ldots i_{\ell}}(\tfrac\cdot\e): \gamma_\ell(\e\nabla) \nabla^{\ell}_{i_1\ldots i_{\ell}}(\bar \Lc_\e^{\ell})^{-1} \nabla^2 f\\
+\sum_{n=1}^\ell\sum_{k=\ell+2-n}^{\ell+1} \e^{n+k-2} \psi^n_{i_1\ldots i_n}(\tfrac\cdot\e) \bar \bb^{k-1}_{i_{n+1}\ldots i_{n+k-2}}:\gamma_\ell(\e\nabla)\nabla_{i_1\ldots i_{n+k-2}}^{n+k-2}(\bar \Lc_\e^{\ell})^{-1} \nabla^2 f.
\end{multline*}
Using this to replace the corresponding terms in the equation for the two-scale expansion error in Proposition~\ref{prop:form-spectral}, appealing to Lemma~\ref{lem:apriori} to estimate its $\Ld^2$-norm, using the corrector estimates of Lemma~\ref{lem:cor-est-per}, using the Sobolev embedding to estimate products with correctors as in~\eqref{eq:use-Sobolev-cor}, and using~\eqref{eq:coerxivityinFourier-appl}, we get for $a>\frac d2$,
\begin{multline*}
\|u_\e^t-S_\e^\ell[\bar U_\e^{\ell;t},f^t]\|_{\Ld^2(\R^d)}\\
\,\le\,(\e C)^\ell\langle t\rangle\|\langle D\rangle^{2\ell+a} f\|_{\Ld^1((0,t);\Ld^2(\R^d))}
+(\e C)^\ell\|\langle D\rangle^{2\ell+a}D\bar U_\e^\ell\|_{\Ld^1((0,t);\Ld^2(\R^d))},
\end{multline*}
and thus, by the a priori estimate of Lemma~\ref{lem:apriori-baru-sp}(i),
\begin{equation}\label{eq:pre-step2-ueps-SbarUeps}
\|u_\e^t-S_\e^\ell[\bar U_\e^{\ell;t},f^t]\|_{\Ld^2(\R^d)}
\,\le\,(\e C)^\ell\langle t\rangle\|\langle D\rangle^{2\ell+a} f\|_{\Ld^1((0,t);\Ld^2(\R^d))}.
\end{equation}
It remains to argue as in Step~1 to replace $\bar U_\e^\ell$ by $\bar u_\e^{(\star),\ell}$ in the left-hand side.
For that purpose, recalling the definition of the spectral two-scale expansion, cf.~\eqref{eq:2s-spec}, and using again the corrector estimates of Lemma~\ref{lem:cor-est-per} together with the Sobolev embedding as in~\eqref{eq:use-Sobolev-cor}, we note that for $a>\frac d2$,
\begin{equation*}
\|S_\e^{\ell}[\bar u_\e^{(\star),\ell},f]-S_\e^{\ell}[\bar U_\e^{\ell},f]\|_{\Ld^2(\R^d)}
\,\le\,C^\ell\|\langle\nabla\rangle^{\ell+a} (\bar u_\e^{(\star),\ell}-\bar U_\e^{\ell})\|_{\Ld^2(\R^d)},
\end{equation*}
hence, by Lemma~\ref{lem:apriori-baru-sp}(iii),
\begin{equation*}
\|S_\e^{\ell}[\bar u_\e^{(\star),\ell;t},f^t]-S_\e^{\ell}[\bar U_\e^{\ell;t},f^t]\|_{\Ld^2(\R^d)}
\,\le\,(\e C)^\ell\langle t\rangle\|\langle D\rangle^{C\ell}f\|_{\Ld^1((0,t);\Ld^2(\R^d))}.
\end{equation*}
Combined with~\eqref{eq:pre-step2-ueps-SbarUeps}, this proves the claim~\eqref{eq:2scale-concl-per} for the $\Ld^2$-norm in case $\supp\hat f\subset B_R$ with $\e R\ll1$.

\medskip
\step3 Approximation argument: proof for general $f$.\\
For $R\ge1$, consider the truncated impulse
\[f_R\,:=\,\chi(\tfrac1R\nabla)f,\]
and let $\bar u_{\e,R}^{(\star),\ell}$ be the solution of the modified effective equations given by Lemma~\ref{lem:apriori-baru-sp}(ii) with impulse $f$ replaced by $f_R$.
Recalling the definition of the spectral two-scale expansion, cf.~\eqref{eq:2s-spec}, and using again corrector estimates, we then note that for $a>\frac d2$,
\begin{multline*}
\|D(S_\e^{\ell}[\bar u_{\e,R}^{(\star),\ell},f_R]-S_\e^{\ell}[\bar u_\e^{(\star),\ell},f])\|_{\Ld^2(\R^d)}\\
\,\le\,C^\ell\|\langle\nabla\rangle^{\ell+a} D(\bar u_\e^{(\star),\ell}-\bar u_{\e,R}^{(\star),\ell})\|_{\Ld^2(\R^d)}
+C^\ell\|\langle D\rangle^{\ell+a}(f-f_R)\|_{\Ld^2(\R^d)},
\end{multline*}
and thus, by linearity and by the a priori estimates of Lemma~\ref{lem:apriori-baru-sp}(ii),
\begin{equation*}
\|D(S_\e^{\ell}[\bar u_{\e,R}^{(\star),\ell;t},f_R^t]-S_\e^{\ell}[\bar u_\e^{(\star),\ell;t},f^t])\|_{\Ld^2(\R^d)}
\,\le\,C^\ell\|{\langle D\rangle^{C\ell}}(f-f_R)\|_{\Ld^1((0,t);\Ld^2(\R^d))}.
\end{equation*}
By definition of $f_R$, as in~\eqref{eq:cutoff-error}, the right-hand side can now be estimated as follows, for any $k\ge0$,
\begin{equation*}
\|D(S_\e^{\ell}[\bar u_{\e,R}^{(\star),\ell;t},f_R^t]-S_\e^{\ell}[\bar u_\e^{(\star),\ell;t},f^t])\|_{\Ld^2(\R^d)}
\,\le\,R^{-k}C^\ell\|\langle D\rangle^{C\ell+k}f\|_{\Ld^1((0,t);\Ld^2(\R^d))}.
\end{equation*}
Combining this with the results of Steps~1 and~2, the conclusion follows for instance with the choice $R=\e^{-1/2}$ and $k=2\ell$.
\qed

\subsection{Proof of Corollary~\ref{cor:summable}}
The bound \eqref{e.bound-summable-1} is obtained along the same line as Theorem~\ref{th:main-per} (using a straightforward adaptation of Lemma~\ref{lem:apriori}),
and it remains to deduce~\eqref{e.bound-summable-2}.
For that purpose, we note that~\eqref{e.bound-summable-2} would actually follow from~\eqref{e.bound-summable-1} together with the bound
\begin{equation}\label{eq:goal-cor1-pre}
\int_{-\infty}^t(t-s) \|\langle D\rangle^{C\ell}   f^s\|_{\Ld^2(\R^d)}\,ds\, \le\, {C_f^\ell},
\end{equation}
in favor of which we presently argue.
Recall that we assume here $f^t(x)=f_1(t)f_2(x)$, where $f_1$ has a smooth and compactly supported Fourier transform on~$\R$ and where~$f_2$ has a compactly supported Fourier transform on~$\R^d$.
The assumption on $f_2$ entails $\|\langle\nabla\rangle^{C\ell} f_2\|_{\Ld^2(\R^d)} \le{C_{f_2}^\ell}$, while the assumption on $f_1$ yields for $s\le0$,
\begin{eqnarray*}
\|\langle\partial_t\rangle^{C\ell}  f_1\|_{\Ld^1((-\infty,s))}&\lesssim&\langle s\rangle^{-2}\|\langle\cdot\rangle^3\langle\partial_t\rangle^{C\ell}  f_1\|_{\Ld^\infty((-\infty,s))}\\
&\le&\langle s\rangle^{-2}C^\ell\|\langle\cdot\rangle^{C\ell}\langle\partial\rangle^3\hat f_1\|_{\Ld^\infty((-\infty,s))}\\
&\le&\langle s\rangle^{-2}{C_{f_1}^\ell},
\end{eqnarray*}
where $\hat f_1$ stands for the temporal Fourier transform of $f_1$.
The claim~\eqref{eq:goal-cor1-pre} follows.
\qed

\subsection{Proof of Theorem~\ref{th:main2}}\label{sec:proof-random}
We briefly explain how the above proof of Theorem~\ref{th:main-per} is adapted to the random setting.
As for Theorem~\ref{th:main-per}, we may assume $\supp\hat f\subset\R\times B_R$ with $\e R\ll1$, and the general conclusion follows by approximation.
As explained in Appendix~\ref{sec:rand-cor}, the only difference with respect to the periodic setting is that only a finite number $\ell_*=\lceil\frac{\beta\wedge d}2\rceil$ of correctors can be defined with stationary gradient, and the highest-order corrector has a nontrivial sublinear growth.
Using Proposition~\ref{prop:form-spectral} in combination with Lemma~\ref{lem:apriori} to estimate the energy norm of the two-scale expansion error of order~$\ell\le\ell_*$, and using the corrector estimates of Appendix~\ref{sec:rand-cor} and the Sobolev embedding, we find for $a>\frac d2$,
\begin{multline*}
\|D(u_\e^t-S_\e^{\ell}[\bar U_\e^{\ell;t},f^t])\|_{\Ld^q(\Omega;\Ld^2(\R^d))}\\
\,\lesssim\,
\e^{\ell}\|\mu_\ell^*(\tfrac\cdot\e)\gamma(\e\nabla)\langle D\rangle^{2\ell+a}f\|_{\Ld^1((0,t);\Ld^2(\R^d))}\\
+\e^{\ell}\|\mu_\ell^*(\tfrac\cdot\e)\gamma(\e\nabla)\langle D\rangle^{2\ell+a}D\bar U_\e^{\ell}\|_{\Ld^1((0,t);\Ld^2(\R^d))},
\end{multline*}
where the weight $\mu_\ell^*$ originates in the growth of correctors and is defined in~\eqref{eq:mun-*}.
A novelty with respect to the proof of Theorem~\ref{th:main-per} is that we now need weighted energy estimates for~$\bar U_\e^{\ell}$ with sublinear weight $\mu_\ell^*$.
For that purpose, as the homogenized equation~\eqref{eq:homog-lim-sp} has constant coefficients,
we note that $\bar U_\e^{\ell}$ displays ballistic transport, and therefore
\begin{eqnarray*}
\|\mu_\ell^*(\tfrac\cdot\e)\gamma(\e\nabla)\langle D\rangle^{2\ell+a}D\bar U_\e^{\ell;t}\|_{\Ld^2(\R^d)}
&\lesssim&\mu_\ell^*(\tfrac{1}\e\langle t\rangle)\,\|\langle\cdot\rangle\gamma(\e\nabla)\langle D\rangle^{C}f\|_{\Ld^1((0,t);\Ld^2(\R^d))}\\
&\lesssim&\mu_\ell^*(\tfrac{1}\e\langle t\rangle)\,\|\langle\cdot\rangle\langle D\rangle^{C}f\|_{\Ld^1((0,t);\Ld^2(\R^d))}.
\end{eqnarray*}
This is easily obtained by interpolation, using that the weight $x$ corresponds to a derivative in Fourier space and using the Duhamel formula for $\bar U_\e^{\ell}$; see e.g.~\cite[Proof of Proposition~3]{BG} for the details.
The above then becomes
\begin{equation*}
\|D(u_\e^t-S_\e^{\ell}[\bar U_\e^{\ell;t},f^t])\|_{\Ld^q(\Omega;\Ld^2(\R^d))}
\,\lesssim\,
\e^{\ell}\langle t\rangle\,\mu_\ell^*(\tfrac1\e\langle t\rangle)\,\|\langle\cdot\rangle\langle D\rangle^{C}f\|_{\Ld^1((0,t);\Ld^2(\R^d))}.
\end{equation*}
Note that the error estimate for $\ell=\ell_*-1$ is occasionally better than the one for $\ell=\ell_*$.
Optimizing between the results for $\ell=\ell_*$ and for $\ell=\ell_*-1$, we easily conclude
\begin{equation*}
\|D(u_\e^t-S_\e^{\ell_*}[\bar U_\e^{\ell_*;t},f^t])\|_{\Ld^q(\Omega;\Ld^2(\R^d))}
\,\lesssim\,
\e^{\ell_*}\langle t\rangle\big(\mu_{\ell_*}^*(\tfrac1\e\langle t\rangle)\wedge\tfrac1\e\big)\,\|\langle\cdot\rangle\langle D\rangle^{C}f\|_{\Ld^1((0,t);\Ld^2(\R^d))}.
\end{equation*}
As in the proof of Theorem~\ref{th:main-per}, we can replace $\bar U_\e^{\ell_*}$ in the left-hand side by the solution of any well-posed modification of the formal homogenized equation, and we can derive a similar estimate for the $\Ld^2$-norm.
\qed

\section{Geometric approach and hyperbolic two-scale expansion} \label{sec:structure}

This section is devoted to the definition of hyperbolic correctors and to the proof of Theorem~\ref{th:main-per2}, including the well-posedness of the formal homogenized equation~\eqref{eq:form-hom}.
This essentially constitutes a rewriting of \cite{ALR,Pouch-17,Pouch-19} and is needed to rigorously relate those works to the spectral two-scale expansion, cf.~Section~\ref{sec:fromGeom2Spec}.

\subsection{Definition of hyperbolic correctors}\label{sec:cor-eqn}
We start with the definition of the hyperbolic correctors $\{\phi^{n,m}\}_{n,m}$ and of the homogenized tensors $\{\bar\Aa^{n,m}\}_{n,m}$, as motivated in Section~\ref{sec:geom-insight}.

\begin{defin}[Hyperbolic correctors]\label{defi:BC}
In the periodic setting, the hyperbolic correctors $\{\phi^{n,m}\}_{n,m\ge0}$, homogenized tensors $\{\bar \Aa^{n,m}\}_{n\ge1,m\ge0}$, and fluxes $\{q^{n,m}\}_{n,m\ge0}$ are inductively defined as follows:
\begin{enumerate}[$\bullet$]
\item We set $\phi^{0,0}:=1$ and $\phi^{0,m}:=0$ for $m\ge1$, while for $n\ge1$ and $m\ge0$ we define $\phi^{n,m}:=(\phi^{n,m}_{j_1\ldots j_n})_{1\le j_1,\ldots, j_n\le d}$ with $\phi_{j_1\ldots j_n}^{n,m}\in H^1_\per(Q)$ the periodic scalar field that has vanishing average $\expecm{\phi_{j_1\ldots j_n}^{n,m}}=0$ and satisfies
\begin{equation*}
\quad-\nabla\cdot\Aa\nabla\phi_{j_1\ldots j_n}^{n,m}\,=\,\nabla\cdot\big(\Aa\phi_{j_1\ldots j_{n-1}}^{n-1,m}\ee_{j_n}\big)
+\ee_{j_n}\cdot q^{n-1,m}_{j_1\ldots j_{n-1}}.
\end{equation*}
\item For $n\ge1$ and $m\ge0$, we define $\bar\Aa^{n,m}:=(\bar\Aa^{n,m}_{j_1\ldots j_{n-1}})_{1\le j_1,\ldots, j_{n-1}\le d}$ as the matrix-valued $(n-1)$th-order tensor given by
\[\quad\bar\Aa^{n,m}_{j_1\ldots j_{n-1}}\ee_{j}\,:=\,\expec{\Aa\big(\nabla\phi_{j_1\ldots j_{n-1}j}^{n,m}+\phi_{j_1\ldots j_{n-1}}^{n-1,m}\ee_j\big)}.\]
\item For $n,m\ge0$, we define $q^{n,m}:=(q^{n,m}_{j_1\ldots j_n})_{1\le j_1,\ldots,j_n\le d}$ with $q^{n,m}_{j_1\ldots j_n}\in\Ld^2_\per(Q)^{d}$ the periodic vector field given by
\[q^{n,m}_{j_1\ldots j_n}\,:=\,\Aa\big(\nabla\phi_{j_1\ldots j_{n}}^{n,m}+\phi_{j_1\ldots j_{n-1}}^{n-1,m}\ee_{j_{n}}\big)-\phi_{j_1\ldots j_{n}j}^{n+1,m-2}\ee_j-\bar\Aa_{j_1\ldots j_{n-1}}^{n,m}\ee_{j_{n}},\]
where the definition of $\bar \Aa^{n,m}$ ensures $\expec{q^{n,m}}=0$.
\end{enumerate}
In particular, there holds $\phi^{n,m}=0$, $\bar\Aa^{n,m}=0$, and $q^{n,m}=0$ whenever $m$ is an odd integer.\footnote{This is natural as time derivatives should always come in even numbers in view of equation~\eqref{eq:hyperb}.}
\end{defin}

For all $n$, we note that $\phi^{n,0}$ coincides with the elliptic corrector of order $n$, cf.~\cite{BLP-78}.
As \no{is} common in the elliptic setting, it is useful to further introduce suitable flux correctors, which indeed allow us to refine error estimates and slightly improve on the result of~\cite{ALR}. More precisely, flux correctors are designed to allow a direct optimal exploitation of cancellations due to fluxes having vanishing average $\expecm{q^{n,m}}=0$.
We start with the definition of flux correctors for $m\ge1$.

\begin{defin}[Hyperbolic flux correctors]\label{def:cocor}
For $n\ge 0$ and $m\ge1$, we define the hyperbolic flux corrector $\sigma^{n,m}:=(\sigma^{n,m}_{j_1\ldots j_n})_{1\le j_1,\ldots,j_n\le d}$ by
\[\sigma^{n,m}_{j_1\ldots j_n}=(\nabla\Phi^{n,m}_{j_1\ldots j_n})',\]
where $\Phi^{n,m}_{j_1\ldots j_n}\in H^1_\per(Q)^d$ is the periodic vector field that satisfies
\mbox{$\expecm{\Phi_{j_1\ldots j_{n}}^{n,m}}=0$} and
\[-\triangle\Phi^{n,m}_{j_1\ldots j_{n}}\,=\,q^{n,m}_{j_1\ldots j_n}.\qedhere\]
\end{defin}

In case $m=0$, as correctors $\phi^{n,0}$ coincide with elliptic correctors, they display more structure than general hyperbolic correctors. We recall how this structure can be exploited to construct a suitable flux corrector $\sigma^{n,0}$ that is skew-symmetric: the following lemma is essentially borrowed from~\cite{DO1}. 

\begin{lem}[Elliptic correctors and flux correctors]\label{lem:mod-cor}
Up to symmetrization of indices, the elliptic correctors $\{\phi^{n,0}\}_{n\ge0}$ and homogenized tensors $\{\bar\Aa^{n,0}\}_{n\ge1}$ coincide with the modified families $\{\tilde\phi^{n,0}\}_{n\ge0}$ and $\{\tilde\Aa^{n,0}\}_{n\ge1}$ defined via the following cell problems:
\begin{enumerate}[$\bullet$]
\item We set $\tilde\phi^{0,0}:=1$ and for $n\ge1$ we define $\tilde\phi^{n,0}:=(\tilde\phi^{n,0}_{j_1\ldots j_n})_{1\le j_1,\ldots j_n,\le d}$ with $\tilde\phi_{j_1\ldots j_n}^{n,0}\in H^1_\per(Q)$ the periodic scalar field that has vanishing average $\expecm{\tilde\phi_{j_1\ldots j_n}^{n,0}}=0$ and satisfies
\begin{equation*}
\quad
-\nabla\cdot\Aa\nabla\tilde\phi_{j_1\ldots j_n}^{n,0}\,=\,\nabla\cdot\big(\Aa\tilde\phi_{j_1\ldots j_{n-1}}^{n-1,0}\ee_{j_n}\big)
+\ee_{j_n}\cdot\tilde q^{n-1,0}_{j_1\ldots j_{n-1}}.
\end{equation*}
\item For $n\ge1$, we define $\tilde\Aa^{n,0}:=(\tilde\Aa^{n,0}_{j_1\ldots j_{n-1}})_{1\le j_1,\ldots, j_{n-1}\le d}$ with $\tilde\Aa^{n,0}_{j_1\ldots j_{n-1}}$ the matrix given by
\[\quad\tilde\Aa^{n,0}_{j_1\ldots j_{n-1}}\ee_{j}\,:=\,\expec{\Aa\big(\nabla\tilde\phi_{j_1\ldots j_{n-1}j}^{n,0}+\tilde\phi_{j_1\ldots j_{n-1}}^{n-1,0}\ee_j\big)}.\]
\item For $n\ge0$, we define $\tilde q^{n,0}:=(\tilde q^{n,0}_{j_1\ldots j_n})_{1\le j_1,\ldots,j_n\le d}$ with $\tilde q^{n,0}_{j_1\ldots j_n}\in\Ld^2_\per(Q)^{d}$ the periodic vector field given by
\[\quad\tilde q_{j_1\ldots j_n}^{n,0}:=\Aa\big(\nabla\tilde\phi_{j_1\ldots j_n}^{n,0}+\tilde\phi_{j_1\ldots j_{n-1}}^{n-1,0}\ee_{j_n}\big)-\tilde\Aa^{n,0}_{j_1\ldots j_{n-1}}\ee_{j_n}-\sigma_{j_1\ldots j_{n-1}}^{n-1,0}\ee_{j_n},\]
where the definition of $\tilde\Aa^{n,0}$ and $\sigma^{n-1,0}$  ensures $\expec{\tilde q^{n,0}}=0$ and $\nabla\cdot\tilde q^{n,0}_{j_1\ldots j_n}=0$.
\item We set $\sigma^{0,0}:=0$ and we define the flux corrector $\sigma^{n,0}:=(\sigma^{n,0}_{j_1\ldots j_n})_{1\le j_1,\ldots,j_n\le d}$ for $n\ge1$  with $\sigma^{n,0}_{j_1\ldots j_n}\in H^1_\per(Q)^{d\times d}$ the periodic skew-symmetric matrix field that has vanishing average $\expecm{\sigma^{n,0}_{j_1\ldots j_n}}=0$ and satisfies
\[\quad-\triangle\sigma^{n,0}_{j_1\ldots j_n}=\nabla\times\tilde q^{n,0}_{j_1\ldots j_n},\qquad\nabla\cdot\sigma^{n,0}_{j_1\ldots j_n}=\tilde q_{j_1\ldots j_n}^{n,0},\]
where we use the vectorial notation $(\nabla\times F)_{ij}:=\nabla_iF_j-\nabla_jF_i$ for a vector field $F$.
\end{enumerate}
More precisely, these modified correctors coincide with $\{\phi^{n,0}\}_{n\ge0}$ in the sense that we have for all $n\ge0$ and $\xi\in\R^d$,
\[\phi^{n,0}\odot\xi^{\otimes n}\,=\,\tilde\phi^{n,0}\odot\xi^n,\qquad
(\bar\Aa^{n,0}\odot\xi^{\otimes(n-1)})\xi\,=\,(\tilde\Aa^{n,0}\odot\xi^{\otimes(n-1)})\xi.\qedhere\]
\end{lem}

\begin{proof}
We refer to~\cite{GNO-reg} or~\cite{DO1} for the construction of the skew-symmetric flux corrector~$\sigma^{n,0}$, based on the compatibility condition $\nabla\cdot\tilde q^{n,0}_{j_1\ldots j_n}=0$,
and we now turn to the equivalence of $\phi^{n,0}$ and $\tilde\phi^{n,0}$.
Setting
\begin{equation}\label{e.sigma0-1}
\hat q_{j_1\ldots j_n}^{n,0}:=\Aa\big(\nabla\phi_{j_1\ldots j_n}^{n,0}+\phi_{j_1\ldots j_{n-1}}^{n-1,0}\ee_{j_n}\big)-\bar\Aa^{n,0}_{j_1\ldots j_{n-1}}\ee_{j_n}- \sigma^{n-1,0}_{j_1\ldots j_{n-1}}\ee_{j_n},
\end{equation}
the equation for $\phi^{n,0}$ in Definition~\ref{defi:BC} for $n\ge1$  can be written as
\begin{equation}\label{e.sigma0-2}
-\nabla\cdot\Aa\nabla\phi_{j_1\ldots j_n}^{n,0}\,=\,\nabla\cdot\big(\Aa\phi_{j_1\ldots j_{n-1}}^{n-1,0}\ee_{j_n}\big)
+\ee_{j_n}\cdot\big( \hat q^{n-1,0}_{j_1\ldots j_{n-1}}+ \sigma_{j_1\ldots j_{n-2}}^{n-2,0}\ee_{j_{n-1}}\big).
\end{equation}
Symmetrizing indices $j_1,\ldots,j_n$, the skew-symmetry of $\sigma^{n-2,0}$ allows us to drop the corresponding right-hand side term, and we may conclude by induction that $\phi^{n,0}$ coincides with its modified version $\tilde\phi^{n,0}$ up to symmetrization as stated.
\end{proof}

An iterative use of the Poincaré inequality on the unit cell $Q$ ensures the well-posedness of the above objects and further provides the following a priori estimates.

\begin{lem}\label{lem:per-cor}
In the periodic setting, the above quantities $\{\phi^{n,m},\sigma^{n,m},\bar\Aa^{n,m})_{n,m}$ are uniquely defined and satisfy for all $n,m\ge0$,
\[\|(\phi^{n,m},\sigma^{n,m})\|_{H^1(Q)}+|\bar\Aa^{n,m}|\,\le\, C^{n+m}.\qedhere\]
\end{lem}

We conclude this paragraph with some important vanishing property of the higher-order hyperbolic homogenized coefficients, which extends the corresponding elliptic result~\cite[Lemma~2.3]{DO1}; see also~\cite[Proposition~1]{BG}, \cite[Theorem~2.13]{ALR}, and~\cite[Theorem~3.5]{Pouch-19}. The proof is postponed to Section~\ref{sec:even-coeff}.

\begin{prop}[Symmetry of homogenized coefficients]\label{prop:even-coeff}
For all $n\ge1$ and $m\ge0$, there holds for any $j_1,\ldots,j_{n+1}$,
\[\ee_{j_{n+1}}\cdot\bar\Aa^{n,m}_{j_1\ldots j_{n-1}}\ee_{j_n}\,=\,(-1)^{n+1}\ee_{j_1}\cdot\bar\Aa^{n,m}_{j_{n+1}\ldots j_{3}}\ee_{j_2}.\]
In particular, whenever $n$ is even, we have
\[\xi\cdot(\bar\Aa^{n,m}\odot\xi^{\otimes(n-1)})\xi=0,\qquad\text{for all $\xi\in\R^d$}.\qedhere\]
\end{prop}

\subsection{Geometric two-scale expansion}
Given a smooth function $\bar w$, we consider its two-scale expansion associated with the above-defined hyperbolic correctors, as in~\eqref{eq:2scale-hyp},
\begin{equation}\label{eq:def-Fell}
H^\ell_\e[\bar w]\,:=\,\sum_{n=0}^\ell\sum_{m=0}^{\ell-n}\e^{n+m} \phi^{n,m}(\tfrac\cdot\e)\odot\nabla^n\partial_t^m\bar w,
\end{equation}
and we show that it is well-adapted to describe the local behavior of the solution to the hyperbolic equation in the following sense: as explained in~\eqref{eq:describe-osc-op}, the heterogeneous hyperbolic operator applied to $H^\ell_\e[\bar w]$ is equivalent to some higher-order effective operator applied to $\bar w$ (up to $O(\e^\ell)$ error terms).
The proof is postponed to Section~\ref{sec:form-hyperb}.

\begin{prop}[Geometric two-scale expansion]\label{prop:form-hyperb}
Let $\ell\ge1$, $\e>0$, and let $\bar w,f$ be smooth function{s} satisfying
\begin{equation}\label{eq:motiv-homog-hypeq}
\partial_t^{2}\bar w-\nabla\cdot\Big(\sum_{n=1}^{\ell}\sum_{m=0}^{\ell-n}\bar\Aa^{n,m}\odot(\e\nabla)^{n-1}(\e\partial_t)^m\Big)\nabla\bar w\,=\,f.
\end{equation}
Then, the associated geometric two-scale expansion of order $\ell$, defined in~\eqref{eq:def-Fell}, satisfies the following relation in the distributional sense in $\R\times\R^d$,
\begin{multline*}
\big(\partial_t^2-\nabla\cdot\Aa(\tfrac\cdot\e)\nabla\big)H^\ell_\e[\bar w]\\
=\,
f
-\e^{\ell}\sum_{n=0}^{\ell}\nabla\cdot\Big(\big(\Aa\phi_{j_1\ldots j_{n}}^{n,\ell-n}-\sigma^{n,\ell-n}_{j_1\ldots j_{n}}\big)(\tfrac\cdot\e)\nabla\nabla^{n}_{j_1\ldots j_{n}}\partial_t^{\ell-n}\bar w\Big)\hspace{3.5cm}
\\
+\e^\ell\sum_{n=1}^\ell\partial_t\Big(\phi^{n,\ell-n}(\tfrac\cdot\e)\odot\nabla^n\partial_t^{\ell+1-n}\bar w
-\sigma^{n-1,\ell+1-n}_{j_1\ldots j_{n-1}}(\tfrac\cdot\e):\nabla^2\nabla^{n-1}_{j_1\ldots j_{n-1}}\partial_t^{\ell-n}\bar w\Big).\qedhere
\end{multline*}
\end{prop}

\begin{rem}\label{rem:flux-corr}
Note that the last two right-hand side terms in the above equation for the two-scale expansion $H_\e^\ell[\bar w]$ are total derivatives (with respect to time or space):
this is not a trivial fact for the last term, {as it is based on the possibility of constructing skew-symmetric flux correctors $\{\sigma^{n,0}\}_n$}. This happens to be crucial when applying Lemma~\ref{lem:apriori} in order to deduce an optimal $\Ld^2$ error estimate. This slightly refines the analysis of~\cite{ALR}.
\end{rem}

\subsection{Homogenized equations and secular growth problem}\label{sec:homog-eqn/hyp}
As motivated in Proposition~\ref{prop:form-hyperb}, cf.~\eqref{eq:motiv-homog-hypeq}, we consider the following formal homogenized equation, for $\ell\ge1$,
\begin{equation}\label{eq:homog-lim}
\left\{\begin{array}{ll}
\partial_t^2\bar W^\ell_\e-\nabla\cdot\big(\sum_{n=1}^\ell\sum_{m=0}^{\ell-n}\bar\Aa^{n,m}\odot(\e\nabla)^{n-1}(\e\partial_t)^m\big)\nabla \bar W^\ell_\e = f,&\text{in $\R\times\R^d$},\\
\bar W_\e^\ell=f=0,&\text{for $t<0$}.
\end{array}\right.
\end{equation}
However, this equation mixes higher-order space and time derivatives, and its well-posedness is problematic.
To avoid the secular growth problem described in Section~\ref{sec:discuss}, we follow~\cite{ALR}
and first show that the differential operator in~\eqref{eq:homog-lim} can be rewritten in such a way that it does no longer mix space and time derivatives. This is achieved by iteratively using the equation to eliminate time derivatives up to higher-order terms, which is referred to as the `criminal method' in~\cite{ALR}.
The proof is postponed to Section~\ref{sec:proof-reform}.
Note that the new homogenized coefficients $\{\bar\bb^n\}_n$ in this reformulation
automatically coincide with the coefficients given by the spectral approach: this can for instance be deduced a posteriori by comparing the corresponding homogenization results, Theorems~\ref{th:main-per} and~\ref{th:main-per2}; this allows us to use already here the same notation $\{\bar\bb^n\}_n$ as for spectral homogenized coefficients.
Note that by definition $\bar\bb^1=\bar\Aa^{1,0}=\bar\Aa$.

\begin{lem}[Revamped homogenized equation]\label{lem:reform}
Given $\ell\ge1$ and $\e>0$, if $\bar W_\e^\ell,f$ are smooth and satisfy the formal homogenized equation~\eqref{eq:homog-lim}, then we have
\[\partial_t^2\bar W_\e^\ell-\nabla\cdot\Big(\sum_{n=1}^\ell \bar\bb^{n}\odot(\e\nabla)^{n-1}\Big)\nabla\bar W_\e^\ell
\,=\, f+\e^2\nabla\cdot\Big(\sum_{n=1}^{\ell-2}\bar\cc^n\odot(\e D)^{n-1}\Big)\nabla f+\nabla \cdot E_\e^\ell,\]
where the coefficients $\{\bar\bb^n,\bar\cc^n\}_n$ and the remainder $E_\e^\ell$ are as follows:
\begin{enumerate}[$\bullet$]
\item We define the matrix-valued symmetric tensor $\bar\bb^p:=(\bar\bb^p_{j_1\ldots j_{p-1}})_{1\le j_1,\ldots,j_{p-1}\le d}$ for $p\ge1$
such that for all $\xi\in\R^d$,
\[\xi\cdot(\bar\bb^p\odot\xi^{\otimes(p-1)})\xi\,:=\,\sum_{k\ge1}\,\sum_{(m_1,\ldots,m_k)\in I_k}\sum_{n_1,\ldots, n_k\ge1\atop k+|n|=p+1}\prod_{j=1}^k\xi\cdot\big(\bar\Aa^{n_j,m_j}\odot\xi^{\otimes (n_j-1)}\big)\xi,\]
in terms of the index set
\[I_k\,:=\,\Big\{m=(m_1,\dots,m_{k}):m_j\ge0~\forall j,~\sum_{j=1}^sm_j\ge 2s~\forall s<k,~|m|=2(k-1)\Big\}.\]
\item We define the matrix-valued symmetric tensor $\bar\cc^p:=(\bar\cc^p_{j_1\ldots j_{p-1}})_{0\le j_1,\ldots,j_{p-1}\le d}$ for all $p\ge1$ such that for all $\hat\xi=(\xi_0,\xi)\in\R\times\R^d$,
\[\quad\xi\cdot(\bar\cc^p\odot\hat\xi^{\otimes(p-1)})\xi:=\sum_{k\ge1}\sum_{(m_1,\ldots,m_k)\in J_k}\sum_{n_1,\ldots,n_k\ge1\atop|n|+|m|=p+k+1}\xi_0^{|m|-2k}\prod_{j=1}^k\xi\cdot\big(\bar\Aa^{n_j,m_j}\odot\xi^{\otimes(n_j-1)}\big)\xi,\]
in terms of the index set
\[J_k\,:=\,\Big\{m=(m_1,\dots,m_{k}):m_j\ge0~\forall j,\,\sum_{j=1}^sm_j\ge 2s~\forall s\le k\Big\}.\]
\item For $\e\ll1$ small enough, the error term $E_\e^\ell$ satisfies pointwise, for all $r\ge0$,
\[|\langle D\rangle^r E_\e^\ell|\,\le\,(\e C \ell)^\ell\Big(\big|\langle D\rangle^{r+\ell}\langle  \e C D\rangle^{\ell^2}D\bar W_\e^\ell\big|+
\big|\langle D\rangle^{r+\ell-1}\langle \e C D\rangle^{\ell^2} f\big|\Big).\]
\end{enumerate}
In particular, Lemma~\ref{lem:per-cor} implies for all $p\ge1$,
\begin{equation}\label{eq:growthrevamped}
|\bar{\bb}^p|+|\bar{\cc}^p|\leq C^p,
\end{equation}
and Proposition~\ref{prop:even-coeff} entails, whenever $p$ is an even integer,
\[\xi\cdot (\bar\bb^p\odot \xi^{\otimes(p-1)})\xi=0,\qquad\text{for all $\xi\in\R^d$.}\qedhere\]
\end{lem}

The above thus leads us to considering the following modified version of the formal effective equation~\eqref{eq:homog-lim},
\begin{equation}\label{eq:homog-lim/ref}
\left\{\begin{array}{ll}
\partial_t^2\bar V^\ell_\e-\nabla\cdot\big(\bar\Aa+\sum_{k=2}^\ell\bar\bb^k\odot(\e\nabla)^{k-1}\big)\nabla\bar V_\e^\ell&\\
\hspace{3cm}=f+\e^2\nabla\cdot\big(\sum_{k=1}^{\ell-2}\bar\cc^k\odot(\e D)^{k-1}\big)\nabla f,&\text{in $\R\times\R^d$},\\
\bar V_\e^\ell=f=0,&\text{for $t<0$}.
\end{array}\right.
\end{equation}
As with the spectral approach, the symbol of the operator $-\nabla\cdot\big(\bar\Aa+\sum_{k=2}^\ell\bar\bb^k\odot(\e\nabla)^{k-1}\big)\nabla$ lacks positivity, so this equation is ill-posed in general.
To cure this issue, we argue exactly as in Section~\ref{sec:homog-eqn/spec}: we can consider several well-posed high{er}-order modifications of this equation, either by high-frequency filtering, by high{er}-order regularization, or by the Boussinesq trick, and we denote by $\bar v_\e^{\text{(I)},\ell},\bar v_\e^{\text{(II)},\ell},\bar v_\e^{\text{(III)},\ell}$ the corresponding solutions, respectively.
We refer to Lemma~\ref{lem:apriori-baru-sp} for the details (up to replacing the impulse $f$ in~\eqref{eq:homog-lim-sp} by the specific right-hand side in~\eqref{eq:homog-lim/ref}).

Next, we show that the above modification procedure $\bar W_\e^\ell\mapsto\bar V_\e^\ell$ for the formal effective equation can be inverted: more precisely, the solution $\bar v_\e^{(\star),\ell}$ of any of the well-posed modifications of~\eqref{eq:homog-lim/ref} is an approximate solution of the formal effective equation~\eqref{eq:homog-lim} up to an~$O(\e^\ell)$ error.

\begin{lem}[Inversion procedure]\label{lem:invert}
Given $\ell\ge1$ and $\e>0$, if $\bar v_\e^{(\star),\ell}$  is the solution of one of the well-posed modifications of equation~\eqref{eq:homog-lim/ref} as given by Lemma~\ref{lem:apriori-baru-sp},
then we have
\[\partial_t^2\bar v^{(\star),\ell}_\e-\nabla\cdot\Big(\sum_{n=1}^\ell\sum_{m=0}^{\ell-n}\bar\Aa^{n,m}\odot(\e\nabla)^{n-1}(\e\partial_t)^m\Big)\nabla \bar v^{(\star),\ell}_\e = f+\nabla \cdot F_\e^{(\star),\ell},\]
where the error $\nabla \cdot F_\e^{(\star),\ell}$ satisfies for all $r\ge0$,
\begin{equation*}
\|\langle D\rangle^rF_{\e}^{(\star),\ell;t}\|_{\Ld^2(\R^d)}\,\le\,
(\e C\ell)^\ell
\|\langle D\rangle^{r+C\ell}\langle \e C D\rangle^{\ell^2}f \|_{\Ld^1((0,t);\Ld^2(\R^d))},
\end{equation*}
where the constant $C$ further depends on the choice of~$\alpha$ in case $(\star)=\operatorname{(I)}$.
\end{lem}

\begingroup\allowdisplaybreaks

\subsection{Proof of Theorem~\ref{th:main-per2}}
Applying Proposition~\ref{prop:form-hyperb} with $\bar w=\bar v^{(\star),\ell}_\e$,
appealing to Lemma~\ref{lem:invert}, and comparing with the solution $u_\e$ of the heterogeneous wave equation~\eqref{eq:hyperb}, we find
\begin{multline*}
\big(\partial_t^2-\nabla\cdot\Aa(\tfrac\cdot\e)\nabla\big)(u_\e-H^\ell_\e[\bar v_\e^{(\star),\ell}])\,=\,
- \nabla \cdot F_\e^{(\star),\ell}\\
+\e^{\ell}\sum_{n=1}^{\ell+1}\nabla\cdot\Big(\big(\Aa\phi_{j_1\ldots j_{n-1}}^{n-1,\ell-n+1}-\sigma^{n-1,\ell-n+1}_{j_1\ldots j_{n-1}}\big)(\tfrac\cdot\e)\nabla\nabla^{n-1}_{j_1\ldots j_{n-1}}\partial_t^{\ell-n+1}\bar v_\e^{(\star),\ell}\Big)\\
-\e^\ell\sum_{n=0}^{\ell}\partial_t \Big(\big(\phi^{n+1,\ell-n-1}_{j_1\ldots j_{n-1}ji}\ee_{j}\otimes\ee_{i}-\sigma^{n-1,\ell-n+1}_{j_1\ldots j_{n-1}}\big)(\tfrac\cdot\e):\nabla^2\nabla^{n-1}_{j_1\ldots j_{n-1}}\partial_t^{\ell-n}\bar v_\e^{(\star),\ell} \Big).
\end{multline*}
By the a priori estimates of~Lemma~\ref{lem:apriori}, using corrector estimates of Lemma~\ref{lem:per-cor}, and using the Sobolev embedding to estimate products with correctors as in~\eqref{eq:use-Sobolev-cor}, we get for $a>\frac d2$,
\begin{multline*}
\|u_\e^t-H_\e^\ell[\bar v_\e^{(\star),\ell;t}]\|_{\Ld^2(\R^d)}+\|D(u_\e^t-H_\e^\ell[\bar v_\e^{(\star),\ell;t}])\|_{\Ld^2(\R^d)}\\
\,\lesssim\,\|\langle D\rangle F_\e^{(\star),\ell}\|_{\Ld^1((0,t);\Ld^2(\R^d))}+(\e C)^\ell\|\langle D\rangle^{\ell+a+1}D\bar v_\e^{(\star),\ell}\|_{\Ld^1((0,t);\Ld^2(\R^d))}.
\end{multline*}
Combining {this bound} with the estimate of Lemma~\ref{lem:invert} on the remainder $F_\e^\ell$, and with the a priori estimates of Lemma~\ref{lem:apriori-baru-sp}(ii) for $\bar v_\e^{(\star),\ell}$, the conclusion follows.
\qed

\subsection{Proof of Proposition~\ref{prop:even-coeff}}\label{sec:even-coeff}
We split the proof into two steps.

\medskip
\step1 Migration process: proof that for all $n,m,p,q\ge0$,
\begin{multline}\label{eq:migrate}
\expec{\big(\nabla\phi^{p,q}_{j_1\ldots j_p}\cdot\Aa\nabla\phi^{n,m}_{i_1\ldots i_n}-\phi^{p-1,q}_{j_1\ldots j_{p-1}}\ee_{j_p}\cdot\Aa\phi^{n-1,m}_{i_1\ldots i_{n-1}}\ee_{i_n}\big)}\\
=-\expec{\big(\nabla\phi^{p+1,q}_{j_1\ldots j_{p}i_n}\cdot\Aa\nabla\phi^{n-1,m}_{i_1\ldots i_{n-1}}-\phi^{p,q}_{j_1\ldots j_{p}}\ee_{i_n}\cdot\Aa\phi^{n-2,m}_{i_1\ldots i_{n-2}}\ee_{i_{n-1}}\big)}\\
-\expec{\big(\phi^{p,q}_{j_1\ldots j_p}\phi^{n,m-2}_{i_1\ldots i_n}+\phi^{p+1,q-2}_{j_1\ldots j_pi_n}\phi^{n-1,m}_{i_1\ldots i_{n-1}}\big)},
\end{multline}
and in addition, for all $n,m,q\ge0$,
\begin{multline}\label{eq:migrate-re}
\sum_{k=1}^n(-1)^k\expec{\phi^{k,q}_{ji_n\ldots i_{n-k+2}}\phi^{n-k+1,m}_{i_1\ldots i_{n-k+1}}}=\sum_{k=1}^{n}(-1)^k\expec{\phi^{k,q+2}_{ji_n\ldots i_{n-k+2}}\phi^{n-k+1,m-2}_{i_1\ldots i_{n-k+1}}}\\
+(-1)^n\expec{\nabla\phi^{n,q+2}_{ji_n\ldots i_{2}}\cdot\Aa\nabla\phi^{0,m}\ee_{i_1}+\phi^{n-1,q+2}_{ji_n\ldots i_{3}}\ee_{i_{2}}\cdot\Aa\phi^{0,m}\ee_{i_{1}}}.
\end{multline}
The equation for $\phi^{n,m}$ in Definition~\ref{defi:BC} yields
\begin{multline*}
\expec{\nabla\phi^{p,q}_{j_1\ldots j_{p}}\cdot\Aa\nabla\phi^{n,m}_{i_1\ldots i_n}}\,=\,-\expec{\nabla\phi^{p,q}_{j_1\ldots j_{p}}\cdot\Aa\phi^{n-1,m}_{i_1\ldots i_{n-1}}\ee_{i_n}}\\
+\expec{\phi^{p,q}_{j_1\ldots j_{p}}\ee_{i_n}\cdot\Aa\big(\nabla\phi^{n-1,m}_{i_1\ldots i_{n-1}}+\phi^{n-2,m}_{i_1\ldots i_{n-2}}\ee_{i_{n-1}}\big)}
-\expec{\phi^{p,q}_{j_1\ldots j_{p}}\phi^{n,m-2}_{i_1\ldots i_n}},
\end{multline*}
while the equation for $\phi^{{p+1},q}$ leads to
\begin{multline*}
\expec{\phi^{p,q}_{j_1\ldots j_p}\ee_{i_n}\cdot\Aa\nabla\phi^{n-1,m}_{i_1\ldots i_{n-1}}}\,=\,-\expec{\nabla\phi^{p+1,q}_{j_1\ldots j_pi_n}\cdot\Aa\nabla\phi^{n-1,m}_{i_1\ldots i_{n-1}}}\\
+\expec{\big(\nabla\phi^{p,q}_{j_1\ldots j_p}+\phi^{p-1,q}_{j_1\ldots j_{p-1}}\ee_{j_p}\big)\cdot\Aa\phi^{n-1,m}_{i_1\ldots i_{n-1}}\ee_{i_n}}
-\expec{\phi^{p+1,q-2}_{j_1\ldots j_pi_n}\phi^{n-1,m}_{i_1\ldots i_{n-1}}}.
\end{multline*}
Summing these two identities, the claim~\eqref{eq:migrate} easily follows.
Next, using~\eqref{eq:migrate} in form of
\begin{multline*}
\expec{\phi^{k,q}_{ji_n\ldots i_{n-k+2}}\phi^{n-k+1,m}_{i_1\ldots i_{n-k+1}}}=-\expec{\phi^{k-1,q+2}_{ji_n\ldots i_{n-k+3}}\phi^{n-k+2,m-2}_{i_1\ldots i_{n-k+2}}}\\
-\expec{\big(\nabla\phi^{k-1,q+2}_{ji_n\ldots i_{n-k+3}}\cdot\Aa\nabla\phi^{n-k+2,m}_{i_1\ldots i_{n-k+2}}-\phi^{k-2,q+2}_{ji_n\ldots j_{n-k+4}}\ee_{j_{n-k+3}}\cdot\Aa\phi^{n-k+1,m}_{i_1\ldots i_{n-k+1}}\ee_{i_{n-k+2}}\big)}\\
-\expec{\big(\nabla\phi^{k,q+2}_{ji_n\ldots i_{n-k+2}}\cdot\Aa\nabla\phi^{n-k+1,m}_{i_1\ldots i_{n-k+1}}-\phi^{k-1,q+2}_{ji_n\ldots i_{n-k+3}}\ee_{i_{n-k+2}}\cdot\Aa\phi^{n-k,m}_{i_1\ldots i_{n-k}}\ee_{i_{n-k+1}}\big)},
\end{multline*}
and summing this identity for $1\le k\le n$, we find after straightforward simplifications
\begin{multline*}
\sum_{k=1}^n(-1)^k\expec{\phi^{k,q}_{ji_n\ldots i_{n-k+2}}\phi^{n-k+1,m}_{i_1\ldots i_{n-k+1}}}=-\sum_{k=1}^n(-1)^k\expec{\phi^{k-1,q+2}_{ji_n\ldots i_{n-k+3}}\phi^{n-k+2,m-2}_{i_1\ldots i_{n-k+2}}}\\
-(-1)^n\expec{\big(\nabla\phi^{n,q+2}_{ji_n\ldots i_{2}}\cdot\Aa\nabla\phi^{1,m}_{i_1}-\phi^{n-1,q+2}_{ji_n\ldots i_{3}}\ee_{i_{2}}\cdot\Aa\phi^{0,m}\ee_{i_{1}}\big)}.
\end{multline*}
Using the equation for $\phi^{1,m}$, the second claim~\eqref{eq:migrate-re} easily follows.

\medskip
\step2 Conclusion.\\
For $n\ge1$ and $m\ge0$, the definition of $\bar\Aa^{n,m}$ and the equation for $\phi^{1,0}$ yield
\[\ee_j\cdot\bar\Aa^{n,m}_{i_1\ldots i_{n-1}}\ee_{i_n}
\,=\,-\expec{\big(\nabla\phi^{1,0}_j\cdot\Aa\nabla\phi^{n,m}_{i_1\ldots i_{n}}-\ee_j\cdot\Aa\phi^{n-1,m}_{i_1\ldots i_{n-1}}\ee_{i_n}\big)}.\]
Iterating identity~\eqref{eq:migrate} then leads to
\begin{multline*}
\ee_j\cdot\bar\Aa^{n,m}_{i_1\ldots i_{n-1}}\ee_{i_n}\,=\,(-1)^n\expec{\big(\nabla\phi^{n,0}_{ji_n\ldots i_2}\cdot\Aa\nabla\phi^{1,m}_{i_1}-\phi^{n-1,0}_{ji_n\ldots i_3}\ee_{i_2}\cdot\Aa\phi^{0,m}\ee_{i_1}\big)}\\
-\sum_{k=1}^{n-1}(-1)^k\expec{\phi^{k,0}_{ji_n\ldots i_{n-k+2}}\phi^{n-k+1,m-2}_{i_1\ldots i_{n-k+1}}},
\end{multline*}
hence, using the equation for $\phi^{1,m}$,
\begin{multline}\label{eq:coeff-pre}
\ee_j\cdot\bar\Aa^{n,m}_{i_1\ldots i_{n-1}}\ee_{i_n}\,=\,(-1)^{n+1}\expec{\big(\nabla\phi^{n,0}_{ji_n\ldots i_2}+\phi^{n-1,0}_{ji_n\ldots i_3}\ee_{i_2}\big)\cdot\Aa\phi^{0,m}\ee_{i_1}}\\
-\sum_{k=1}^{n}(-1)^k\expec{\phi^{k,0}_{ji_n\ldots i_{n-k+2}}\phi^{n-k+1,m-2}_{i_1\ldots i_{n-k+1}}}.
\end{multline}
For $m=0$, this already yields the conclusion
\[\ee_j\cdot\bar\Aa^{n,0}_{i_1\ldots i_{n-1}}\ee_{i_n}\,=\,(-1)^{n+1}\expec{\big(\nabla\phi^{n,0}_{ji_n\ldots i_2}+\phi^{n-1,0}_{ji_n\ldots i_3}\ee_{i_2}\big)\cdot\Aa\ee_{i_1}}\,=\,(-1)^{n+1}\ee_{i_1}\cdot\bar\Aa^{n,0}_{ji_n\ldots i_3}\ee_{i_2}.\]
For $m\ge2$, identity~\eqref{eq:coeff-pre} rather takes the form
\begin{equation}\label{eq:ident-barAnm-touse}
\ee_j\cdot\bar\Aa^{n,m}_{i_1\ldots i_{n-1}}\ee_{i_n}\,=\,-\sum_{k=1}^{n}(-1)^k\expec{\phi^{k,0}_{ji_n\ldots i_{n-k+2}}\phi^{n-k+1,m-2}_{i_1\ldots i_{n-k+1}}},
\end{equation}
which we shall combine with an iterative use of identity~\eqref{eq:migrate-re}. More precisely, in case $m=4m'+2$ with an integer $m'\ge0$, iterating identity~\eqref{eq:migrate-re} (starting from $q=0$ with $m$ replaced by $m-2=4m'$), and recalling $\phi^{0,l}=0$ for $l\ge1$, we find
\begin{equation*}
\sum_{k=1}^n(-1)^k\expec{\phi^{k,0}_{ji_n\ldots i_{n-k+2}}\phi^{n-k+1,m-2}_{i_1\ldots i_{n-k+1}}}=\sum_{k=1}^{n}(-1)^k\expec{\phi^{k,2m'}_{ji_n\ldots i_{n-k+2}}\phi^{n-k+1,2m'}_{i_1\ldots i_{n-k+1}}},
\end{equation*}
and thus, combining this with~\eqref{eq:ident-barAnm-touse}, and replacing $k$ by $n-k+1$ in the sum,
\[\ee_j\cdot\bar\Aa^{n,m}_{i_1\ldots i_{n-1}}\ee_{i_n}\,=\,-\sum_{k=1}^{n}(-1)^k\expec{\phi^{k,2m'}_{ji_n\ldots i_{n-k+2}}\phi^{n-k+1,2m'}_{i_1\ldots i_{n-k+1}}}\,=\,(-1)^{n+1}\ee_{i_1}\cdot\bar\Aa^{n,m}_{ji_n\ldots i_{3}}\ee_{i_2}.\]
Similarly, in the case $m=4m'$ with integer $m'\ge1$, we find
\begin{eqnarray*}
\ee_j\cdot\bar\Aa^{n,m}_{i_1\ldots i_{n-1}}\ee_{i_n}
&=&-\sum_{k=1}^{n}(-1)^k\expec{\phi^{k,2m'}_{ji_n\ldots i_{n-k+2}}\phi^{n-k+1,2m'-2}_{i_1\ldots i_{n-k+1}}}\\
&=&-\sum_{k=1}^{n}(-1)^k\expec{\phi^{k,2m'-2}_{ji_n\ldots i_{n-k+2}}\phi^{n-k+1,2m'}_{i_1\ldots i_{n-k+1}}}\\
&=&(-1)^{n+1}\ee_{i_1}\cdot\bar\Aa^{n,m}_{ji_n\ldots i_{3}}\ee_{i_2},
\end{eqnarray*}
and the conclusion follows.
\qed

\subsection{Proof of Proposition~\ref{prop:form-hyperb}}\label{sec:form-hyperb}
We focus on the case $\e=1$ and drop it from all subscripts in the notation, while the final result is obtained after $\e$-rescaling.
Since the correctors~$\phi_{j_1\ldots j_{n}}^{n,m}$ do not depend on time, a direct calculation yields for all $n,m\ge0$,
\begin{multline*}
(\partial_t^2-\nabla\cdot\Aa\nabla)\big(\phi^{n,m}\odot\nabla^n\partial_t^m\bar w\big)\,=\,\phi^{n,m}\odot\nabla^n\partial_t^{m+2}\bar w\\
+(-\nabla\cdot\Aa\nabla\phi^{n,m}_{j_1\ldots j_n})\nabla^n_{j_1\ldots j_n}\partial_t^m\bar w
-\nabla\cdot(\Aa\phi^{n,m}_{j_1\ldots j_n}\ee_{j_{n+1}})\nabla^{n+1}_{j_1\ldots j_{n+1}}\partial_t^m\bar w\\
-(\ee_{j_{n+1}}\cdot\Aa\nabla\phi^{n,m}_{j_1\ldots j_n})\nabla^{n+1}_{j_1\ldots j_{n+1}}\partial_t^m\bar w
-(\ee_{j_{n+2}}\cdot\Aa\phi^{n,m}_{j_1\ldots j_n}\ee_{j_{n+1}})\nabla^{n+2}_{j_1\ldots j_{n+2}}\partial_t^m\bar w.
\end{multline*}
Inserting the defining equation for the hyperbolic corrector $\phi^{n,m}$, cf.~Definition~\ref{defi:BC}, we get for all $n\ge1$ and $m\ge0$,
\begin{multline*}
(\partial_t^2-\nabla\cdot\Aa\nabla)\big(\phi^{n,m}\odot\nabla^n\partial_t^m\bar w\big)\,=\,\phi^{n,m}\odot\nabla^n\partial_t^{m+2}\bar w-\phi^{n,m-2}\odot\nabla^n\partial_t^m\bar w\\
+\nabla\cdot(\Aa\phi^{n-1,m}_{j_1\ldots j_{n-1}}e_{j_n})\nabla^n_{j_1\ldots j_n}\partial_t^m\bar w
-\nabla\cdot(\Aa\phi^{n,m}_{j_1\ldots j_n}\ee_{j_{n+1}})\nabla^{n+1}_{j_1\ldots j_{n+1}}\partial_t^m\bar w\\
+(e_{j_n}\cdot\Aa\nabla\phi_{j_1\ldots j_{n-1}}^{n-1,m})\nabla^n_{j_1\ldots j_n}\partial_t^m\bar w
-(\ee_{j_{n+1}}\cdot\Aa\nabla\phi^{n,m}_{j_1\ldots j_n})\nabla^{n+1}_{j_1\ldots j_{n+1}}\partial_t^m\bar w\\
+(e_{j_n}\cdot\Aa\phi_{j_1\ldots j_{n-2}}^{n-2,m}e_{j_{n-1}})\nabla^n_{j_1\ldots j_n}\partial_t^m\bar w
-(\ee_{j_{n+2}}\cdot\Aa\phi^{n,m}_{j_1\ldots j_n}\ee_{j_{n+1}})\nabla^{n+2}_{j_1\ldots j_{n+2}}\partial_t^m\bar w\\
-\bar\Aa^{n-1,m}_{j_1\ldots j_{n-2}}:\nabla^2\nabla^{n-2}_{j_1\ldots j_{n-2}}\partial_t^m\bar w,
\end{multline*}
and thus, after summation over $1\le n\le\ell$ and $0\le m\le\ell-n$, recalling the definition~\eqref{eq:def-Fell} of the geometric two-scale expansion, and using $\phi^{0,0}=1$ and $\phi^{0,m}=0$ for $m>0$,
\begin{multline*}
(\partial_t^2-\nabla\cdot\Aa\nabla)H^\ell[\bar w]\,=\,
\partial_t^2\bar w-\sum_{n=1}^\ell\sum_{m=0}^{\ell-n}\bar\Aa^{n-1,m}_{j_1\ldots j_{n-2}}:\nabla^2\nabla^{n-2}_{j_1\ldots j_{n-2}}\partial_t^m\bar w\\
+\sum_{n=1}^\ell\Big(\phi^{n,\ell-n-1}\odot\nabla^n\partial_t^{\ell+1-n}\bar w
+\phi^{n,\ell-n}\odot\nabla^n\partial_t^{\ell+2-n}\bar w\Big)\\
-\sum_{n=1}^{\ell}\Big(\nabla\cdot(\Aa\phi_{j_1\ldots j_{n}}^{n,\ell-n}e_{j_{n+1}})+e_{j_{n+1}}\cdot\Aa\nabla\phi_{j_1\ldots j_{n}}^{n,\ell-n}\Big)\nabla^{n+1}_{j_1\ldots j_{n+1}}\partial_t^{\ell-n}\bar w\\
-\sum_{n=1}^{\ell}\Big((e_{j_{n+1}}\cdot\Aa \phi^{n-1,\ell-n}_{j_1\ldots j_{n-1}}e_{j_{n}})\nabla^{n+1}_{j_1\ldots j_{n+1}}\partial_t^{\ell-n}\bar w+(e_{j_{n+2}}\cdot\Aa \phi^{n,\ell-n}_{j_1\ldots j_n}e_{j_{n+1}})\nabla^{n+2}_{j_1\ldots j_{n+2}}\partial_t^{\ell-n}\bar w\Big).
\end{multline*}
As $\bar\Aa^{0,m}=0$ for all $m\ge0$, the second right-hand side term can be rewritten as
\begin{multline*}
\sum_{n=1}^{\ell}\sum_{m=0}^{\ell-n}\bar\Aa_{j_1\ldots j_{n-2}}^{n-1,m}:\nabla^2\nabla^{n-2}_{j_1\ldots j_{n-2}}\partial_t^m\bar w
\,=\,\sum_{n=1}^{\ell-1}\sum_{m=0}^{\ell-1-n}\bar\Aa_{j_1\ldots j_{n-1}}^{n,m}:\nabla^2\nabla^{n-1}_{j_1\ldots j_{n-1}}\partial_t^m\bar w\\
=\sum_{n=1}^{\ell}\sum_{m=0}^{\ell-n}\nabla\cdot(\bar\Aa^{n,m}\odot\nabla^{n-1}\partial_t^m)\nabla\bar w
-\sum_{n=1}^{\ell}\bar\Aa_{j_1\ldots j_{n-1}}^{n,\ell-n}:\nabla^2\nabla^{n-1}_{j_1\ldots j_{n-1}}\partial_t^{\ell-n}\bar w.
\end{multline*}
Further rearranging the terms,
and noting that
\begin{multline*}
\nabla\cdot(\Aa\phi_{j_1\ldots j_{n}}^{n,\ell-n}e_{j_{n+1}})\nabla^{n+1}_{j_1\ldots j_{n+1}}\partial_t^{\ell-n}\bar w+(e_{j_{n+2}}\cdot\Aa \phi^{n,\ell-n}_{j_1\ldots j_n}e_{j_{n+1}})\nabla^{n+2}_{j_1\ldots j_{n+2}}\partial_t^{\ell-n}\bar w\\
\,=\,\nabla\cdot\Big(\Aa\phi_{j_1\ldots j_{n}}^{n,\ell-n}\nabla\nabla^{n}_{j_1\ldots j_{n}}\partial_t^{\ell-n}\bar w\Big),
\end{multline*}
we are led to
\begin{multline}\label{eq:pre-eqn-Hell}
(\partial_t^2-\nabla\cdot\Aa\nabla)H^\ell[\bar w]\,=\,
\partial_t^2\bar w-\sum_{n=1}^{\ell}\sum_{m=0}^{\ell-n}\nabla\cdot(\bar\Aa^{n,m}\odot\nabla^{n-1}\partial_t^m)\nabla\bar w\\
+\sum_{n=1}^\ell\Big(\phi^{n,\ell-n-1}\odot\nabla^n\partial_t^{\ell+1-n}\bar w
+\phi^{n,\ell-n}\odot\nabla^n\partial_t^{\ell+2-n}\bar w\Big)
\\
-\sum_{n=1}^{\ell}\nabla\cdot\Big(\Aa\phi_{j_1\ldots j_{n}}^{n,\ell-n}\nabla\nabla^{n}_{j_1\ldots j_{n}}\partial_t^{\ell-n}\bar w\Big)\\
-\sum_{n=1}^{\ell}\Big(\Aa\big(\nabla\phi_{j_1\ldots j_{n}}^{n,\ell-n}+\phi^{n-1,\ell-n}_{j_1\ldots j_{n-1}}e_{j_{n}}\big)-\bar\Aa_{j_1\ldots j_{n-1}}^{n,\ell-n}\ee_{j_{n}}\Big)\cdot\nabla\nabla^{n}_{j_1\ldots j_{n}}\partial_t^{\ell-n}\bar w.
\end{multline}
The last right-hand side terms can be reformulated in terms of fluxes, cf.~Definition~\ref{defi:BC},
\begin{multline*}
{\sum_{n=1}^{\ell}\Big(\Aa\big(\nabla\phi_{j_1\ldots j_{n}}^{n,\ell-n}+\phi^{n-1,\ell-n}_{j_1\ldots j_{n-1}}e_{j_{n}}\big)-\bar\Aa_{j_1\ldots j_{n-1}}^{n,\ell-n}\ee_{j_{n}}\Big)\cdot\nabla\nabla^{n}_{j_1\ldots j_{n}}\partial_t^{\ell-n}\bar w}\\
\,=\,\sum_{n=1}^{\ell}q^{n,\ell-n}_{j_1\ldots j_n}\cdot\nabla\nabla^{n}_{j_1\ldots j_{n}}\partial_t^{\ell-n}\bar w+\sum_{n=1}^{\ell-1}\phi^{n+1,\ell-n-2}\odot\nabla^{n+1}\partial_t^{\ell-n}\bar w.
\end{multline*}
Further noting that we can write
\[q^{\ell,0}_{j_1\ldots j_\ell}\cdot\nabla\nabla^{\ell}_{j_1\ldots j_{\ell}}\bar w\,=\,\tilde q^{\ell,0}_{j_1\ldots j_\ell}\cdot\nabla\nabla^{\ell}_{j_1\ldots j_{\ell}}\bar w,\]
in terms of the modified fluxes of Lemma~\ref{lem:mod-cor},
and then inserting the definition of hyperbolic flux correctors {and using the skew-symmetry of $\sigma^{\ell,0}_{j_1\ldots j_\ell}$,} cf.~Definition~\ref{def:cocor} and Lemma~\ref{lem:mod-cor}, we deduce
\begin{multline*}
{\sum_{n=1}^{\ell}\Big(\Aa\big(\nabla\phi_{j_1\ldots j_{n}}^{n,\ell-n}+\phi^{n-1,\ell-n}_{j_1\ldots j_{n-1}}e_{j_{n}}\big)-\bar\Aa_{j_1\ldots j_{n-1}}^{n,\ell-n}\ee_{j_{n}}\Big)\cdot\nabla\nabla^{n}_{j_1\ldots j_{n}}\partial_t^{\ell-n}\bar w}\\
\,=\,-\sum_{n=0}^{\ell}\nabla_j (\sigma^{n,\ell-n}_{j_1\ldots j_n})_{jj_{n+1}}\nabla^{n+1}_{j_1\ldots j_{n+1}}\partial_t^{\ell-n}\bar w+\sum_{n=0}^{\ell-1}\phi^{n+1,\ell-n-2}\odot\nabla^{n+1}\partial_t^{\ell-n}\bar w.
\end{multline*}
Rearranging the terms and further using the skew-symmetry of $\sigma^{\ell,0}_{j_1\ldots j_\ell}$, we get
\begin{multline*}
{\sum_{n=1}^{\ell}\Big(\Aa\big(\nabla\phi_{j_1\ldots j_{n}}^{n,\ell-n}+\phi^{n-1,\ell-n}_{j_1\ldots j_{n-1}}e_{j_{n}}\big)-\bar\Aa_{j_1\ldots j_{n-1}}^{n,\ell-n}\ee_{j_{n}}\Big)\cdot\nabla\nabla^{n}_{j_1\ldots j_{n}}\partial_t^{\ell-n}\bar w}\\
\,=\,-\sum_{n=0}^{\ell}\nabla\cdot\Big( \sigma^{n,\ell-n}_{j_1\ldots j_n}\nabla\nabla^{n}_{j_1\ldots j_{n}}\partial_t^{\ell-n}\bar w\Big)
+\sum_{n=0}^{\ell-1}\sigma^{n,\ell-n}_{j_1\ldots j_n}:\nabla^2\nabla^{n}_{j_1\ldots j_{n}}\partial_t^{\ell-n}\bar w\\
+\sum_{n=0}^{\ell-1}\phi^{n+1,\ell-n-2}\odot\nabla^{n+1}\partial_t^{\ell-n}\bar w.
\end{multline*}
Inserting this into~\eqref{eq:pre-eqn-Hell}
yields the conclusion.\qed

\subsection{Proof of Lemma~\ref{lem:reform}}\label{sec:proof-reform}
We split the proof into three steps.

\medskip
\step1
Proof that for all $p\ge1$ there holds
\begin{multline}\label{eq:claim_onlyspatial}
\partial_t^2\bar W^{\ell}_\e-\sum_{k=1}^{p}\e^{2(k-1)}\sum_{n\in[\ell]^k}\sum_{m\in I^{n}_k}\Big(\prod_{j=1}^{k}\nabla\cdot\big(\bar\Aa^{n_j,m_j}\odot(\e\nabla)^{n_j-1}\big)\nabla\Big)\bar W_\e^{\ell}\\
\,=\,f+\sum_{k=1}^{p-1}\sum_{n\in[\ell]^k}\sum_{m\in J^{n}_k}\e^{|m|}\Big(\prod_{j=1}^k\nabla\cdot\big(\bar\Aa^{n_j,m_j}\odot(\e\nabla)^{n_j-1}\big)\nabla\Big)\partial_t^{|m|-2k} f\\
+\sum_{n\in[\ell]^p}\sum_{m\in J^{n}_p}\e^{|m|}\Big(\prod_{j=1}^{p}\nabla\cdot\big(\bar\Aa^{n_j,m_j}\odot(\e\nabla)^{n_j-1}\big)\nabla\Big)\partial_t^{|m|-2(p-1)}\bar W_\e^{\ell},
\end{multline}
where we have defined the following sets of indices, for all $k\ge1$ and $n=(n_1,\ldots,n_k)\in[\ell]^k:=\{1,\dots,\ell\}^k$,
\begin{align*}
I^n_k&=\Big\{m=(m_1,\dots,m_{k}):0\le m_j\le\ell-n_j~\forall j,\,\sum_{j=1}^sm_j\ge 2s~\forall s<k,\,|m|=2(k-1)\Big\},
\\
J^{n}_k&=\Big\{m=(m_1,\dots,m_k):0\le m_j\le\ell-n_j~\forall j,\,\sum_{j=1}^sm_j\ge 2s~\forall s\le k\Big\}.
\end{align*}
We argue by induction. For $p=1$, the stated identity~\eqref{eq:claim_onlyspatial} reduces to
\begin{equation*}
\partial_t^2\bar{W}_\e^{\ell}-\nabla\cdot\Big(\sum_{n=1}^\ell\bar\Aa^{n,0}\odot(\e\nabla)^{n-1}\Big)\nabla\bar W_\e^{\ell}
\,=\,f+\nabla\cdot\Big(\sum_{n=1}^\ell\sum_{m=2}^{\ell-n}\bar\Aa^{n,m}\odot(\e\nabla)^{n-1}(\e\partial_t)^{m}\Big)\nabla\bar W_\e^{\ell},
\end{equation*}
which is a simple reformulation of \eqref{eq:homog-lim}, keeping in mind that $\bar\Aa^{n,m}=0$ whenever $m$ is odd, cf.~Definition~\ref{defi:BC}.
Next, we assume that the claim~\eqref{eq:claim_onlyspatial} holds for some $p\ge1$ and we prove that it then also holds at level $p+1$. Let $n\in[\ell]^{p}$ and $m\in J_p^n$ be momentarily fixed. As by definition $|m|\ge2p$, we may use~\eqref{eq:homog-lim} in the form
\begin{equation*}
\partial_t^{|m|-2(p-1)}\bar{W}_\e^{\ell}=\partial_t^{|m|-2p}\bigg(f+\nabla\cdot\Big(\sum_{n'=1}^{\ell}\sum_{m'=0}^{\ell-n'}\bar{\Aa}^{n',m'}\odot(\e\nabla)^{n'-1}(\e\partial_t)^{m'}\Big)\nabla\bar{W}_\e^{\ell}\bigg).
\end{equation*}
Note that for $1\leq n'\leq \ell$ and $0\leq m'\leq \ell-n'$ we have the equivalences
\begin{align*}
(m,m')~\in~\left\{\begin{array}{lll}
I_{p+1}^{(n,n')}&\iff& |m|=2p,\,m'=0,\\
J_{p+1}^{(n,n')}&\iff& |m|+m'\ge 2(p+1).
\end{array}\right.
\end{align*}
Using these observations, the last term in~\eqref{eq:claim_onlyspatial} can be decomposed as
\begin{multline*}
\sum_{n\in[\ell]^p}\sum_{m\in J^{n}_p}\e^{|m|}\Big(\prod_{j=1}^{p}\nabla\cdot\big(\bar\Aa^{n_j,m_j}\odot(\e\nabla)^{n_j-1}\big)\nabla\Big)\partial_t^{|m|-2(p-1)}\bar W_\e^{\ell}\\
\,=\,\sum_{n\in[\ell]^p}\sum_{m\in J^{n}_p}\e^{|m|}\Big(\prod_{j=1}^{p}\nabla\cdot\big(\bar\Aa^{n_j,m_j}\odot(\e\nabla)^{n_j-1}\big)\nabla\Big) \partial_t^{|m|-2p} f\\
+\e^{2p}\sum_{n\in[\ell]^{p+1}}\sum_{m\in {I}^{n}_{p+1}}\Big(\prod_{j=1}^{p+1}\nabla\cdot\big(\bar\Aa^{n_j,m_j}\odot(\e\nabla)^{n_j-1}\big)\nabla\Big)\bar{W}_\e^{\ell}\\
+\sum_{n\in[\ell]^{p+1}}\sum_{m\in J^{n}_{p+1}}\e^{|m|}\Big(\prod_{j=1}^{p+1}\nabla\cdot\big(\bar\Aa^{n_j,m_j}\odot(\e\nabla)^{n_j-1}\big)\nabla\Big)\partial_t^{|m|-2p}\bar{W}_\e^{\ell}.
\end{multline*}
Inserting this expression into the induction assumption~\eqref{eq:claim_onlyspatial} yields the conclusion.

\medskip
\step2 Reformulation.\\
The definition of the coefficients $\{\bar\bb^n\}_n$ in the statement leads to
\begin{eqnarray*}
\lefteqn{\nabla\cdot\Big(\sum_{n=1}^{\ell}\bar\bb^n\odot(\e\nabla)^{n-1}\Big)\nabla}\\
&=&\sum_{r=1}^{\ell}\sum_{k=1}^{\lceil r/2\rceil}\e^{2(k-1)}\sum_{n\in[\ell]^k\atop k+|n|=r+1}\sum_{m\in I_k^n}\,\prod_{j=1}^k\nabla\cdot \big(\bar\Aa^{n_j,m_j}\odot(\e\nabla)^{n_j-1}\big)\nabla\\
&=&\sum_{k=1}^{\lceil\ell/2\rceil}\e^{2(k-1)}\sum_{n\in[\ell]^k\atop |n|\le\ell-k+1}\sum_{m\in I_k^n}\,\prod_{j=1}^k\nabla\cdot \big(\bar\Aa^{n_j,m_j}\odot(\e\nabla)^{n_j-1}\big)\nabla.
\end{eqnarray*}
{(Here we have used the following observation: for $m=(m_1,\ldots,m_k)$ and $n=(n_1,\ldots,n_k)$ with $m_j\ge0$ and $n_j\ge1$ for all $j$, the restrictions $|m|=2(k-1)$ and $k+|n|=r+1$ entail $m_j+n_j\le r$, thus showing that the index set $I_k$ in the definition of the coefficients~$\{\bar\bb^n\}_n$ is indeed interchangeable with $I_k^n$ here.)}
Choosing $p:=\lceil\frac\ell2\rceil$ in~\eqref{eq:claim_onlyspatial} and inserting the above, we find
\begin{multline}\label{eq:rewr-claim_onlyspatial}
\partial_t^2\bar{W}_\e^{\ell}-\nabla\cdot\Big(\sum_{n=1}^{\ell}\bar\bb^n\odot(\e\nabla)^{n-1}\Big)\nabla\bar W_\e^{\ell}\,=\,\nabla \cdot Q_\e^\ell\\
+f+\sum_{k=1}^{p-1}\sum_{n\in[\ell]^k}\sum_{m\in J^{n}_k}\e^{|m|}\Big(\prod_{j=1}^k\nabla\cdot\big(\bar\Aa^{n_j,m_j}\odot(\e\nabla)^{n_j-1}\big)\nabla\Big)\partial_t^{|m|-2k} f,
\end{multline}
where the remainder $\nabla \cdot Q_\e^\ell$ is given by
\begin{multline*}
Q_\e^\ell\,:=\\
\sum_{n\in[\ell]^p}\sum_{m\in J^{n}_p}\e^{|m|} \big(\bar\Aa^{n_1,m_1}\odot(\e\nabla)^{n_1-1}\big)\nabla \Big(\prod_{j=2}^{p}\nabla\cdot\big(\bar\Aa^{n_j,m_j}\odot(\e\nabla)^{n_j-1}\big)\nabla\Big)\partial_t^{|m|-2(p-1)}\bar W_\e^{\ell}\\
+\sum_{k=1}^{p}\e^{2(k-1)}\sum_{n\in[\ell]^k\atop |n|\ge\ell-k+2}\sum_{m\in I^{n}_k}\big(\bar\Aa^{n_1,m_1}\odot(\e\nabla)^{n_1-1}\big)\nabla \Big(\prod_{j=2}^{k}\nabla\cdot\big(\bar\Aa^{n_j,m_j}\odot(\e\nabla)^{n_j-1}\big)\nabla\Big)\bar W_\e^{\ell}.
\end{multline*}
(We use the standard convention $\prod_{j=l}^k=1$ if $l>k$.)
In order to estimate the remainder~$Q_{\e}^\ell$, let us first examine the $\e$-scaling and the number of derivatives of $\bar W_\e^\ell$ appearing in each of the two contributions in its definition:
\begin{enumerate}[---]
\item In the first contribution in the definition of $Q_\e^\ell$, the terms have scaling $\e^{|n|+|m|-p}$ and involve $|n|+|m|-p+1$ space-time derivatives of $\bar W_\e^\ell$, while the condition on $m,n$ in the sum ensures the lower bound $|n|+|m|-p \ge |m|\ge 2p \ge \ell$ and the upper bound $|n|+|m|-p \le |n|+p\ell-|n|-p=p(\ell-1)=\lceil\frac\ell2\rceil (\ell-1)\le\ell^2$.
\item In the second contribution in the definition of $Q_\e^\ell$, the terms have scaling $\e^{|n|+k-2}$ and involve $|n|+k-1$ derivatives, while the condition on $n,m$ in the sum ensures the lower bound $|n|+k-2 \ge \ell$ and the upper bound $|n|+k-2 \le k\ell+k-2=\lceil\frac\ell2\rceil (\ell+1)-2\le\ell^2$.
\end{enumerate}
Hence, in view of Lemma~\ref{lem:per-cor}, we deduce for $\e\ll1$, for all $r\geq 0$,
\[|\langle D\rangle^r Q_\e^\ell|\,\lesssim\, (\e C\ell)^\ell  |\langle D\rangle^{r+\ell}\langle \e CD\rangle^{\ell^2}D\bar W_\e^\ell|.\]
Likewise, the definition of $\{\bar\cc^n\}_n$ in the statement leads to
\begin{multline*}
\e^2\nabla\cdot\Big(\sum_{n=1}^{\ell-2}\bar\cc^{n}\odot(\e D)^{n-1}\Big)\nabla\\
\,=\,\sum_{k=1}^{\lceil \ell/2\rceil-1}\sum_{n\in[\ell]^k}\sum_{m\in J_k^n\atop|n|+|m|\le\ell+k-1}\e^{|m|}\Big(\prod_{j=1}^k\nabla\cdot\big(\bar\Aa^{n_j,m_j}\odot(\e\nabla)^{n_j-1}\big)\nabla\Big)\partial_t^{|m|-2k},
\end{multline*}
which, inserted into~\eqref{eq:rewr-claim_onlyspatial}, yields
\begin{equation*}
\partial_t^2\bar{W}_\e^{\ell}-\nabla\cdot\Big(\sum_{n=1}^{\ell}\bar\bb^n\odot(\e\nabla)^{n-1}\Big)\nabla\bar W_\e^{\ell}\,=\,f+\e^2\nabla\cdot\Big(\sum_{n=1}^{\ell-2}\bar\cc^{n}\odot(\e D)^{n-1}\Big)\nabla f+\nabla \cdot (Q_\e^\ell+ R_\e^\ell),
\end{equation*}
where the additional remainder $\nabla \cdot R^\ell_\e$ (which is not zero only for $\ell\ge 3$) is given by
\begin{multline*}
R^\ell_\e
\,=\,\sum_{k=1}^{{p-1}}\sum_{n\in[\ell]^k}\sum_{m\in J_k^n\atop|n|+|m|\ge \ell+k}\e^{|m|}
\big(\bar\Aa^{n_1,m_1}\odot(\e\nabla)^{n_1-1}\big)\\
\times\nabla\Big(\prod_{j=2}^k\nabla\cdot\big(\bar\Aa^{n_j,m_j}\odot(\e\nabla)^{n_j-1}\big)\nabla\Big)\partial_t^{|m|-2k}f.
\end{multline*}
Again, in order to estimate this remainder, we first check the $\e$-scaling and the number of derivatives of $f$: the terms have scaling $\e^{|n|+|m|-k}$ and involve $|n|+|m|-k-1$ derivatives, while the condition on $n,m$ in the sum ensures the lower bound $|n|+|m|-k \ge  \ell$ and the upper bound $|n|+|m|-k \le k(\ell-1)\le (\lceil\frac\ell2\rceil-1) (\ell-1)\le\ell^2$. 
Hence, in view of Lemma~\ref{lem:per-cor}, we deduce for $\e\ll1$, for all $r\geq 0$,
\[|\langle D\rangle^r R_\e^\ell|\,\lesssim\, (\e C\ell)^\ell |\langle D\rangle^{r+\ell-1}\langle \e CD\rangle^{\ell^2}f|.\]
The conclusion with the remainder estimate follows {upon setting $E_{\e}^{\ell}=Q_{\e}^{\ell}+R_{\e}^{\ell}$}.

\medskip
\step3 Growth of the revamped coefficients: proof of~\eqref{eq:growthrevamped}.\\
We focus on $\bar\bb^p$, while the estimation of $\bar\cc^p$ is obtained similarly.
By definition of $\bar\bb^p$ in the statement, together with Lemma~\ref{lem:per-cor}, we find
\begin{align*}
|\bar{\bb}^p|
\le \sum_{k\ge1}\,\sum_{(m_1,\ldots,m_k)\in I_k}\sum_{n_1,\ldots, n_k\ge1\atop k+|n|=p+1}C^{|n|+|m|}.
\end{align*}
For all admissible indices in the above sum, we have $|n|\geq k$, and thus $p+1=k+|n|\geq 2k$, hence $k\leq \tfrac{p+1}{2}$. Moreover, any $m\in I_k$ satisfies $|m|=2(k-1)$, hence $p+1=k+|n|=|n|+|m|-k+2$. The upper bound for $k$ then yields $|n|+|m|\leq \tfrac{3p-1}{2}$, which leads us to the bound
\begin{equation*}
|\bar{\bb}^p|\,\le\,C^p\,\sharp\!\left\{r\in\N^{p+1}:\,|r|\le\tfrac{3p-1}2\right\},
\end{equation*}
and the conclusion $|\bar\bb^p|\le C^p$ follows from a simple counting argument with the balls in bins formula.
\qed

\endgroup

\subsection{Proof of Lemma~\ref{lem:invert}}
Let $\bar v_\e^{(\star),\ell}$ be the solution of a well-posed modification of~\eqref{eq:homog-lim/ref}.
In terms of
\begin{equation}\label{a+2}
f_{\e}^{(\star),\ell}(x)\,:=\,\partial_t^2\bar v^{(\star),\ell}_\e-\nabla\cdot\Big(\sum_{n=1}^\ell\sum_{m=0}^{\ell-n}\bar\Aa^{n,m}\odot(\e\nabla)^{n-1}(\e\partial_t)^m\Big)\nabla \bar v^{(\star),\ell}_\e,
\end{equation}
we aim to decompose $f_\e^{(\star),\ell}=f+\nabla \cdot F_\e^{(\star),\ell}$ for some remainder $F_\e^{(\star),\ell}$ satisfying the claimed estimates.
We split the proof into three steps, separately considering the cases $(\star)=\text{(I)}$, $\text{(II)}$, $\text{(III)}$.

\medskip
\step1 High-order filtering: $(\star)=\text{(I)}$.\\
Applying Lemma~\ref{lem:reform} to~\eqref{a+2}, we find
\begin{multline*}
\partial_t^2\bar v_\e^{\text{(I)},\ell}-\nabla\cdot\Big(\bar\Aa+\sum_{k=2}^\ell\bar\bb^k\odot(\e\nabla)^{k-1}\Big)\nabla\bar v_\e^{\text{(I)},\ell}\\
=f_\e^{\text{(I)},\ell}+\e^2\nabla\cdot\Big(\sum_{n=1}^{\ell-2}\bar\cc^n\odot(\e D)^{n-1}\Big)\nabla f_\e^{\text{(I)},\ell}+\nabla\cdot E_\e^{\text{(I)},\ell},
\end{multline*}
where the remainder satisfies for $\e\ll1$, for all $r\ge0$,
\begin{equation}\label{eq:ident-fII-ell-err}
|\langle D\rangle^rE_\e^{\text{(I)},\ell}|\,\le\,(\e C\ell)^\ell\Big(|\langle D\rangle^{r+\ell}\langle \e CD\rangle^{\ell^2}D\bar v_\e^{\text{(I)},\ell}|+|\langle D\rangle^{r+\ell-1}\langle \e CD\rangle^{\ell^2}f_\e^{\text{(I)},\ell}|\Big).
\end{equation}
As $\bar v_\e^{\text{(I)},\ell}$ is the solution of the well-posed modification of~\eqref{eq:homog-lim/ref} obtained by high-order filtering in the sense of Lemma~\ref{lem:apriori-baru-sp}, we deduce
\begin{multline*}
\chi(\e^\alpha\nabla)\bigg(f+\e^2\nabla\cdot\Big(\sum_{n=1}^{\ell-2}\bar\cc^n\odot(\e D)^{n-1}\Big)\nabla f\bigg)\\
\,=\,f_\e^{\text{(I)},\ell}+\e^2\nabla\cdot\Big(\sum_{n=1}^{\ell-2}\bar\cc^n\odot(\e D)^{n-1}\Big)\nabla f_\e^{\text{(I)},\ell}+\nabla\cdot E_\e^{\text{(I)},\ell},
\end{multline*}
or equivalently,
\begin{multline}\label{eq:ident-fII-ell}
f_\e^{\text{(I)},\ell}+\e^2\nabla\cdot\Big(\sum_{n=1}^{\ell-2}\bar\cc^n\odot(\e D)^{n-1}\Big)\nabla f_\e^{\text{(I)},\ell}\\
\,=\,f+\e^2\nabla\cdot\Big(\sum_{n=1}^{\ell-2}\bar\cc^n\odot(\e D)^{n-1}\Big)\nabla f+\nabla\cdot\hat E_\e^{\text{(I)},\ell},
\end{multline}
with modified remainder
\begin{equation}\label{eq:ident-fII-ell-re}
\hat E_\e^{\text{(I)},\ell}\,:=\,-E_\e^{\text{(I)},\ell}-(1-\chi(\e^\alpha\nabla))\triangle^{-1}\nabla\bigg(f+\e^2\nabla\cdot\Big(\sum_{n=1}^{\ell-2}\bar\cc^n\odot(\e D)^{n-1}\Big)\nabla f\bigg).
\end{equation}
In order to estimate $f_\e^{\text{(I)},\ell}$, it remains to invert the differential operator
\begin{equation}\label{eq:def-pepsD}
1+\e^2\nabla\cdot\Big(\sum_{n=1}^{\ell-2}\bar\cc^n\odot(\e D)^{n-1}\Big)\nabla
\end{equation}
in the left-hand side of~\eqref{eq:ident-fII-ell}.
Although it is not invertible in general, we may invert it approximately using {a Neumann series truncated at order $O(\e^{\ell})$. More precisely, we set
\begin{equation*}
\mathcal{O}^{\ell}_{\e}\,:=\,1+\sum_{k=1}^{\lfloor \ell/2\rfloor}(-\e^2)^k\sum_{n\in[\ell-2]^k\atop |n|\le\ell-k}~\prod_{j=1}^{k}\Big(\nabla\cdot\big(\bar\cc^{n_j}\odot(\e D)^{n_j-1}\big)\nabla\Big).	
\end{equation*}
{By a simple counting argument with the balls in bins formula, the number of terms in this sum defining $\mathcal{O}^{\ell}_{\e}$ can be bounded by
\begin{equation*}
\sum_{k=1}^{\lfloor \ell/2\rfloor}\sharp\{n\in [\ell-2]^k:|n|\le\ell-k\}\,\le\,C^\ell,
\end{equation*}
and therefore,} combining this cardinality bound
with the bound $|\bar\cc^n|\leq C^n$, we easily deduce for any function $g$, for all $r\ge0$,
\begin{equation}\label{eq:boundapproxInv}
|\langle D\rangle^r \mathcal{O}^{\ell}_{\e}g|\leq C^{\ell}|\langle D\rangle^{r}\langle\e CD\rangle^{\ell}g|.
\end{equation}
{Next, we check that $\mathcal O_\e^\ell$ is indeed an approximate inverse for the differential operator in~\eqref{eq:def-pepsD}: by definition of $\mathcal O_\e^\ell$, working out cancellations, we find for any function $g$,
\begin{equation*}
\mathcal O_\e^\ell\bigg(1+\e^2\nabla\cdot\Big(\sum_{n=1}^{\ell-2}\bar\cc^n\odot(\e D)^{n-1}\Big)\nabla\bigg)g
\,=\,g+\nabla\cdot\mathcal H_\e^\ell g,
\end{equation*}
where the remainder is explicitly given by
\begin{align*}\nabla\cdot\mathcal H_\e^\ell g\,:=\,&-\sum_{k=2}^{\lfloor\ell/2\rfloor}(-\e^2)^k\sum_{n\in[\ell-2]^k\atop\ell+1-k\le|n|\le n_k+\ell+1-k}~\prod_{j=1}^k\Big(\nabla\cdot\big(\bar\cc^{n_j}\odot(\e D)^{n_j-1}\big)\nabla\Big)g\\&-(-\e^2)^{\lfloor\ell/2\rfloor+1}\sum_{n\in[\ell-2]^{\lfloor\ell/2\rfloor+1}\atop |n|\le n_k+\ell-\lfloor\ell/2\rfloor}~\prod_{j=1}^{\lfloor\ell/2\rfloor+1}\Big(\nabla\cdot\big(\bar\cc^{n_j}\odot(\e D)^{n_j-1}\big)\nabla\Big)g,\end{align*}
and can be estimated as follows{:} for all $r\ge0$, using $|\bar\cc^n|\le C^n$, {we have that}
\begin{equation}\label{a+3}
|\langle D\rangle^r\mathcal H_\e^\ell g| \,\le\, (\e C)^\ell | \langle D\rangle^{r+2\ell} g|.	
\end{equation}
Applying} $\mathcal O_\e^\ell$ to both sides of~\eqref{eq:ident-fII-ell}, we then get
\begin{equation*}
f_\e^{\text{(I)},\ell}\,=\,f+\nabla\cdot F_\e^{\text{(I)},\ell},
\qquad F_\e^{\text{(I)},\ell}\,:=\, \mathcal H_\e^\ell (f-f_\e^{\text{(I)},\ell})+\mathcal O_\e^\ell\hat E_\e^{\text{(I)},\ell},
\end{equation*}
and the desired remainder estimate follows by combining~\eqref{a+2}, \eqref{eq:ident-fII-ell-err}, \eqref{eq:ident-fII-ell-re}, \eqref{eq:boundapproxInv}, and~\eqref{a+3}, together with the a priori estimate of Lemma~\ref{lem:apriori-baru-sp}(ii) for $\bar v_\e^{\text{(I)},\ell}$. Note that the remainder term that is local with respect to $f$ can also be bounded with an $\Ld^1(0,t)$-norm instead of an $\Ld^\infty(0,t)$-norm up to loosing a time derivative.}

\medskip
\step2 High{er}-order regularization: $(\star)=\text{(II)}$.\\
As $\bar v_\e^{\text{(II)},\ell}$ is the solution of the well-posed modification of~\eqref{eq:homog-lim/ref} obtained by high{er}-order regularization in the sense of Lemma~\ref{lem:apriori-baru-sp}, we get instead of~\eqref{eq:ident-fII-ell},
\begin{multline*}
f_\e^{\text{(II)},\ell}+\e^2\nabla\cdot\Big(\sum_{n=1}^{\ell-2}\bar\cc^n\odot(\e D)^{n-1}\Big)\nabla f_\e^{\text{(II)},\ell}\\
\,=\,f+\e^2\nabla\cdot\Big(\sum_{n=1}^{\ell-2}\bar\cc^n\odot(\e D)^{n-1}\Big)\nabla f+\nabla\cdot\hat E_\e^{\text{(II)},\ell},
\end{multline*}
with remainder
\begin{equation*}
\hat E_\e^{\text{(II)},\ell}\,:=\,-E_\e^{\text{(II)},\ell}+\kappa_\ell(\e|\nabla|)^{\ell}\nabla\bar v_\e^{\text{(II)},\ell},
\end{equation*}
where $E_\e^{\text{(II)},\ell}$ satisfies the corresponding estimate~\eqref{eq:ident-fII-ell-err}.
Now applying the same approximate inverse operator $\mathcal O_\e^\ell$ to both sides of this identity, we get
\begin{equation*}
f_\e^{\text{(II)},\ell}\,=\,f+\nabla\cdot F_\e^{\text{(II)},\ell},
\qquad F_\e^{\text{(II)},\ell}\,:=\, \mathcal H_\e^\ell (f-f_\e^{\text{(II)},\ell})+\mathcal O_\e^\ell\hat E_\e^{\text{(II)},\ell},
\end{equation*}
and the desired remainder estimate follows similarly.

\medskip
\step3 Boussinesq trick: $(\star)=\text{(III)}$.\\
By definition, $\bar v_\e^{\text{(III)},\ell}$ is the solution of the well-posed modification of~\eqref{eq:homog-lim/ref} obtained by Boussinesq trick in the sense of Lemma~\ref{lem:apriori-baru-sp}, that is,
\begin{multline}\label{eq:proxy-baru-III-rewrrr}
\partial_t^2\Big(1+\sum_{l=2}^\ell\kappa_l(\e|\nabla|)^{l-1}\Big)\bar v_\e^{\text{(III)},\ell}\\
-\nabla\cdot\bigg(\sum_{n=1}^\ell\Big(\kappa_n\bar\Aa+\sum_{l=1}^{n-1}\kappa_l\bar\bb^{n+1-l}\odot(\tfrac\nabla{|\nabla|})^{n-l}\Big)(\e|\nabla|)^{n-1}\bigg)\nabla\bar v_\e^{\text{(III)},\ell}\\
\,=\,\Big(1+\sum_{l=2}^\ell\kappa_l(\e|\nabla|)^{l-1}\Big)\bigg(f+\e^2\nabla\cdot\Big(\sum_{k=1}^{\ell-2}\bar\cc^k\odot(\e D)^{k-1}\Big)\nabla f\bigg),
\end{multline}
with $\bar v_\e^{\text{(III)},\ell}=f=0$ for $t\le0$. We recall that the coefficients $\{\kappa_l\}_l$ are defined in~\eqref{eq:choice-kappaell-B} and that well-posedness is indeed ensured by Lemma~\ref{lem:apriori-baru-sp}(ii).
Using the identity
\begin{multline*}
\sum_{n=1}^\ell \Big(\kappa_n\bar\Aa+\sum_{l=1}^{n-1} \kappa_{l}\bar \bb^{n+1-l}\odot(\tfrac{\nabla}{|\nabla|})^{n-l}\Big)(\e|\nabla|)^{n-1}\\
=\Big(1+\sum_{l=2}^{\ell}  \kappa_l (\e |\nabla|)^{l-1}\Big)\Big(\bar\Aa+\sum_{k=2}^\ell\bar\bb^k\odot(\e\nabla)^{k-1}\Big)\\
-\sum_{n=\ell+1}^{2\ell-1} \Big(
\sum_{l=n+1-\ell}^{\ell} \kappa_{l}\bar \bb^{n+1-l}\odot(\tfrac{\nabla}{|\nabla|})^{n-l}\Big)(\e|\nabla|)^{n-1},
\end{multline*}
the above equation~\eqref{eq:proxy-baru-III-rewrrr} can be alternatively written as
\begin{equation*}
\partial_t^2\bar v_\e^{\text{(III)},\ell}
-\nabla\cdot\Big(\bar\Aa+\sum_{k=2}^\ell\bar\bb^k\odot(\e\nabla)^{k-1}\Big)\nabla\bar v_\e^{\text{(III)},\ell}
\,=\,f+\e^2\nabla\cdot\Big(\sum_{k=1}^{\ell-2}\bar\cc^k\odot(\e D)^{k-1}\Big)\nabla f-\nabla\cdot G_\e^{\ell},
\end{equation*}
in terms of
\[G_\e^{\ell}\,:=\,\sum_{n=\ell+1}^{2\ell-1} \Big(
\sum_{l=n+1-\ell}^{\ell} \kappa_{l}\bar \bb^{n+1-l}\odot(\tfrac{\nabla}{|\nabla|})^{n-l}\Big)(\e|\nabla|)^{n-1}\Big(1+\sum_{l=2}^{\ell}  \kappa_l (\e |\nabla|)^{l-1}\Big)^{-1}\nabla\bar v_\e^{\text{(III)},\ell},\]
where the inverse operator is obviously well-defined as $\kappa_l\ge0$ for all $l$.
We then get, instead of~\eqref{eq:ident-fII-ell},
\begin{multline*}
f_\e^{\text{(III)},\ell}+\e^2\nabla\cdot\Big(\sum_{n=1}^{\ell-2}\bar\cc^n\odot(\e D)^{n-1}\Big)\nabla f_\e^{\text{(III)},\ell}\\
\,=\,f+\e^2\nabla\cdot\Big(\sum_{k=1}^{\ell-2}\bar\cc^k\odot(\e D)^{k-1}\Big)\nabla f+\nabla\cdot\hat E_\e^{\text{(III)},\ell},
\end{multline*}
with remainder
\[\hat E_\e^{\text{(III)},\ell}\,:=\,-E_\e^{\text{(III)},\ell}-G_\e^{\ell},\]
where $E_\e^{\text{(III)},\ell}$ satisfies the corresponding estimate~\eqref{eq:ident-fII-ell-err}.
The conclusion then follows similarly as in the first two steps.\qed

\section{Relating spectral to hyperbolic correctors}\label{sec:fromGeom2Spec}

This last section is devoted to the proof of \eqref{eq:2scale-concl-perOpt}.
It relies on an algorithmic procedure to relate spectral and hyperbolic correctors.  We split the proof into two steps.

\begingroup\allowdisplaybreaks
\medskip
\step1 Reformulation of the hyperbolic two-scale expansion: for $\e\ll1$, we have for all~$\ell\ge1$,
\begin{multline}\label{eq:tadam}
\bigg\|H^\ell_\e[\bar v_\e^{\text{(I)},\ell;t}]
-\sum_{n=0}^\ell\e^{n}\mathring\psi_\ell^{n}(\tfrac{\cdot}{\e})\odot\nabla^{n}\bar{u}^{\text{(I)},\ell;t}_{\e}
\\
-\e^3\sum_{n=0}^{\ell-3}\sum_{2m=0}^{\ell-n-3}\e^{n+2m}\mathring\zeta_\ell^{n,2m}(\tfrac\cdot\e)\odot \nabla^{n+1}\partial_t^{2m}f^t\bigg\|_{H^1(\R^d)}\\
\,\le\,(\e C)^{\ell}\|\langle D\rangle^{C\ell}f\|_{\Ld^1((0,t);\Ld^2(\R^d))},
\end{multline}
where for each $n,m$ the correctors $\mathring\psi_\ell^{n}$ and $\mathring\zeta_\ell^{n,m}$ are suitable linear combinations of hyperbolic correctors $\{\phi^{n',m'}\}_{n',m'}$ with coefficients involving $\ell$ and $\{\bar\bb^{n'},\bar\cc^{n'}\}_{n'}$.

Comparing the effective equation in the spectral and in the geometric approach, cf.~\eqref{eq:homog-lim-sp} and~\eqref{eq:homog-lim/ref}, and using the well-posed modification by high-order filtering, we get
\begin{equation}\label{eq:u_and_v}
\bar v_\e^{\text{(I)},\ell}\,=\,\bar u_\e^{\text{(I)},\ell}+\e^2\nabla\cdot\Big(\sum_{k=1}^{\ell-2}\bar\cc^k\odot(\e D)^{k-1}\Big)\nabla \bar u_\e^{\text{(I)},\ell}.
\end{equation}
Inserting this into the definition~\eqref{eq:def-Fell} of the geometric two-scale expansion $H_{\e}^{\ell}[\bar v_{\e}^{\text{(I)},\ell}]$, and rearranging terms, we get
\begin{equation}\label{eq:decomp-Fell-Rell}
H^\ell_\e[\bar v_\e^{\text{(I)},\ell}]\,=\,\sum_{n=0}^\ell\sum_{m=0}^{\ell-n}\e^{n+m}z^{n,m}_\ell(\tfrac{\cdot}{\e})\odot\nabla^n\partial_t^{m}\bar u_\e^{\text{(I)},\ell}+ R^{\ell}_{\e,1}[\bar u_{\e}^{\text{(I)},\ell}],
\end{equation}
{where for each $n,m$ the corrector $z^{n,m}_\ell$ is some linear combination of hyperbolic correctors $\{\phi^{n',m'}\}_{n',m'}$ with coefficients involving~$\ell$ and~$\{\bar\cc^{n'}\}_{n'}$,} and where the remainder term is given by
\begin{multline*}
R^{\ell}_{\e,1}[\bar u_{\e}^{\text{(I)},\ell}]\,:=\\
\sum_{n=0}^\ell\sum_{m=0}^{\ell-n}\sum_{k=1}^{\ell-2}\mathds1_{n+m+k\ge\ell}\,\e^{n+m+k+1} \phi^{n,m}_{j_1\ldots j_n}(\tfrac\cdot\e)(\bar\cc^k\odot D^{k-1}):\nabla^2\nabla_{j_1\ldots j_n}^n\partial_t^{m}\bar u_\e^{\text{(I)},\ell}.
\end{multline*}
By construction, we note that $z^{0,0}_\ell=\phi^{0,0}=1$, $z^{0,m}_\ell=\phi^{0,m}=0$ for all $m\ge1$, and that~$z^{n,m}_\ell$ does not depend on $\ell$ provided $n+m<\ell$. We emphasize that the $z^{n,m}_\ell$'s do a priori not have vanishing average.
Using Lemma \ref{lem:per-cor} together with~\eqref{eq:growthrevamped}, and using the Sobolev embedding in form of~\eqref{eq:use-Sobolev-cor} to bound products with correctors, the remainder can be estimated as follows, for $a>\frac d2$,
\begin{equation*}
\|R^{\ell}_{\e,1}[\bar u_{\e}^{\text{(I)},\ell}]\|_{H^1(\R^d)}\,\le\,(\e C)^{\ell}\|\langle D\rangle^{2\ell+a-1}D\bar u_{\e}^{\text{(I)},\ell}\|_{\Ld^2(\R^d)},
\end{equation*}
and thus, by the a priori estimate of Lemma~\ref{lem:apriori-baru-sp}(ii),
\begin{equation}\label{eq:final_bound_error1}
\|R^{\ell}_{\e,1}[\bar u_{\e}^{\text{(I)},\ell;t}]\|_{H^1(\R^d)}\leq (\e C)^{\ell}\|\langle D\rangle^{2\ell+a-{1}}f\|_{\Ld^1((0,t);\Ld^2(\R^d))}.
\end{equation}

Next, using the equation~\eqref{eq:homog-lim-sp} for $\bar{u}_{\e}^{\text{(I)},\ell}$, we shall proceed to remove the time derivatives appearing in the geometric two-scale expansion~\eqref{eq:decomp-Fell-Rell}.
For that purpose, arguing by induction and successively {separating} terms of order $\e^{\ell+1}$ or higher, we {show that} the equation for $\bar{u}_{\e}^{\text{(I)},\ell}$ yields for all~\mbox{$m\ge1$},
\begin{multline}\label{eq:expresstimederivative}
\partial_t^{2m} \bar{u}^{\text{(I)},\ell}_{\e}-\sum_{1\le \beta_1,\ldots,\beta_m\le \ell\atop|\beta|\le m+\ell}\e^{|\beta|-m}\bar B^\beta(\nabla)\bar{u}^{\text{(I)},\ell}_{\e}\\
=\chi(\e^{\alpha}\nabla)\partial_t^{2(m-1)}f+\chi(\e^{\alpha}\nabla)\sum_{k=1}^{m-1}\sum_{1\le \beta_1,\ldots,\beta_k\le\ell\atop|\beta|\le k+\ell}\e^{|\beta|-k}\bar B^\beta(\nabla)\partial_t^{2(m-k-1)}f\\
+\sum_{k=2}^m\sum_{\beta\in L_k}\e^{|\beta|-k}\bar B^\beta(\nabla)\partial_t^{2(m-k)}\bar u^{\text{(I)},\ell}_\e,
\end{multline}
in terms of the index sets
\[L_k\,:=\,\Big\{\beta=(\beta_1,\dots,\beta_k):1\le\beta_j\le\ell~\forall j,~\sum_{j=1}^s\beta_j\le s+\ell~\forall s<k,~|\beta|>k+\ell\Big\},\]
where for all $k\ge1$ and $\beta=(\beta_1,\ldots,\beta_k)$ we use the short-hand notation
\begin{equation}\label{eq:not-bigBb}
\bar{B}^\beta(\nabla)\,:=\,\prod_{j=1}^k\nabla\cdot(\bar\bb^{\beta_j}\odot\nabla^{\beta_j-1})\nabla.
\end{equation}
We prove identity~\eqref{eq:expresstimederivative} by induction.
First, for $m=1$, it simply coincides with (the filtered version of) equation~\eqref{eq:homog-lim-sp} for $\bar{u}_{\e}^{\text{(I)},\ell}$.
Next, assuming that~\eqref{eq:expresstimederivative} holds for some~$m\ge1$, applying $\partial_t^2$ to both sides of the identity, and using the equation for $\bar{u}_{\e}^{\text{(I)},\ell}$ to replace~$\partial_t^2\bar{u}_{\e}^{\text{(I)},\ell}$ in the left-hand side, we find
\begin{multline*}
\partial_t^{2(m+1)} \bar{u}^{\ell}_{\e}-\sum_{1\le \beta_1,\ldots,\beta_m\le \ell\atop|\beta|\le m+\ell}\e^{|\beta|-m}\bar B^\beta(\nabla)\Big(\sum_{n=1}^\ell\e^{n-1}\nabla\cdot(\bar\bb^n\odot\nabla^{n-1})\nabla\Big)\bar{u}^{\ell}_{\e}\\
=\chi(\e^{\alpha}\nabla)\partial_t^{2m}f+\chi(\e^{\alpha}\nabla)\sum_{k=1}^{m}\sum_{1\le \beta_1,\ldots,\beta_k\le\ell\atop|\beta|\le k+\ell}\e^{|\beta|-k}\bar B^\beta(\nabla)\partial_t^{2(m-k)}f\\
+\sum_{k=2}^m\sum_{\beta\in L_k}\e^{|\beta|-k}\bar B^\beta(\nabla)\partial_t^{2(m-k+1)}\bar u^\ell_\e.
\end{multline*}
After suitably splitting powers of~$\e$ in the left-hand side, this indeed proves identity~\eqref{eq:expresstimederivative} with $m$ replaced by $m+1$. 

Now inserting~\eqref{eq:expresstimederivative} into the two-scale expansion~\eqref{eq:decomp-Fell-Rell} in order to replace time derivatives, and again {separating} terms of order $\e^{\ell+1}$ or higher, we are led to
\begin{multline}\label{eq:Fell-decomp-final-ploplop}
H^\ell_\e[\bar v_\e^{\text{(I)},\ell}]
\,=\,
\sum_{n=0}^\ell\e^{n}z^{n,0}_\ell(\tfrac{\cdot}{\e})\odot\nabla^n\bar{u}^{\text{(I)},\ell}_{\e}\\
+\sum_{n=0}^\ell\sum_{2m=2}^{\ell-n}\sum_{1\le \beta_1,\ldots,\beta_{m}\le \ell\atop n+m+|\beta|\le\ell}\e^{n+m+|\beta|}z^{n,2m}_\ell(\tfrac{\cdot}{\e})\odot \bar B^\beta(\nabla)\nabla^n\bar{u}^{\text{(I)},\ell}_{\e}\\
+\sum_{n=0}^\ell\sum_{2m=2}^{\ell-n}\e^{n+2m}z^{n,2m}_\ell(\tfrac{\cdot}{\e})\odot\chi(\e^{\alpha}\nabla)\nabla^n\partial_t^{2(m-1)}f\\
+\sum_{n=0}^\ell\sum_{2m=2}^{\ell-n}\sum_{j=1}^{m-1}\sum_{1\le \beta_1,\ldots,\beta_j\le\ell\atop n+2m+|\beta|\le j+\ell}\e^{n+2m+|\beta|-j}z^{n,2m}_\ell(\tfrac{\cdot}{\e})\odot\chi(\e^{\alpha}\nabla)\bar B^\beta(\nabla)\nabla^n\partial_t^{2(m-j-1)}f\\
+R^{\ell}_{\e,1}[\bar u_{\e}^{\text{(I)},\ell}]+R^{\ell}_{\e,2}[\bar u_{\e}^{\text{(I)},\ell}]+R^{\ell}_{\e,3}[f],
\end{multline}
where the last two remainder terms take the form
\begin{eqnarray*}
R^{\ell}_{\e,2}[\bar u_{\e}^{\text{(I)},\ell}]&:=&\sum_{n=0}^\ell\sum_{2m=2}^{\ell-n}\sum_{j=2}^m\sum_{\beta\in L_j}\e^{n+2m+|\beta|-j}z_\ell^{n,2m}\left(\tfrac{\cdot}{\e}\right)\odot\bar B^\beta(\nabla)\nabla^n\partial_t^{2(m-j)}\bar u^{\text{(I)},\ell}_\e\\
&&+\sum_{n=0}^\ell\sum_{2m=2}^{\ell-n}\sum_{1\le \beta_1,\ldots,\beta_{m}\le \ell\atop|\beta|\le m+\ell,n+m+|\beta|>\ell}\e^{n+m+|\beta|}z_\ell^{n,2m}\left(\tfrac{\cdot}{\e}\right)\odot\bar B^\beta(\nabla)\nabla^n\bar{u}^{\text{(I)},\ell}_{\e},\\
R^{\ell}_{\e,3}[f]&:=&\sum_{n=0}^\ell\sum_{2m=2}^{\ell-n}\sum_{j=1}^{m-1}\sum_{1\le \beta_1,\ldots,\beta_j\le\ell\atop |\beta|\le j+\ell,~n+2m+|\beta|>j+\ell}\e^{n+2m+|\beta|-j} z_\ell^{n,2m}(\tfrac{\cdot}{\e})\\
&&\hspace{4cm}\odot\chi(\e^{\alpha}\nabla)\bar B^\beta(\nabla)\nabla^n\partial_t^{2(m-j-1)}f.
\end{eqnarray*}
Equivalently, in view of~\eqref{eq:not-bigBb}, this can be reformulated as
\begin{multline}\label{eq:not-bigBb-better}
H^\ell_\e[\bar v_\e^{\text{(I)},\ell}]
\,=\\
\sum_{n=0}^\ell\e^{n}\mathring\psi_\ell^{n}(\tfrac{\cdot}{\e})\odot\nabla^{n}\bar{u}^{\text{(I)},\ell}_{\e}
+\e^3\sum_{n=0}^{\ell-3}\sum_{2m=0}^{\ell-n-3}\e^{n+2m}\mathring\zeta_\ell^{n,2m}(\tfrac\cdot\e)\odot\chi(\e^{\alpha}\nabla)\nabla^{n+1}\partial_t^{2m}f\\
+R^{\ell}_{\e,1}[\bar u_{\e}^{\ell}]+R^{\ell}_{\e,2}[\bar u_{\e}^{\ell}]+R^{\ell}_{\e,3}[f],
\end{multline}
{where for each $n,m$ the correctors $\mathring\psi^n_\ell$ and $\mathring\zeta^{n,2m}_\ell$ are suitable linear combinations of $\{z^{n',2m'}\}_{n',m'}$ with coefficients involving $\ell$ and $\{\bar\bb^{n'}\}_{n'}$.}
We emphasize that $\mathring\psi^n_\ell$ and $\mathring\zeta^{n,2m}_\ell$ do a priori not have vanishing average.

We turn to the estimation of the last two remainder terms in~\eqref{eq:not-bigBb-better}.
Recalling the way that $z^{n,2m}$ {depends on $\{\phi^{n',m'}\}_{n',m'}$ and on $\{\cc^{n'}\}_{n'}$}, using Lemma~\ref{lem:per-cor} together with~\eqref{eq:growthrevamped}, and using again the Sobolev embedding in form of~\eqref{eq:use-Sobolev-cor} to bound products with correctors, we find for $\e\ll1$ and $a>\frac d2$,
\begin{equation*}
\|R_{\e,2}^\ell[\bar u_\e^{\text{(I)},\ell}]\|_{H^1(\R^d)}
\,\le\,(\e C)^{\ell}\|\langle D\rangle^{3\ell+a-1}D\bar u_\e^\ell\|_{\Ld^2(\R^d)},
\end{equation*}
and thus, further combining {this} with the a priori estimate of Lemma~\ref{lem:apriori-baru-sp}(ii),
\[\|R_{\e,2}^\ell[\bar u_\e^{\text{(I)},\ell;t}]\|_{H^1(\R^d)}\,\le\,(\e C)^{\ell}\|\langle D\rangle^{3\ell+a-1}f\|_{\Ld^1((0,t);\Ld^2(\R^d))}.\]
Similarly, we find for $\e\ll1$ and $a>\frac d2$,
\[\|R_{\e,3}^\ell[f]\|_{H^1(\R^d)}\,\le\,(\e C)^{\ell}\|\langle D\rangle^{2\ell+a-1}f\|_{\Ld^2(\R^d)}.\]
Finally, we recall that the cut-off $\chi(\e^\alpha\nabla)$ can be removed in the right-hand side of~\eqref{eq:not-bigBb-better} up to higher-order errors: for $a>\frac d2$,
\begin{multline*}
{\bigg\|\e^3\sum_{n=0}^{\ell-3}\sum_{2m=0}^{\ell-n-3}\e^{n+2m}\mathring\zeta_\ell^{n,2m}(\tfrac\cdot\e)\odot(1-\chi(\e^{\alpha}\nabla))\nabla^{n+1}\partial_t^{2m}f\bigg\|_{H^1(\R^d)}}\\
\,\le\,\e^3C^\ell\|(1-\chi(\e^{\alpha}\nabla))\langle D\rangle^{\ell+a-1}f\|_{\Ld^2(\R^d)}
\,\le\,(\e C)^{\ell}\|\langle D\rangle^{\ell+{\lceil\frac{\ell}\alpha\rceil}+a-1}f\|_{\Ld^2(\R^d)}\\
{\,\le\,(\e C)^{\ell}\|\langle D\rangle^{\ell+{\lceil\frac{\ell}\alpha\rceil}+a}f\|_{\Ld^1((0,t);\Ld^2(\R^d))}} .
\end{multline*}
Inserting these bounds together with~\eqref{eq:final_bound_error1} into~\eqref{eq:not-bigBb-better}, the claim~\eqref{eq:tadam} follows.

\medskip
\step2 Conclusion.\\
In view of~\eqref{eq:tadam}, the result of Theorem~\ref{th:main-per2} yields
\begin{multline*}
\bigg\|u_\e^t
-\sum_{n=0}^\ell\e^{n}\mathring\psi_\ell^{n}(\tfrac{\cdot}{\e})\odot\nabla^{n}\bar{u}^{\text{(I)},\ell;t}_{\e}
-\e^3\sum_{n=0}^{\ell-3}\sum_{2m=0}^{\ell-n-3}\e^{n+2m}\mathring\zeta_\ell^{n,2m}(\tfrac\cdot\e)\odot\nabla^{n+1}\partial_t^{2m}f^t\bigg\|_{H^1(\R^d)}\\
\,\lesssim\,
(\e C\ell)^{\ell} \langle t\rangle \|\langle D\rangle^{C\ell^2} f\|_{\Ld^1((0,t);\Ld^2(\R^d))}.
\end{multline*}
{{Now we compare this with the result of Theorem~\ref{th:main-per}:
using~\eqref{eq:expand-gam} to expand $\gamma_{\ell}(\e\nabla)$ in the spectral two-scale expansion~\eqref{eq:goal1}, discarding terms of order $O(\e^{\ell})$, and using Lemma~\ref{lem:apriori-baru-sp}(ii) to estimate the latter, the result~\eqref{eq:2scale-concl-per} of Theorem~\ref{th:main-per} leads us to}
\begin{multline}\label{eq:expansion_small}
	\bigg\|
	\sum_{n=0}^\ell\e^{n}\widetilde{\psi}_{\ell}^n(\tfrac{\cdot}{\e})\odot\nabla^{n}\bar{u}^{\text{(I)},\ell;t}_{\e}
	+\e^3\sum_{n=0}^{\ell-3}\sum_{2m=0}^{\ell-n-3}\e^{n+2m}\widetilde{\zeta}_{\ell}^{n,2m}(\tfrac{\cdot}{\e})\odot\nabla^{n+1}\partial_t^{2m}f^t\bigg\|_{H^1(\R^d)}\\
	\,\lesssim\,
	(\e C\ell)^{\ell} \langle t\rangle \|\langle D\rangle^{C\ell^2} f\|_{\Ld^1((0,t);\Ld^2(\R^d))},
\end{multline}
in terms of the corrector differences
\begin{equation*}
	\widetilde{\psi}_{\ell}^n=\sum_{k=0}^n\gamma_{\ell}^k\otimes_s\psi^{n-k}-\mathring\psi_\ell^{n},\qquad \widetilde{\zeta}_{\ell}^{n,2m}=\sum_{k=0}^n\gamma_{\ell}^k\otimes_s\zeta^{n-k,2m}-\mathring\zeta_\ell^{n,2m},
\end{equation*}
{where $\otimes_s$ stands for symmetric tensor product.
We claim that~\eqref{eq:expansion_small} entails $\widetilde{\psi}_{\ell}^n=0$ for all $n\leq \ell-1$, $\widetilde{\zeta}_{\ell}^{n,2m}=0$ for all $n+2m+3\leq\ell-1$, and $\nabla\widetilde{\psi}_{\ell}^{\ell}=0$ and $\nabla \widetilde{\zeta}_{\ell}^{n,2m}=0$ for $n+2m+3=\ell$.
We argue by induction and prove:
\begin{equation}\label{eq:iterate-vanishcorrdiff}
\left\{\begin{array}{lll}
\text{for all $j\le \ell-1$}&:&\text{$\widetilde{\psi}_{\ell}^j=0$ and $\widetilde{\zeta}_{\ell}^{n,2m}=0$ for $n+2m+3=j$},\\
\text{for $j=\ell$}&:&\text{$\nabla\widetilde{\psi}_{\ell}^{\ell}=0$ and $\nabla \widetilde{\zeta}_{\ell}^{n,2m}=0$ for $n+2m+3=\ell$.}
\end{array}\right.
\end{equation}
Assume that this result is known to hold for $j< j_0$, given some $j_0\le\ell-1$.
Using the a priori estimates of Lemma~\ref{lem:apriori-baru-sp}(ii), first note that we have $\bar u_\e^{\operatorname{(I)},\ell}\to\bar u$ in $C^\infty_\loc(\R;H^\infty(\R^d))$ as $\e\downarrow0$, where $\bar u$ is the solution of the standard homogenized wave equation~\eqref{eq:standard-hom}. Given $h\in C_\per^\infty(Q)$, multiplying the expression in the left-hand side of~\eqref{eq:expansion_small} by $\e^{-j_0}h(\frac\cdot\e)$, and passing to the limit $\e\downarrow0$ in the $\Ld^2$-norm, we then get
\begin{equation*}
\expecm{h\widetilde{\psi}_{\ell}^{j_0}}\odot\nabla^{j_0}\bar{u}+\sum_{n+2m+3=j_0}\expecm{h\widetilde{\zeta}_{\ell}^{n,2m}}\odot\nabla^{n+1}\partial_t^{2m}f\,=\,0.
\end{equation*}
As $\bar u$ satisfies the homogenized wave equation~\eqref{eq:standard-hom},
applying the wave operator $\partial_t^2-\nabla\cdot\bar\Aa\nabla$ to this pointwise identity leads us to
\begin{equation*}
\expecm{h\widetilde{\psi}_{\ell}^{j_0}}\odot\nabla^{j_0}f+\sum_{n+2m+3=j_0}\expecm{h\widetilde{\zeta}_{\ell}^{n,2m}}\odot\nabla^{n+1}\partial_t^{2m}(\partial_t^2-\nabla\cdot\bar\Aa\nabla)f\,=\,0.
\end{equation*}
Choosing for instance $f^t(x)=\exp(-(t^2+|x|^2))$ for $t\ge1$, it is easily deduced by induction, by a linear independence argument, that $\expecm{h\widetilde{\psi}_{\ell}^{j_0}}=0$ and $\expecm{h\widetilde{\zeta}_{\ell}^{n,2m}}=0$ for all $n,m$ with \mbox{$n+2m+3=j_0$}. As $h\in C^\infty_\per(Q)$ is arbitrary, we deduce that the claim~\eqref{eq:iterate-vanishcorrdiff} also holds for $j=j_0$. The same argument can be adapted to the case~$j_0=\ell$, rather starting from the estimate on the $H^1$-norm in~\eqref{eq:expansion_small}. This ends the proof of~\eqref{eq:iterate-vanishcorrdiff}.}

In other words, we have thus proved that, given $\ell\geq1$, for all $j\le\ell-1$ and for all $n,m$ with $n+m+3\le\ell-1$, we have
\begin{eqnarray*}
\mathring\psi_\ell^j=\sum_{k=0}^j\gamma_\ell^k\otimes_s\psi^{j-k},\qquad \mathring\zeta_\ell^{n,2m}=\sum_{k=0}^n\gamma_\ell^k\otimes_s\zeta^{n-k,2m},
\end{eqnarray*}
while for $j=\ell$ and for all $n,m$ with $n+m+3=\ell$,
\begin{eqnarray*}
\nabla_i\mathring\psi_\ell^{\ell}=\sum_{k=0}^{\ell}\gamma_\ell^k\otimes_s\nabla_i\psi^{\ell-k},\qquad \nabla_i\mathring\zeta_\ell^{n,2m}=\sum_{k=0}^n\gamma_\ell^k\otimes_s\nabla_i\zeta^{n-k,2m}.
\end{eqnarray*}
Inserting this back into~\eqref{eq:tadam}, using~\eqref{eq:expand-gam} to reconstruct $\gamma_\ell(\e\nabla)$, discarding terms of order~$O(\e^\ell)$, and using Lemma~\ref{lem:apriori-baru-sp}(ii) to estimate the latter,
we may then recognize the definition~\eqref{eq:goal1} of the spectral two-scale expansion, to the effect of
\begin{equation*}
\|H^\ell_\e[\bar v_\e^{\text{(I)},\ell;t}]
-S_\e^\ell[\bar u_\e^{\text{(I)},\ell;t},f]\|_{H^1(\R^d)}
\,\lesssim\,
(\e C)^{\ell}\|\langle D\rangle^{C\ell}f\|_{\Ld^1((0,t);\Ld^2(\R^d))}.
\end{equation*}
Combined} with the result~\eqref{eq:2scale-concl-per} of Theorem~\ref{th:main-per}, this yields the conclusion~\eqref{eq:2scale-concl-perOpt}.\qed
\endgroup

\appendix
\section{Correctors in the random setting}\label{sec:rand-cor}

This section is devoted to the definition and bounds on spectral and hyperbolic correctors in the random setting.
As in the elliptic case~\cite{GNO-quant,AKM-book,Gu-17,DO1}, the main difference with the periodic setting is that only a finite number of correctors can be defined, depending both on space dimension and on mixing properties of the coefficient field.
For simplicity, we focus on the Gaussian setting of Definition~\ref{def:gauss}.
The corrector estimates below were first obtained in~\cite{GO1,Gloria-Otto-10b,GNO-quant} for the first corrector in the elliptic setting.
A proof of the present statement follows from applying iteratively the annealed Calder\'on-Zygmund estimates of \cite{DO1,DFG-22}; see in particular a similar argument in~\cite[Proof of Proposition~2.2]{DO1}. Note that the present result corrects inaccuracies of the corresponding statement given in~\cite[Proposition~C.4]{BG} for spectral correctors.

\begin{theor}\label{th:corr-bounds}
Let $\Aa$ be Gaussian with parameter $\beta>0$ in the sense of Definition~\ref{def:gauss}.
 We then define $\ell_*\,:=\,\lceil\tfrac{\beta\wedge d}2\rceil$ and
\begin{equation}\label{eq:mun-*}
\mu_n^*(x)\,:=\,\left\{\begin{array}{lll}
1&:&n<\ell_*,\\
\log(2+|x|)^\frac12&:&n=\ell_*,~\beta>d,~\text{$d$ even},\\
&&\text{or }n=\ell_*,~\beta<d,~\beta \in 2\N,\\
\log(2+|x|)&:&n=\ell_*,~\beta=d,~\text{$d$ even},\\
\langle x\rangle^\frac12&:&n=\ell_*,~\beta>d,~\text{$d$ odd},\\
\langle x\rangle^\frac12\log(2+|x|)^\frac12&:&n=\ell_*,~\beta = d,~\text{$d$ odd},\\
\langle x\rangle^{n-\frac\beta2}&:&n=\ell_*,~\beta<d,~\beta \notin 2\N.
\end{array}\right.
\end{equation}
\begin{enumerate}[\hspace{-0.15cm}(i)]
\item\emph{Spectral correctors:}
The correctors $\{\psi^n,\sigma^n\}_{0\le n< \ell_*}$ and $\{\zeta^{n,m},\tau^{n,m}\}_{n+2m<\ell_*-3}$ can be uniquely defined by the corrector equations of Section~\ref{sec:def-corr-spec} as centered stationary random fields, and in addition the correctors $\psi^n,\sigma^n$ with $n=\ell_*$ and $\zeta^{n,m},\tau^{n,m}$ with $n+2m=\ell_*-3$ can be uniquely defined as (non-stationary) random fields with centered stationary gradient.
The homogenized coefficient $\bar\bb^n$ is well-defined for $0\le n\le \ell_*$.
Moreover, the following moment bounds hold for all $q<\infty$ and $x\in\R^d$,
\[\begin{aligned}
\qquad\Big\|\Big(\fint_{B(x)}|(\psi^n,\sigma^n)|^2\Big)^\frac12\Big\|_{\Ld^q(\Omega)}&~\lesssim_q~\mu_n^*(x),\quad&&\text{for $0\le n\le \ell_*$},\\
\qquad\Big\|\Big(\fint_{B(x)}|(\nabla\psi^n,\nabla\sigma^n)|^2\Big)^\frac12\Big\|_{\Ld^q(\Omega)}&~\lesssim_q~1,\quad&&\text{for $0\le n\le \ell_*$},\\
\quad\Big\|\Big(\fint_{B(x)}|(\zeta^{n,m},\tau^{n,m})|^2\Big)^\frac12\Big\|_{\Ld^q(\Omega)}&~\lesssim_q~\mu^*_{n+2m+3}(x),\quad&&\text{for $n+2m\le \ell_*-3$},\\
\qquad\Big\|\Big(\fint_{B(x)}|(\nabla\zeta^{n,m},\nabla\tau^{n,m})|^2\Big)^\frac12\Big\|_{\Ld^q(\Omega)}&~\lesssim_q~1,\quad&&\text{for $n+2m\le \ell_*-3$}.
\end{aligned}\]
\item\emph{Hyperbolic correctors:} The correctors $\{\phi^{n,m},\sigma^{n,m}\}_{n+m<\ell_*}$ can be uniquely defined by the corrector equations of Section~\ref{sec:cor-eqn} as centered stationary random fields, and in addition the correctors $\phi^{n,m},\sigma^{n,m}$ with $n+m=\ell_*$ can be uniquely defined as (non-stationary) random fields with centered stationary gradient. The homogenized coefficient $\bar\Aa^{n,m}$ is well-defined for $n+m\le\ell_*$.
Moreover, for $n+m\le\ell_*$, the following moment bounds hold for all $q<\infty$ and $x\in\R^d$,
\begin{align*}
\qquad\Big\|\Big(\fint_{B(x)}|(\phi^{n,m},\sigma^{n,m})|^2\Big)^\frac12\Big\|_{\Ld^q(\Omega)}&~\lesssim_q~\mu_{n+m}^*(x),\\
\qquad\Big\|\Big(\fint_{B(x)}|(\nabla\phi^{n,m},\nabla\sigma^{n,m})|^2\Big)^\frac12\Big\|_{\Ld^q(\Omega)}&~\lesssim_q~1.\qedhere
\end{align*}
\end{enumerate}
\end{theor}

\section{A priori estimates for the wave equation}

We state the following general a priori estimates for linear wave equations, which are used throughout; a short proof is included for convenience. In contrast with the situation in the elliptic setting, we emphasize that putting the impulse in divergence form essentially only brings an improvement when estimating the $\Ld^2$-norm, and not the energy norm. 

\begin{lem}[A priori estimates]\label{lem:apriori}
Let $\Lc$ be a self-adjoint operator on $\Ld^2(\R^d)$ satisfying the bound $-\triangle\le\Lc\le-C_0\triangle$ for some constant $C_0<\infty$. Given
$F_1\in\Ld^1_\loc(\R^+;\Ld^2(\R^d))$ and $F_2,F_3\in W^{1,1}_\loc(\R^+;\Ld^2(\R^d))$ with $F_2|_{t=0}=F_3|_{t=0}=0$, let $z$ be the solution of the wave equation
\begin{equation*}
\left\{\begin{array}{l}
(\partial_t^2+\Lc) z\,=\,F_1+\nabla \cdot F_2+\partial_t F_3\quad\text{in $\R^d$},\\
 z|_{t=0}=\partial_t z|_{t=0}=0.
\end{array}\right.
\end{equation*}
Then, for all $t\ge 0$, we have
\[\|z^t\|_{\Ld^2(\R^d)}\,\lesssim_{C_0}\,t\|F_1\|_{\Ld^1((0,t);\Ld^2(\R^d))}+\|(F_2,F_3)\|_{\Ld^1((0,t);\Ld^2(\R^d))},\]
and
\[\|Dz^t\|_{\Ld^2(\R^d)} \,\lesssim_{C_0}\,\|(F_1,\partial_tF_2,\partial_tF_3)\|_{\Ld^1((0,t);\Ld^2(\R^d))}.\qedhere\]
\end{lem}

\begin{proof}
The assumption $-\triangle\le\Lc\le-C_0\triangle$ entails that $\sqrt\Lc$ defines a bounded linear operator $\Ld^2(\R^d)\to\dot H^{-1}(\R^d)$ with bounded inverse. We may therefore define $\tilde F_2$ as the solution of $\sqrt{\Lc} \tilde F_2=\nabla\cdot F_2$, which satisfies $\|\tilde F_2\|_{\Ld^2(\R^d)}\lesssim_{C_0}\|F_2\|_{\Ld^2(\R^d)}$.
The solution $z$ of the wave equation can then be represented in terms of Duhamel's formula,
\[z^t=\int_0^t\frac{\sin((t-s)\sqrt\Lc)}{\sqrt\Lc}\,(F_1^s+\partial_s F_3^s)\,ds+\int_0^t\sin((t-s)\sqrt\Lc)\,\tilde F_2^s\,ds.\]
Integrating by parts in the integral for $F_3$, with $F_3|_{t=0}=0$, this can be rewritten as
\[z^t=\int_0^t\frac{\sin((t-s)\sqrt\Lc)}{\sqrt\Lc}\,F_1^s\,ds+\int_0^t\sin((t-s)\sqrt\Lc)\,\tilde F_2^s\,ds+\int_0^t\cos((t-s)\sqrt\Lc)\,F_3^s\,ds,\]
and the stated $\Ld^2$-estimate follows from spectral calculus.
For the energy estimate, we rather use energy conservation in form of
\begin{eqnarray*}
\tfrac12\partial_t\int_{\R^d}\big(|\partial_tz|^2+z\Lc z+2F_2\cdot\nabla z\big)
&=&\int_{\R^d} (\partial_tz)\,(\partial^2_t+\Lc)z+\partial_t\int_{\R^d}F_2\cdot\nabla z\\
&=&\int_{\R^d}(\partial_tz)\big(F_1+\partial_tF_3+\nabla\cdot F_2\big)+\partial_t\int_{\R^d}F_2\cdot\nabla z\\
&=&\int_{\R^d}(\partial_tz)\big(F_1+\partial_tF_3\big)+\int_{\R^d}\partial_tF_2\cdot\nabla z\\
&\le&\|Dz\|_{\Ld^2(\R^d)}\|(F_1,\partial_tF_2,\partial_t F_3)\|_{\Ld^2(\R^d)},
\end{eqnarray*}
and the conclusion follows.
\end{proof}

\section*{Acknowledgements}
MD acknowledges financial support from the F.R.S.-FNRS, as well as from the European Union (ERC, PASTIS, Grant Agreement n$^\circ$101075879),
and AG from the European Research Council (ERC) under the European Union's Horizon 2020 research and innovation programme (Grant Agreement n$^\circ$864066).\footnote{{Views and opinions expressed are however those of the authors only and do not necessarily reflect those of the European Union or the European Research Council Executive Agency. Neither the European Union nor the granting authority can be held responsible for them.}}

\def\cprime{$'$} \def\cprime{$'$} \def\cprime{$'$}

\end{document}